\documentclass[12pt,twoside]{amsart}

\usepackage{amssymb}
\usepackage{amsthm}
\usepackage{amsmath}
\usepackage{amsbsy}
\usepackage{amscd}
\usepackage{mathabx}
\usepackage{enumerate}
\usepackage{pdfsync}
\usepackage{graphicx}
\usepackage{anysize}
\marginsize{2.5cm}{2.5cm}{2cm}{3cm}
\usepackage{float}

\theoremstyle{plain}
\newtheorem{thm}{Theorem}[section]
\newtheorem{lem}[thm]{Lemma}
\newtheorem{prop}[thm]{Proposition}
\newtheorem{cor}[thm]{Corollary}

\theoremstyle{definition} 

\newtheorem{ass}{Assumption}

\theoremstyle{remark}

\newcommand{\new}{\newcommand}

\providecommand{\nor}[1]{\lVert{#1}\rVert}
\providecommand{\norop}[2]{\lVert{#1}\rVert}
\providecommand{\abs}[1]{\lvert{#1}\rvert}
\providecommand{\set}[1]{\{#1\}}
\providecommand{\scal}[2]{\langle{#1},{#2}\rangle}

\providecommand{\supp}[1]{\operatorname{supp}(#1)} 

\providecommand{\lspan}[1]{\operatorname{span}\{#1\}}

\providecommand{\tr}[1]{{\,^t\!#1}}

\new{\R}{\mathbb R}
\new{\C}{\mathbb C}
\new{\N}{\mathbb N}
\new{\Z}{\mathbb Z}

\new{\hh}{\mathcal H}
\new{\kk}{\mathcal K}
\new{\cT}{\mathcal T}
\new{\sS}{\mathcal S}

\new{\la}{\lambda}
\new{\eps}{\epsilon}

\DeclareMathOperator{\e}{e}
\DeclareMathOperator{\sinc}{sinc}

\new{\Lloc}{ L^1_\text{loc}(G)}
\new{\Lz}{ L^0(G)}
\new{\T}{\mathcal T} 
\new{\Ss}{\mathcal S} 
\new{\U}{\mathcal U}
\new{\Co}[1]{\operatorname{Co}(#1)}
\new{\I}{(1,+\infty)}
\new{\con}{\convolution}
\new{\vu}{K}
\new{\Luw}{L^1_w(G)}
\new{\Liw}{L^\infty_{w^{-1}}(G)}
\new{\ff}{\mathcal F} 

\new{\E}{E}
\new{\F}{F}

\begin{document}


\title{Coorbit spaces with  voice in a Fr\'echet space}

\author{S.~Dahlke}
\address{S.~Dahlke , FB12 Mathematik und Informatik, Philipps-Universit\"at
  Marburg, Hans-Meerwein Stra{\ss}e, Lahnberge, 35032 Marburg, Germany}
\email{dahlke@mathematik.uni-marburg.de}
\author{F.~De~Mari}
\address{F.~De~Mari, Dipartimento di Matematica, Universit\`a di Genova,  Via
  Dodecaneso 35, Genova,   Italy  }
\email{demari@dima.unige.it}
\author{E.~De Vito}
\address{E.~De Vito ,Dipartimento di Matematica, Universit\`a di Genova,  Via
  Dodecaneso 35, Genova,   Italy  }
\email{devito@dima.unige.it}
\author{D.~Labate}
\address{D.~Labate, Department of Mathematics, University of Houston,
651 PGH building, Houston, Texas U.S.A.}
\email{dlabate@math.uh.edu}
\author{G.~Steidl}
\address{G.~Steidl, Department of Mathematics, University of Kaiserslautern,
  Paul-Ehrlich-Str. 31, 67663 Kaiserslautern, Germany} 
\email{steidl@mathematik.uni-kl.de}
\author{G.~Teschke}
\address{G.~Teschke, Institute for Computational Mathematics in Science and
  Technology, Hochschule Neubrandenburg, University of Applied
  Sciences, Brodaer Str. 2, 17033 Neubrandenburg, Germany}
\email{teschke@hs-nb.de}
\author{S.~Vigogna}
\address{S.~Vigogna, Dipartimento di Matematica, Universit\`a di Genova,  Via
  Dodecaneso 35, Genova,   Italy  }
\email{vigogna@dima.unige.it}

\begin{abstract}
  We set up a new general coorbit space theory for reproducing
  representations of a locally compact second countable group $G$ that
  are not necessarily irreducible nor integrable. Our basic assumption
  is that the kernel associated with the voice transform belongs to a
  Fr\'echet space $\T$ of functions on $G$, which generalizes
  the classical choice $\T=L_w^1(G)$. Our basic example is $ \T
  =\bigcap_{p\in(1,+\infty)} L^p(G)$, or a weighted versions of it. By
  means of this choice it is possible to treat, for instance, Paley-Wiener
  spaces and coorbit spaces related to Shannon wavelets and {\it
    Schr\"odingerlets}. 
\end{abstract}



\maketitle

\section{Introduction}\label{intro}

One of the central problems in applied mathematics is the analysis of
signals. Usually signals are  modelled by functions in  suitable functions
spaces (e.g., $L^2$ or Sobolev spaces) and they might be given explicitly 
 or implicitly as the solution of an operator equation. In most applications,
the signal is transformed via a mapping into a suitable parameter space 
where  it is easier to extract the
information of interest. By discretization, one
 obtains suitable building blocks that give rise to a discrete
representation of the signal and can be used to decompose, compress and
process the signal. Over the years, many different transforms have been
derived in response  to particular  problems, including the wavelet and Gabor transforms.  
Representation theory, however,  gives a general approach to construct continuous transforms for $L^2$-functions,  and  coorbit space theory allows both to extend these transforms to more general function spaces and to provide discrete systems.  Indeed,  it was shown that 
virtually all  well-known transforms used in signal analysis can be
derived from this general setting. In this sense,
coorbit space theory serves  as a common thread in the jungle    of all possible
signal transformations. Nevertheless, as will be explained  below, the classical
coorbit space  setting relies on specific assumptions that might be hard to
verify in practice. The purpose of this paper is  to investigate how to weaken these basic assumptions 
 with the goal of extending the applicability of this framework to a much larger class of problems.

Coorbit space theory was originally introduced by H. Feichtinger and K. Gr\"ochenig in a series of papers in 1988-89 \cite{fegr88, fegr89a,fegr89b, gro91}. By means of this theory,  given a square integrable representation, it is  possible to construct in an efficient and systematic way  a full scale of smoothness spaces where the smoothness of  a function is measured by the decay 
of the so-called {\it voice transform}. For any unitary  representation $\pi$ of a locally compact topological group $G$ on a Hilbert space $\hh$ and a fixed $u\in\hh$, the  voice transform  $V$ is  the map assigning to $v\in\hh$ the corresponding transform $x\mapsto Vv(x)=\langle v,\pi(x)u\rangle$ as $x$ ranges in $G$.
Evidently, $Vv$ is a function on the group. Since $\hh$ is often a Hilbert space of functions, the voice transform connects two function spaces, that is, it maps signals to functions on the group. 

Coorbit space theory has been very successful in many ways and has given rise to a wealth of results, enabling to derive  a very large family of smoothness spaces as coorbit spaces, including both classical functions spaces and new ones. In particular, the classical Besov spaces are derived as coorbit spaces where smoothness is measured by the decay of the wavelet transform, i.e., the voice transform associated with the affine group. Similarly, the   well-established class of modulation spaces corresponds to the family of 
coorbit spaces where smoothness is measured by the decay of the Gabor transform, i.e., the voice transform associated  with the Weyl-Heisenberg group (cf.~\cite{fegr88, fegr89a,fegr89b,LSU12, UH11}). As another example, let us mentions the  $\alpha$--modulation spaces \cite{feigro85} which can be interpreted as coorbit spaces related to group representations modulo quotients \cite{DFH08}. 
Another advantage of coorbit space  theory is to provide  atomic decompositions and Banach frames for the coorbit spaces, through a procedure
which generates discrete function systems by discretization of the group representation.  This is important since it provides a way to understand the properties
of discrete signal representations through the group theoretic properties of their corresponding  continuous  voice transforms.

In recent years, a new generation of multiscale transforms has emerged
in applied harmonic analysis, such as the shearlet and curvelet
transforms, which were introduced to overcome the limitations of the
traditional multiscale framework in multi-dimensional setting with
high efficiency \cite{GL07,CD04}. Recent results have shown that the continuous shearlet transform, in particular, stems from a square  integrable group representation of  the so-called {\it full shearlet group} \cite{DKMSST08, DKS09,KL09}.   By applying the coorbit space theory to this setup, it is possible to define some useful anisotropic smoothness spaces 
via the decay properties of the shearlet transform and to relate these spaces to other well-known function spaces \cite{DST10,DST11, DST12, LMN13}. This has stimulated the investigation of a  larger class of group representations,  primarily those arising from the restriction of the metaplectic representation to a class of  triangular subgroups of the symplectic group \cite{ADDM13}. This class  includes many known cases of interest in signal analysis and gives rise to several new examples, such as the  Schr\"odingerlets that we discuss in this paper. Yet another potential extension of this framework is the general context of the
so-called mock metaplectic representations, introduced in \cite{dede13}.
However, the classical coorbit space theory appears to be too restrictive to deal with  this more general class of group representations
and the corresponding voice transforms and function spaces. 

Let us recall that the classical coorbit space theory {\it \`a la} Feichtinger-Gr\"ochenig
makes the following two assumptions:\\
(FG1) The {\it kernel} $K=Vu$, that is, the voice transform of the admissible vector itself,
is an absolutely integrable function on the group\footnote{For simplicity, in this introduction we use unweighted versions of $L^p(G)$.}.\\
(FG2) The representation is assumed to be irreducible.\\
A major part of this paper is concerned with replacing (FG1) by some weaker condition.
The fundamental concepts will be presented in Sections 2 and 3.
First of all, let us mention that  the problem of removing the integrability condition has already been addressed by
J. Christensen and G. \'Olafsson in \cite{C2012, chol11}. 
Classically, the reservoir of test functions is obtained by taking the functions whose voice transform is in $L^1(G)$,
hence it is a Banach space in a natural way.
In the papers \cite{C2012, chol11}, the reservoir is a fixed Fr\'echet space $\Ss$ densely embedded into $\hh$.
The basic example is the set of $C^\infty$ vectors for the representation.
The approach that we consider in this paper is indeed similar to \cite{chol11},
but features an important new datum that we call  the {\it target space} $\cT$.
This is a Fr\'echet space of functions on $G$ and plays the role of $L^1(G)$ in the classical setup.
In our theory, the reservoir $\Ss$ is the set of functions whose voice transform is in $\cT$.
In Section 4, we provide a new model for the target space, namely
\begin{equation}
\cT=\bigcap_{p>1}L^p(G).
\end{equation}
For this choice we are able to produce concrete cases arising from triangular subgroups of $Sp(2,\R)$,
notably the case of the so-called {\it Schr\"odingerlets}, discussed in Section 4.3.
As a toy example, in Section 4.2  we also consider the case of  the band-limited functions\footnote{Notice that the sinc function is in every $L^p$ with $p>1$ but not in $L^1$.}.

As for assumption (FG2),    most of the classical coorbit space theory can  be carried out also in the
reducible case.  In the irreducible case, it is possible to
show  that the construction of  the coorbit spaces is independent of the
choice of the admissible vector.  Exactly this property is lost in the
reducible case, as showed by an important example in \cite{F2013}. Neither in \cite{chol11} nor in our setting irreducibility is needed. In Section~\ref{AV} we give a reasonable description of admissible vectors leading to the same coorbit spaces. Finally, in Section~\ref{L1} we present a detailed account of the extent to which the classical $L^1$ theory can be developed without the assumption of irreducibility.

Let us briefly describe in some detail the main features of our approach.
 The starting ingredients  are  a unitary reproducing representation $\pi$ of a locally compact second countable group $G$ on a separable Hilbert space $\hh$ and an admissible vector $u\in\hh$.  Next,  the Fr\'echet spaces $\Ss$ and $\T$ come into play, and their roles  in the theory can be described by 
the following  very basic conceptual picture
 \[
 \begin{CD} 
                 \Ss@ >i>>\hh\\
@.@VV{V}V\\
\T@>j>>L^0(G)
\end{CD}
\]
where $L^0(G)$ denotes the space of measurable functions on $G$. Thus,
$\sS$ and $\cT$ are Fr\'echet spaces that embed into $\hh$ and
$L^0(G)$, respectively, and should therefore be thought of as signals
and functions on the group, respectively. The space $\T$ is a free
choice, as long as one can embed it continuously into $L^0(G)$ in such
a way that some basic properties are satisfied (Assumptions \ref{H1}
and \ref{H2}). The space $\sS$ will serve as the reservoir of {\it
  test functions} and it is defined as the subset of $\hh$ consisting
of those vectors whose voice transform belongs to $j(\cT)$. Test
functions are modeled in terms of their voice transforms and the
latter ones constitute the true degree of freedom in the construction,
the target space~$\cT$.

Once the basic structures are laid out, one then follows the lines of coorbit space theory and defines first the distributions $\sS'$ and then coorbits associated to Banach spaces of functions.

Technically speaking, our theory is determined by the  data set  $(G,\hh,\pi,u,\T,Y)$, where:
\begin{itemize}
\item $G$ is a locally compact second countable topological group;
\item  $\hh$ is a separable Hilbert space;
\item $\pi$ is a continuous unitary reproducing representation  of $G$ on $\hh$;
\item $u\in\hh$ is an admissible vector;
\item  $\T$ is a Fr\'echet space continuously embedded via $j$ into $L^0(G)$;
\item  $Y$ is a Banach space continuously embedded in $\Lz$ and left invariant.
\end{itemize}
The main ancillary objects attached to them are:
\begin{itemize} 
\item the voice transform $V_2:\hh\to L^2(G)$;
\item the reproducing kernel $K=V_2u$;
\item the reproducing kernel space  $\mathcal M=\set{f\in L^0(G)\mid f\con \vu=f}$;
\item the space of test functions $\Ss= \set{v\in \hh \mid V_2v\in j(\T)}$;
\item the space of distributions $\Ss'$;
\item the extended voice transform $V_e:\Ss'\to C(G)$;
\item the coorbit space $ \Co{Y} = \set{T\in \Ss' \mid V_e T\in Y}$.
\end{itemize}
These data are assumed to satisfy several assumptions that are made both for the target space
$\T$ and for the model space $Y$. 
Assumption~\ref{H1}  asks
that  the kernel  $K$ is in $\T$,  which substitutes the classical integrability
condition  $K\in L^1(G)$,  and that  $Kf\in L^1(G)$ whenever
$f\in\T$. This
second requirement has a twin version for $Y$, namely
Assumption~\ref{H5}, and it is trivially satisfied  in the case
$\T=L^1(G)$  because $K$ is bounded.  
Assumption~\ref{H2} and Assumption~\ref{H6} ask that the product of
any voice transform and any ``good'' function in $\T\cap\mathcal M$
(or in $Y\cap\mathcal M$) is in $L^1(G)$.

Assumption~\ref{H3} ensures that the extended voice transform is
injective. This is a necessary condition  to reconstruct 
a distribution  from  its voice transform. Finally,
Assumption~\ref{H4} requires that the reproducing formula extends to all 
distributions. In particular, we prove in Proposition~\ref{repr-formula} that
Assumption~\ref{H4} holds
true if   $\T$ is reflexive and $V_e\Ss'\subset\T'$. 

Our theory is succesful in the sense that:
 it provides a workable substitute for the classical integrability condition $K\in L^1(G)$;
 it contains the classical coorbit space theory even for non irreducible representations;
 it applies to several interesting examples;
 it is consistent with the recent theory developed in \cite{C2012, chol09}.


\section{Fr\'echet spaces of functions}\label{sec:target}
In this section, we recall some  properties of the Fr\'echet
spaces that are relevant to the main objects of our theory, namely the {\it
  target space} $\cT$ and the space $\Ss$ of {\it test signals}
that will be defined in the next section.   We introduce  abstract spaces $E$ and  $F$.
The space $E$ must be
interpreted as modeling a subspace of the Hilbert space $\hh$, hence
of signals, but possibly with a different topology. Its properties will be 
used primarily for the test space $\Ss$, but also for~$\hh$ itself.
Similarly, the space $F$ should be thought of as an abstract 
model of a Fr\'echet space of functions on the group. The results
proved for $F$ will be primarily applied to  the target space $\cT$,
which in many examples is a genuine Fr\'echet space but not a
Banach space, but will also be useful for $F=L^1(G)$, $F=L^2(G)$ 
and, most notably, for $F=Y$, the  space of functions used
to define  coorbit spaces. From this point of view, our theory
indicates that it  is possible to develop a useful coorbit space
theory assuming that  $Y$ is a Fr\'echet space rather than the more 
common choice of a Banach space. This further extension, however,
is beyond the scope of this article, and we content ourselves with
the classical case in which $Y$ is Banach space.


\subsection{Background}\label{background}
We now introduce the basic notation and  recall some elementary
properties. Further technical results are recalled in Section~\ref{Notation}. 

Throughout this paper, $G$,  denotes a fixed locally compact second countable
group with a left Haar measure $\beta$ and $\Delta$ is  its
modular function.  We write $\int_G f(x)dx$ instead of $\int_G f(x)\,d\beta(x)$ and 
denote the classical spaces of complex functions on  $G$ as follows:

\vspace{.5cm}

  \begin{tabular}{ll}
    $\Lz$ & $\beta$-measurable functions,\\
    $L^p(G)$ & $p$-integrable functions with respect to $\beta$,
    $p\in [1,+\infty)$,\\
    $L^\infty(G)$ & $\beta$-essentially bounded functions, \\
    $\Lloc$ & locally $\beta$-integrable functions,\\
    $C(G)$ &  continuous functions, \\
    $ C_0(G)$ & continuous functions going to zero at
    infinity, \\
    $C_c(G)$ & compactly supported continuous functions.
  \end{tabular}

\vspace{.3cm}

\noindent The space $\Lz$ is a metrizable complete topological vector space with
respect to the topology of convergence in measure
(see Section~\ref{Notation}). The norm of $f\in  L^p(G)$ and the scalar
product between $f,g\in L^2(G)$ are denoted by
$\nor{f}_p$ and $\scal{f}{g}_2$, respectively. The space $\Lloc$ is a Fr\'echet space 
with respect to the topology defined by the family of semi-norms
\[ f\mapsto \int_{\mathcal K} \abs{f(x)}dx,\]
where $\mathcal K$ runs over the compact subsets of $G$ (see
Section~\ref{Notation}).

We denote by $\la$ and $\rho$ the left and right regular
representations of $G$ on $L^0(G)$, namely  
\begin{align*}
  \la(x)f\,(y) & =f(x^{-1}y) \\
  \rho(x)f\,(y) & = f(yx) 
\end{align*}
for all $x\in G$,  all $f\in L^0(G)$ and  almost all $y\in G$. 
Both $\la$ and $\rho$ leave $\Lloc$ and each $L^p(G)$ invariant,
 and $\la$ is equicontinuous both  on $\Lloc$ and on each $L^p(G)$. In
 Section~\ref{Reps} we recall the main properties of the representations acting on
 Fr\'echet spaces. For general background on representations the
 reader is referred to \cite{fol95}. 

For all $f\in L^0(G)$ we denote by $\check{f}$ the element in $L^0(G)$ given by
\[ \check{f}(x)=f(x^{-1})\]
for almost all $ x \in G $ (see Section~\ref{Notation}).
Given two functions $f,g\in \Lz$, we say that the convolution $f\con g$
exists if for almost all $x\in G$ the function $f\la(x)\widecheck{g}$ is
in $L^1(G)$. We write
\[
f\con g(x)=\int_G f(y) \la(x)\widecheck{g}(y)\,dy= \int_G f(y)
g(y^{-1}x)\,dy\qquad\text{a.e. }x\in G
\]
and we have that $f\con g\in\Lz$
(see Section~\ref{convolution}).

\subsection{Voice transform }

In what follows, $\pi$  denotes a fixed strongly continuous unitary
representation of $G$ acting on the  separable Hilbert space $\hh$ and
$u$ a fixed vector in $\hh$. We stress that $\pi$ is not assumed to be
irreducible, nor, at this stage, reproducing. 
 As it is customary, the {\it voice transform} associated to the these data is the map
\[
V:\hh\to L^\infty(G)  \cap C(G), \qquad
Vv(x)=\scal{v}{\pi(x)u}_\hh .
\]
It  intertwines $\pi$ and $\lambda$, that is  
\begin{equation}
V\pi(x) =\la(x)V
\label{eq:4}
\end{equation}
for all $x\in G$. The corresponding  kernel is given by
\begin{equation}
K:G\to\C,\qquad K(x)=Vu(x)=\scal{u}{\pi(x)u}_\hh.
\label{eq:31}
\end{equation}
It enjoys the basic symmetry property
\begin{equation}
  \label{eq:32}
  \overline{K}=\widecheck{K}. 
\end{equation}
 For all $f\in L^1(G)$ the {\it Fourier transform} $\pi(f)$ is
the bounded operator on $\hh$ defined by
\[
\scal{\pi(f)w}{v}_\hh= \int_G  f(x) \scal{\pi(x)w}{v}_\hh\,dx
\]
for all $w,v\in\hh$. Note that, with the choice $w=u$, we get 
\begin{equation}
\scal{\pi(f)u}{v}_\hh =\int_G  f(x) \overline{Vv(x)}\,dx .\label{eq:42}
\end{equation}

\subsection{Functions on the group}\label{subsec:target}

In this section,  we consider a space $\F$ of functions on $G$ and 
we study the properties of the convolution operator 
$f\mapsto f\con K$ from $\F$ into $C(G)$. In particular,  we introduce  the
subspace $\mathcal M^\F$ of those
functions which are left fixed by the convolution operator.  
In the theory of  reproducing representations, $\F$ is the
Hilbert space $L^2(G)$; in the theory  developed by
H. Feichtinger and K. Gr\"ochenig  it is a weighted version of $L^1(G)$; in
our setting  it  is the target space~$\T$.

We assume that $\F$ is a  Fr\'echet space
with a  continuous embedding  $j:\F\to \Lz$. With slight abuse of
notation, given $f\in \F$, we denote by 
$f(\cdot)$ a $\beta$-measurable function  such that for almost every
$x\in G$, $ f(x)=j(f)(x)$. Further, we assume that there exist
\begin{enumerate}[i)]
\item a continuous involution $f\mapsto \overline{f}$ on
 $\F$ such that $j(\overline{f})=\overline{j(f)}$ (so that
 $\overline{f}(x)=\overline{f(x)}$);
\item 
a continuous representation $\ell$ of $G$  acting on $\F$ for which 
\[ j(\ell(x)f)=\la(x)j(f) \qquad f\in \F,\, x\in G,\]
so that $(\ell(x)f)(y)=f(x^{-1}y)$ and $\overline{\ell(x)f}=\ell(x)\overline{f}$.
\end{enumerate}
Standard examples of spaces satisfying the above assumptions 
are the $L^p$-spaces or their weighted  versions. Other important examples 
are the space of $C^\infty$ functions, if   $G$ is a Lie group,  or the space of 
rapidly decreasing functions, whenever this notion makes sense.

The space $j(\F)$ is a subspace of $\Lz$,  stable under
complex conjugation and   $\la$-invariant.  Clearly, we could identify $\F$ with 
$j(\F)$ avoiding the cumbersome map $j$. However, we want to stress that 
$\F$ has its own  topology, which  is not necessarily the topology of $j(\F)$, that is, the relative
topology as a topological subspace of  $\Lz$. In order to clarify the role of the two topologies, 
we shall not identify $\F$ with $j(\F)$.   

Since $j$ is continuous from $\F$ into
$\Lz$, for any  sequence
$(f_n)$ converging to  an element $f$ in $\F$, there exists a subsequence
$(f_{n_k})_k$  such that $(f_{n_k}(x))_k$ converges to $f(x)$ for
almost all $x\in G$ (see~\eqref{eq:25} for details). 

We denote by $\F'$ the topological dual of $\F$.
For each $T\in \F'$, the map $f\mapsto T(\overline{f})$ defines a
continuous anti-linear  function on $\F$, which we denote by
$\scal{T}{\cdot}_\F$. The map $T \mapsto \scal{T}{\cdot}_\F$ is a linear
isomorphism of $\F'$ onto the anti-dual $\F^\wedge$, the space of  anti-linear
continuous forms on $\F$. In what follows, we identify $\F^\wedge$ with $\F'$.
Observe that   the map~$(T,f)\mapsto \scal{T}{f}_\F$ is a
sequilinear form on $\F'\times \F$, linear in the first entry and
anti-linear in the second.

The K\"othe dual of $\F$ is defined by
\[ \F^\#=\set{g \in \Lz \mid g j(f)\in L^1(G), \text{ for all }f\in \F}.\]
It is closed under complex conjugation  and  its elements can be
regarded as  anti-linear forms on $\F$, as shown by the next lemma. Here
and below, we fix a  countable fundamental system
$\{q_i\}_i$ of saturated\footnote{A family is saturated if the
  maximum of any finite subset is in the family.} 
 semi-norms in $\F$. 
\begin{lem}\label{ltre}
 Given $g\in  \F^\#$, the map 
\begin{equation}
  \label{eq:9}
\F\ni f\mapsto  \scal{g}{f}_\F= \int_G g(x)\overline{f(x)}\,dx \in\C
\end{equation}
is a continuous anti-linear form, that is,  an element of $\F'$,
which we denote again by $g$.
The representation $\la$ leaves $\F^\#$  invariant 
and for all $x\in G$
\begin{equation}
  \label{eq:14}
\scal{\la(x)g}{f}_\F=\scal{g}{\ell(x^{-1})f}_\F.
\end{equation}
Finally, there exist a constant $C>0$ and a semi-norm $q_k$ in the
fundamental saturated family $\set{q_i}_i$ such that for all
$f\in \F$  
\begin{equation}
  \label{eq:50}
\int_G   \abs{f(x)}\abs{g(x)}\,dx \leq C q_k(f).
\end{equation}
\end{lem}
\begin{proof}
By definition, for each $f\in \F$ the function $g\,j(f)$ is in
$L^1(G)$. We claim that the linear map 
\[{\mathcal L}: \F\to L^1(G), \qquad
{\mathcal L}f=g\,j(f) \]
is continuous. Since both $\F$ and $L^1(G)$ are separable metrizable
vector spaces, by the closed graph theorem (Corollary 5 of Chapter I.3.3 of \cite{BTVS})
it is enough to show that   the graph of ${\mathcal L}$  is  sequentially
closed in $\F\times L^1(G)$. Take a sequence $(f_n)_n$ in $\F$ converging
to $f$ in $\F$ and such that $({\mathcal L} f_n)_n$ converges to $\varphi$ in
$L^1(G)$. Since both
$\F$ and $L^1(G)$ are continuously embedded in $\Lz$, possibly
passing to a subsequence,  we can assume that both $(f_n(x))_n$ and
$({\mathcal L} f_n(x))_n$  converge  to $f(x)$ and $\varphi(x)$, respectively, for almost every $x$. 
 Hence for almost all $x\in G$
\[{\mathcal L}f\,(x)=g(x)f(x)=\lim_{n\to+\infty} g(x)f_n(x)=\lim_{n\to+\infty}{\mathcal L}f_n(x)=\varphi(x),\]
that is, ${\mathcal L} f=\varphi$ in $L^1(G)$. Hence ${\mathcal L}$ is continuous, as well as
the anti-linear form 
\[f\mapsto \int_G  {\mathcal L} \overline{f}(x)dx=
\int_G g(x)\overline{f(x)}\,dx=\scal{g}{f}_\F.\] 
We now prove that $\la$ leaves $\F^\#$  invariant. Indeed, given $x\in
G$ and $f\in \F$
\[ \int_G \abs{\la(x)g\,(y) f(y)}dy= \int_G \abs{g(x^{-1}y)
f(y)}dy= \int_G \abs{g(y)
f(xy)}dy =\int_G \abs{g(y)
\ell(x^{-1})f\,(y)}dy <+\infty,\]
where the last integral is finite since  $\ell(x^{-1})f\in \F$. Hence 
$\la(x)g\in \F^\#$. The same string of equalities gives~\eqref{eq:14}.
Finally, the last formula follows directly from the continuity of ${\mathcal L}$.
\end{proof}
In general, $\F^\#$ is a proper subset of $\F'$ as the following example clarifies.
Take $\F=L^p(G)$ with $p\in [1,+\infty]$. 
Then $ L^p(G)^\# = L^{p'}(G) $, where $p'$ is the dual exponent of $p$,
so that $ L^p(G)^\# = L^p(G)' $ for all $ p < +\infty $,
but of course $ L^{\infty}(G)^\# = L^1(G) \subsetneq L^\infty(G)' $.

The next proposition shows that, if the kernel $K$ belongs to
  the K\"othe dual of $\F$, then for all $f\in\F$ the convolution
  $j(f)\con K$ exists. Furthermore, we introduce the subspace $\mathcal
M^\F\subset \F$ whose elements  are those reproduced by convolution with
$K$ on the  right.  In the following statement $C(G)$ is endowed
with the topology of the compact\footnote{Short for: uniform
  convergence on compact sets.}  convergence.

\begin{prop}\label{uno}
Assume that ${K}\in \F^\#$. Then: 
\begin{enumerate}[a)]
\item for all $f\in \F$, $j(f)\con K$ exists everywhere, it is
  a continuous function and, for
  all $x\in G$,
  \begin{equation}
    \label{eq:15}
   j(f)\con K(x)=\scal{\la(x) \check{K}}{\overline{f}}_\F
   = \scal{ \check{K}}{\overline{\ell(x^{-1}) f}}_\F;
  \end{equation}
\item the map $f\mapsto j(f)\con  K$
 is continuous from $\F$ to $C(G)$;
\item the set
  \begin{equation}
    \label{eq:11}
    \mathcal M^\F=\set{f\in \F\mid j(f)\con K= j(f)}
\end{equation}
is an $\ell$-invariant closed subspace of $\F$, and therefore it is a Fr\'echet space respect to the relative topology.
\end{enumerate}
\end{prop}
\begin{proof} Notice that, in view of \eqref{eq:32}, $ K \in \F^\# $ if and only if $ \check{K} \in \F^\# $.
Since $\la$ leaves $\F^\#$ invariant (Lemma \ref{ltre}), for all $x\in G$ we have $\la(x)\check{K} \in \F^\#$ and, for all $f\in \F$,
\[ \scal{ \la(x) \check{K}}{\overline{f}}_\F =\int_G f(y)\la(x)
\check{K}\,(y)\,dy= \int_G f(y)
K(y^{-1}x)\,dy =j(f)\con K\,(x).\]
Hence $j(f)\con K(x)$ exists and the first equality of~\eqref{eq:15} holds
true.  The change of variables $y\mapsto xy$ proves the second equality of~\eqref{eq:15}.
Since the involution  and $x\mapsto\ell(x^{-1})f$ are continuous,~\eqref{eq:15} implies that
$j(f)\con K$ is a continuous function.

To prove b), fix a compact subset ${\mathcal K}\subset G$. By \eqref{eq:15}, since 
$\check{K}\in\F^\#\subset\F'$, there exist two semi-norms $q_j,q_k$  
in the  fundamental saturated system $\set{q_i}_i$ and  constants $C$ and $C'$ such that
\[
  \sup_{x\in {\mathcal K}} \abs{j(f)\con K\,(x)} 
   \leq C \sup_{x\in {\mathcal K}^{-1}}
  q_j(\ell(x)f) 
 \leq CÕ q_k(f),
\]
where the last inequality follows from the fact that $\ell({\mathcal K}^{-1})$ is equicontinuous
since ${\mathcal K}^{-1}$ is compact (see Section~\ref{Reps}).

As for~ c), since $\F$ is a metrizable vector space, it is sufficient to prove that
$\mathcal M^\F$ is sequentially closed.  Take  a sequence $(f_n)_n$ in
$\mathcal M^\F$ converging to $f\in \F$. Possibly passing to a 
subsequence, we can assume that there exists a negligible set $N$ such
that  for all $x\notin N$ $(f_n(x))_n$ converges to $f(x)$.
Furthermore, possibly changing $N$, we can also assume that, for all
$n\in\N$ and $x\notin N$,  $j(f_n)\con K\,(x)=f_n(x)$.
Hence, given $x\notin N$,  by  b)  we have
\[j(f) \con K\,(x)= \lim_n j(f_n)\con K\,(x)=\lim_n f_n(x)=f(x).\]
Hence $j(f)\con K=j(f)$ in $\Lz$, that is $f\in\mathcal M^\F$.
Finally, given $x\in G$ and $f\in\mathcal M^\F$, by~\eqref{eq:27} in the appendix
\[ j(\ell(x)f)=\la(x) j(f)=\la(x)(j(f)\con K)= \la(x)j(f)\con K= j(\ell(x)f)\con K ,\]
that is $\ell(x)f\in\mathcal M^\F$.
\end{proof}
In what follows, for each $f\in \mathcal M^\F$, we choose the
continuous everywhere defined function $j(f)\con K$ as
representative of $f$, so that for all $x\in G$
\begin{equation}
  \label{eq:16}
   j(f)\con K\,(x)=f(x).
\end{equation}

\subsection{Extension of the voice transform and Fourier transform}\label{sec:ext}

We are interested in extending the voice transform $V$ from $\hh$  to some bigger space,
namely the dual of a Fr\'echet space  $\E$ which is continuously embedded into
  $\hh$,  in such a way that the duality relation~\eqref{eq:42} still
  holds true.  In the classical coorbit space theory and in our setting, $\E$
is the space of test functions $\Ss$, which in  \cite{chol11} is
  the basic object on which the theory is developed.

We fix a Fr\'echet space $\E$ together with a continuous representation
$\tau$ of $G$ acting on $\E$ and  a continuous embedding $i:\E\to
\hh$ intertwining $\tau$ and $\pi$. As above, 
we identify the dual $\E'$ and the anti-dual $\E^\wedge$.
We are interested in 
the transpose $\tr{i}:\hh\to \E'$, defined as usual by
\[ \scal{\tr{i}(w)}{v}_\E= \scal{w}{i(v)}_\hh.\]
We assume that  $u\in i(\E)$ and, with slight
  abuse of notation, we regard $u$ as an element both in $\E$ and in $\hh$.
 Hence, we define the {\it extended voice transform} by
\begin{equation}\label{EVT}
V_e:\E'\to C(G),\qquad V_eT=\scal{T}{\tau(\cdot)u}_\E,
\end{equation}
which intertwines the contragredient representation $\tr{\tau}$ with
the left regular representation $\la$.

Hereafter we establish some useful properties of the abstract space $\E$, that will
 be applied to the space $\Ss$ of test functions in Section~\ref{testanddist}.

A basic requirement on $V_e$ is that it must be
injective.  In the next lemma some standard equivalent conditions are established.
\begin{lem}\label{b} The following facts are equivalent:
  \begin{enumerate}[i)]
  \item\label{b1} the map $V_e$ is injective;
  \item\label{b2} the set $\tau(G)u$ is total in $\E$;
  \item\label{b3} the set $\tau(G)u$ is total in\footnote{Evidently, $\E_{weak}$ is just $\E$ endowed with the weak topology.} $\E_{weak}$.
  \end{enumerate}  
\end{lem}
\begin{proof}
Define $D$ as the closure in $\E$ of the linear span of $\tau(G)u$ .
Since the linear span of $\tau(G)u$ is convex, $D$ is convex (Proposition~14 Chapter~II.2.6 of \cite{BTVS}).
Then $D$ is a closed convex subset, and so it is also weakly closed (Proposition~1 Chapter~IV.1.1 of \cite{BTVS}).
Therefore $D$ is also the closure in $\E_{weak}$ of the linear span of $\tau(G)u$.
 This proves the equivalence between \ref{b2}) and \ref{b3}).

Next, assume \ref{b2}), that is $ D = \E $.
If $ T \in \E' $ is such that $ V_eT = 0 $, then $ \scal{T}{v}_\E = 0 $ for all $ v \in D = \E $, hence $ T = 0 $.
Therefore \ref{b2}) implies \ref{b1}).
Finally, suppose $ D \subsetneq \E $.
Then, as a consequence of the Hahn--Banach theorem (\cite{hor66}, Proposition~2, p.180),
there exist $ T \in \E' $ and $ v_0 \in E \setminus D $ such that
$ \scal{T}{v}_\E = 0 $ for all $ v \in D $ and $ \scal{T}{v_0}_\E \neq 0 $,
so that $ V_eT = 0 $ but $ T \neq 0 $.
Therefore \ref{b1}) implies \ref{b2}).
\end{proof}

By definition, $u$ is a cyclic vector for the representation $\tau$ if condition \ref{b2}) holds true.
Notice that the cyclicity for $\pi$ does not imply the cyclicity for $\tau$,
since the topology of $\E$ is finer than the topology of $\hh$.

The topological dual of $E$ comes now into play and will be denoted by $\E'$.
When topological properties are involved, $\E'$ is understood to have the topology of the convergence on the bounded subsets of $E$.
We will write $\E'_s$ to stress when $\E'$ is rather thought with the topology of the simple convergence (compare Section~\ref{scalar_integration}).

The following proposition shows that, for any function $f\in L^0(G)$ satisfying a suitable integrability condition,
it is possible to define an element in $\E'$ which plays the role of the Fourier transform of $f$ at $u$.
It is useful to compare our  assumption~\eqref{eq:17} with condition~(R3) in \cite{chol11}.

\begin{prop}\label{due}  Take $f\in L^0(G)$ and assume that
  \begin{equation}
    \label{eq:17}
   f\,  Vi(v) \in L^1(G)\qquad \text{for all }v\in \E.
  \end{equation}
Then there exists a unique $\pi(f)u\in \E'$ such that, for all $v\in E$,
 \begin{equation}
   \label{eq:12}
   \scal{\pi(f)u}{v}_\E= \int_G f(x) \scal{\pi(x)u}{i(v)}_\hh\,dx = \int_G f(x) \overline{Vi(v)(x)}\,dx.
 \end{equation}
For any such $f$, we have
\begin{equation}
  \label{eq:40}
  V_e\pi(f)u= f\con \vu.
\end{equation}
Finally, assume that $\vu\con \vu$ exists, is equal to $\vu$ and that
$f\con (\abs{\vu}\con\abs{\vu})$ exists. Then 
\begin{equation}
  \label{d}
  V_e\pi(f)u \con K = V_e\pi(f)u.
\end{equation}
\end{prop}
\begin{proof}
Define the map $\Psi: G\to \E_s'$ by $\Phi(x)=f(x)
\tr{i}(\pi(x)u) $. Since for all $v\in \E$
\[\scal{\Psi(x)}{v}_\E= f(x)  \scal{\pi(x)u}{i(v)}_\hh=f(x)\overline{V i(v)}(x),\]
by \eqref{eq:17} we know that  the map $\Psi$ is
scalarly $\beta$-integrable. Since $\E$ is a Fr\'echet space, then it satisfies the (GDF)
property. Then Theorem~\ref{Gelfand-Dunford} applies,
showing that the scalar integral $\int \Psi(x)dx$ exists and belongs to  $\E'$ (see Section~\ref{scalar_integration}).
We set $\pi(f)u=\int\Psi(x)dx$ and, by definition of scalar integral, \eqref{eq:12} holds true for all $v\in \E$.
Also, for all $x\in G$,
\[ V_e\pi(f)u\, (x)= \int_G f(y) \scal{\pi(y)u}{i(\tau(x)u)}_\hh\,dy
=\int_G f(y) \scal{\pi(y)u}{\pi(x)u}_\hh\,dy=
(f\con \vu)(x).\]
Finally, under the ongoing assumptions, \eqref{eq:35} in the appendix implies that
$(f\con \vu)\con \vu= f\con (\vu \con \vu)= f\con \vu$,
so that~\eqref{d} is a direct consequence of \eqref{eq:40}.
\end{proof}

 If~\eqref{eq:17} is satisfied,
  we say that the Fourier transform of $f$ at $u$ exists in
  $\E'$ or, simply,  that   $\pi(f)u\in \E'$ exists.  Condition~~\eqref{eq:17} is actually
  both  necessary and  sufficient  to define $\pi(f)u$ as an element of $\E'$. 
The next lemma ensures
that the voice transform is reproduced by convolution. 
\begin{lem}\label{rsei} Assume that the extended voice transform is
  injective and take $T\in \E'$. The following assertions are equivalent:
\begin{enumerate}[a)]
  \item\label{c3} $V_eT\con \vu$ exists and satisfies the reproducing formula
\begin{equation}
V_eT\con \vu = V_eT;\label{eq:58}
\end{equation}
 \item\label{e2} for all $x\in G$, the map $y\mapsto
    \scal{T}{\tau(y)u}_\E \scal{\pi(y)u}{\pi(x)u}_\hh $ is in $L^1(G)$ and 
\[ \int_G \scal{T}{\tau(y)u}_\E \scal{\pi(y)u}{\pi(x)u}_\hh dy = \scal{T}{\tau(x)u}_\E.\]
\end{enumerate}
If the Fourier transform of $V_eT$ at $u$  exists in $\E'$, {\em i.e.} the map $x\mapsto
  V_eT(x) \tr{i}(\pi(x)u) $ is scalarly   integrable,  then a) and b) are also
  equivalent to each of the following assertions:
\begin{enumerate}[a)]\setcounter{enumi}{2}
   \item\label{c1} $\pi(V_eT)u=T$;
  \item\label{c2}  the reconstruction formula 
  \begin{equation}
T = \int_G \scal{T}{\tau(x)u}_\E \tr{i}(\pi(x)u) dx\label{eq:41}.
\end{equation}
holds true weakly.
  \end{enumerate}
\end{lem}

\begin{proof} The equivalence between \ref{c3}) and \ref{e2}) is just
  the definition of $V_e$ and $K$.  Taking into account that $V_e$ is
  injective, the equivalence between \ref{c3}) and \ref{c1}) follows
  from~\eqref{eq:40} with  $f=V_eT$. The equivalence between \ref{c1})
  and \ref{c2}) is just the definition of scalar integral. 
\end{proof}
In the next  proposition we assume that  the Fourier
transform $\pi(f)u$ exists in $\E'$  for all $f\in \F$, where $\F$ is a
Fr\'echet space satisfying the assumptions of Section~\ref{subsec:target}.
With slight abuse of notation, we write $\pi(f)u$ instead of $\pi(j(f))u$. 
We define the coorbit space
\[
\Co{\E',\F}=\set{T\in \E'\mid V_eT\in j(F)}.
\]
\begin{prop}\label{cuno}
Take $\E\overset{i}{\hookrightarrow} \hh$,  $\F\overset{
j}{\hookrightarrow} \Lz$ and $K(x)=\scal{u}{\pi(x)u}_\hh$  as
above. Assume that  for all $f\in \F$ and all $v\in \E$
 \begin{equation}
    \label{eq:36}
   j(f)  Vi(v)\in L^1(G),
  \end{equation}
which may be rephrased as $V(i(\E))\subset \F^\#$. Then:
\begin{enumerate}[a)]
\item the Fourier transform of any $f\in F$ at $u$ exists in $\E'$,
so does the convolution $ j(f)\con \vu$, and
  \begin{equation}
    \label{eq:37}
    V_e\pi(f)u= j(f)\con \vu;
  \end{equation}
\item the space
 $\mathcal M^\F$, defined by \eqref{eq:11},  is an $\ell$-invariant closed subspace of $\F$,
 $\Co{\E',\F}$ is a $\tr{\tau}$-invariant subspace of $\E'$, and $V_e$
 intertwines $\tr{\tau}$ with $\la$;
\item the map $(f,v)\mapsto \scal{\pi(f)u}{v}_\E$ is continuous from
  $\F\times \E$ into $\C$;
\item if $V_e$ is injective and the reproducing formula~\eqref{eq:58} holds for all
$T\in \Co{\E',\F}$, then
  \begin{subequations}
    \begin{align}
      & V_e\Co{\E',\F}=j(\mathcal M^\F),  \label{eq:38a}\\
      & \set{\pi(f)u\mid f\in \mathcal M^\F}=
      \Co{\E',\F},  \label{eq:38b}\\
      &  V_e\pi(f)u = j(f),  \qquad f\in\mathcal M^\F ,\label{eq:38c}\\
      & \pi( V_eT)u = T, \qquad\quad T\in\Co{\E',\F}.\label{eq:38d}
    \end{align}
  \end{subequations}
Hence, $V_e$ is a bijection of $\Co{\E',\F}$ onto $j(\mathcal M^\F)$ and therefore it induces a bijection,  denoted again by $V_e$, from $\Co{\E',\F}$ onto $\mathcal M^\F$,  whose inverse is the Fourier transform at 
$u$.
\end{enumerate}
\end{prop}
\begin{proof} Item~a) is a direct consequence of  Proposition~\ref{due}.
Item~b)  is due to Proposition~\ref{uno}. The
invariance property of $\Co{\E',\F}$ is a consequence of the fact
 that $V_e\tr{\tau}(x)=\la(x)V_e$ for all $x\in G$.
 
 As for c), since $\F$ and $\E$ are Fr\'echet spaces it is enough to
 show that $(f,v)\mapsto \scal{\pi(f)u}{v}_\E$ is separately
 continuous. Clearly, given $f\in \F$, the map $v\mapsto
 \scal{\pi(f)u}{v}_\E$ is continuous since $\pi(f)u\in \E'$. On the
 other hand, given $v$ in $\E$, the hypothesis~\eqref{eq:36} states
 that $Vi(v)\in \F^\#$. Lemma~\ref{ltre} shows that $Vi(v)\in \F'$
 where the identification is given by~\eqref{eq:9}, namely
\[ \scal{Vi(v)}{f}_\F = \int_G Vi(v)(x)
\overline{ f(x)}\,dx=\overline{\int_G f(x) \scal{\pi(x)u}{i(v)}_\hh\,dx} =\overline{\scal{\pi(f)u}{v}_\E},\]
so that $f\mapsto \scal{\pi(f)u}{v}_\E$ is continuous.

Finally, we prove d). The definition of
$\mathcal M^\F$ and \eqref{eq:37} imply~\eqref{eq:38c}.  
 Given $T\in \Co{\E';\F}$, by definition $V_eT\in j(\F)$ and, hence, the
convolution $V_eT\con K$ exists. Furthermore,  by assumption $V_eT\con
K =V_eT $. Hence, condition~\ref{c3}) of Lemma~\ref{rsei} is satisfied
and this implies that $\pi(V_eT)u=T$, which is~\eqref{eq:38d}.   To
prove~\eqref{eq:38a} and~\eqref{eq:38b}, observe that the reproducing
formula implies that $V_e\Co{\E';\F}\subset j(\mathcal M^\F)$ and
equality~\eqref{eq:38d}  that $\Co{\E',\F}\subset 
\set{\pi(f)u\mid f\in \mathcal M^\F}$. Furthermore,  since $\mathcal
M^\F\subset\F$,  implies 
$\set{\pi(f)u\mid f\in \mathcal M^\F}\subset \Co{\E',\F}$ and, hence,
$j(\mathcal M^\F)\subset V_e\Co{\E',\F}$.  
\end{proof}
The above result is an adaptation of Theorem 2.3 of
\cite{chol11}. Conditions (R3)  and (R4) in \cite{chol11} are
replaced by~\eqref{eq:36} and the reproducing property~\eqref{eq:58}, respectively.

 Under all the assumptions of Proposition~\ref{cuno}, in particular
 the conditions of item d), the space $\Co{\E',\F}$ has a natural
 topology that makes it a Fr\'echet space.
\begin{cor}\label{topology}
The space $\Co{\E',\F}$ is a Fr\'echet space with respect to any of the following
equivalent topologies:
\begin{enumerate}[a)]
\item the topology induced by the family of semi-norms
  $\set{q_i(V_e(\cdot))}_i$, where $\set{q_i}_i$ is any fundamental family of
  semi-norms of $\F$;
\item the initial topology induced from the topology of $\F$ by the restriction of the voice $V_e$;
\item the final topology induced from the topology of $\F$, restricted
  to $\mathcal M^\F$, by the Fourier transform at $u$.
\end{enumerate}
\end{cor}
\begin{proof}
By Proposition~\ref{cuno}, $V_e$ is a bijection from $\Co{\E',\F}$ onto $\mathcal M^\F$ whose inverse is the Fourier transform at $u$,
the initial and final topologies on $\Co{\E',\F}$ coincide
and they realize $\Co{\E',\F}$ as a Fr\'echet space (isomorphic to $\mathcal M^\F$).
The equivalence between a) and b) is a standard result (see remark before Example~4 Ch. 2.11 of \cite{hor66}). The equivalence between b) and c) follows from the fact that $V_e$ is a bijection.
Since $\mathcal M^\F$ is a closed subspace of a Fr\'echet space, then both
$\mathcal M^\F$  and $\Co{\E',\F}$ are (isomorphic) Fr\'echet spaces.
\end{proof}

\subsection{Reproducing representations: the standard setup}\label{sec:reprod}
In this section, we  further assume that $\pi$ is 
a {\it reproducing} representation and   that  the vector $u$ is  an
admissible vector for $\pi$.  This means that the voice transform $V$
maps $\hh$ into $L^2(G)$ and that for all $v\in\hh$ 
\begin{equation}
  \label{eq:8}
  \nor{v}_\hh = \nor{Vv}_2.
\end{equation}
To stress that  the voice transform is an isometry of $\hh$ into $L^2(G)$, we write it with the suffix $2$:
\[
V_2: \hh\to L^2(G),\qquad V_2v(x)=\scal{v}{\pi(x)u}_\hh.
\]
Recalling that $K=V_2u$ and~\eqref{eq:32}, we have
  \begin{equation}
   \widecheck{\vu}=\overline{\vu}\in L^2(G). \label{eq:1}
  \end{equation}
In the following proposition, some consequences of
    the assumption that $\pi$ is reproducing are drawn. The results are 
    well known for irreducible representations \cite{dumo76,grmopa85} and their
     extensions to non-irreducible representations are taken for granted in many
    papers. We provide a proof based on Proposition~\ref{cuno}.

  \begin{prop}\label{tre} Suppose that $\pi$ is a reproducing representation of $G$ on $\hh$
  and that $u\in\hh$ is an admissible vector. Then:
\begin{enumerate}[a)]
\item  for every $f\in L^2(G)$, the Fourier transform of $f$ at $u$ exists in
$\hh$ and for all $v\in\hh$
\[ \scal{\pi(f)u}{v}_\hh=\scal{f}{V_2v}_2;\] 
\item  for every $f\in L^2(G)$ the convolution
$f\con \vu$ exists and
\begin{equation}
  \label{eq:13}
  V_2 \pi(f)u= f\con \vu,
\end{equation}
where both sides belong to $C_0(G)$ and, for every $v\in\hh$,
  \begin{equation}
    V_2v \con K = V_2v  \label{eq:2};
  \end{equation}
  in particular, $K=K\con K$.
\item the space
\[ \mathcal M^2=\set{f\in L^2(G)\mid f\con \vu=f}\] is a
$\la$-invariant closed subspace of $L^2(G)$ and
\begin{align}
  &  V_2 \hh= \mathcal M^2;  \label{eq:18} \\
  &   V_2 \pi(f)u = f,\qquad \text{for all } f\in \mathcal M^2; \label{eq:19}\\
  & \pi(V_2v)u = v,\qquad\text{for all } v\in \hh.\label{eq:20}
\end{align}
Hence,  the voice transform $V_2$ is a unitary map from $\hh$ onto $\mathcal M^2$ whose
inverse is given by the map $f\mapsto \pi(f)u$.
\end{enumerate}
\end{prop}
\begin{proof} We first prove \eqref{eq:2}.
By~\eqref{eq:6} with $p=q=2$, the convolution
  $V_2v \con  \vu$ exists and is in $C_0(G)$ because $\widecheck{\vu}\in
  L^2(G)$ by~\eqref{eq:1}.  Furthermore, given $x\in G$, for all $y\in G$
  \[V_2v(y) \,\la(x)\widecheck{\vu}(y)= V_2v(y)\, \overline{\la(x)
    \vu(y)}=V_2v(y)\, \overline{ (V_2 \pi(x) u)(y)}.\] 
Integrating with respect to  $y$, we  obtain
  \[ V_2v\con \vu = \scal{V_2v}{V_2\pi(x)u}_{2}=
  \scal{v}{\pi(x)u}_\hh=V_2v.\] 
To prove the remaining statements,
we apply  Proposition~\ref{cuno} with $\F=L^2(G)$ and $\E=\E'=\hh$, with the understanding that $i$ and $j$ are the canonical inclusions, $\la=\ell$, $\pi=\tau$ and $V=V_2$. Observe that  \eqref{eq:36} is
satisfied since $V_2\hh\subset L^2(G)=L^2(G)^\#$ and, by
\eqref{eq:2}, the
reproducing formula~\eqref{eq:58} holds for every $v\in\hh$, regarded
as anti-linear form on $\hh$. Furthermore, by~\eqref{eq:6} in the appendix  with $p=q=2$,
for all $f\in L^2(G)$ the function $f\con \vu$ is in $C_0(G)$, taking~\eqref{eq:1} into account.
\end{proof}

\section{Main results}

In this section, we assume that the representation $\pi$ is  reproducing 
and that the vector $u\in\hh$ is  admissible, as in Section~\ref{sec:reprod}.
We will construct a coorbit space theory based on the choice of a
suitable target space $\T$ embedded in $\Lz$.

\subsection{The space of test functions and distributions}\label{testanddist} 
We choose  a Fr\'echet space $\T$ with 
\begin{enumerate}[i)]
\item a continuous embedding $j:\T\to L^0(G)$;
\item  a continuous representation $\ell$ of $G$ acting on $\T$ such that
  $j\ell(x)=\la(x)j$ for all $x\in G$;
\item a continuous involution $f\mapsto \overline{f}$ such that $\overline{j(f)}=j({\overline{f}})$,
\end{enumerate}
so that $\T$ enjoys  all the properties  of the space $\F$  in
Section~\ref{sec:target}, from which we adopt the notations.
In particular, as in \eqref{eq:11}, we put
\[ 
\mathcal M^\T=\set{f\in \T\mid j(f)\con \vu=j(f)} . 
\]
The classical theory corresponds to the choice $\T=L^1(G)$, or a
weighted version of it. The following assumptions are at the root of our
construction and are trivially satisfied for $L^1(G)$. 

\begin{ass}\label{H1}
The kernel $K$ is in $j(\T)$ and $j(f) \vu \in L^1(G)$ for all
$f\in\T$, {\em i.e} $K\in j(\T)\cap\T^\#$.
\end{ass}
\begin{ass}\label{H2} For all $f\in\mathcal M^\T$  and all $v\in\hh$ we have $j(f) V_2v \in L^1(G)$, {\em i.e} 
$V_2\hh\subset (\mathcal M^\T)^\#$.
\end{ass}
\begin{ass}\label{H3}
The linear space spanned by the orbit $\set{\ell(x)K\mid x\in G}$ is  dense in $\mathcal M^\T$.
\end{ass} 
By Proposition~\ref{uno},  Assumption~\ref{H1} implies
that for all $f\in\T$ the convolution $j(f)\con K$ exists and   $\mathcal M^\T$
is a closed $\ell$-invariant subspace of $\T$,
so  that $\lspan{\ell(x)K\mid x\in G}$  is a subspace of  $\mathcal
M^\T$. Assumption~\ref{H3} is formulated with a slight abuse of notation,  regarding $K$ as
an element of $\T$. It is a strengthening of Assumption~\ref{H1} because it is equivalent to the 
requirement that $K$ is actually a  cyclic vector for the representation $\ell$ restricted to $\mathcal M^\T$. 

Assumptions~\ref{H1} and~\ref{H2} should be compared with
hypotheses~(R2) and (R3) of \cite{chol11}. In our approach
they are needed to define the test space, as in the classical
setting, whereas in \cite{chol11} the test space is given a-priori. 
\noindent 
We are now in a position to define the space of {\it test signals}, namely 
\begin{equation}
 \Ss= \set{v\in \hh \mid V_2v\in j(\T)}.
 \label{test}
 \end{equation}
 We define the {\it restricted voice transform} $V_0:\Ss\to \T$ as
the unique map satisfying $jV_0=V_2i$, that is, for all $v\in\Ss$ and $x\in G$ we put
\[  (V_0v)(x)=\scal{i(v)}{\pi(x)u}_\hh  \]
where $i: \Ss\to\hh$ is the canonical inclusion.
 It is by means of $V_0$ that we topologize $\Ss$: 
we endow $\Ss$ with the  initial topology induced by
$V_0$. As it will be shown in Theorem~\ref{tuno} below,  this is
just an explicit description of the topology that $\Ss$ naturally inherits
as coorbit space, because $\Ss=\Co{\hh,\T}$.
 Observe that Assumption~\ref{H1} implies that  $u\in\Ss$, since
$K\in j(\T)$. 
\begin{thm}\label{tuno}
 The space $\Ss$ is a Fr\'echet space
isomorphic to $\mathcal{M}^\T$ via $V_2$, and $ j(\mathcal{M}^\T) \subset L^2(G) $.
 The canonical embedding
 $i:\Ss\to\hh$ is continuous and   has dense range. The transpose $\tr{i}:\hh_s\to\Ss'_s$ is continuous, injective and has dense range. The representation   $\pi$ leaves  $\Ss$ invariant, its restriction $\tau$ to $\Ss$ is a
continuous  representation of $G$ acting on $\Ss$ and $u$ is a cyclic
vector of $\tau$.
\end{thm}
\begin{proof} We first prove that $\Ss$ is a Fr\'echet space.
Let $\E=\hh$ and $\F=\T$. By the properties i), ii) and iii) stated at
the beginning of this section, we are in the general setting
of Section~\ref{sec:ext}. Observe that $\hh'=\hh$, 
$V_e=V_2$ and  clearly $\Ss=\Co{\hh,\T}$. 
Furthermore, the fact that $\pi$ is reproducing
implies that $V_2$ is injective and, by~\eqref{eq:2} in
Proposition~\ref{tre}, the reproducing formula holds true for all
$v\in\hh$. Hence  $V_2v\in j(\mathcal
M^\T)$ for all $v\in\Ss$,  and we actually get  $\Ss=\Co{\hh,\mathcal M^\T}$.
We can
apply Corollary~\ref{topology} because the hypotheses  of 
Proposition~\ref{cuno}  that imply it are both satisfied: \eqref{eq:36} is just
Assumption~\ref{H2} and, as already noticed, the reproducing property holds for all
$v\in\Ss$ because  $\pi$ is reproducing. Hence $\Ss$ is a
Fr\'echet space and $V_0$ induces a topological linear  isomorphism from
$\Ss$  onto $\mathcal M^\T$.
Since $V_2\hh\subset L^2(G)$, clearly $j(\mathcal M^\T)\subset L^2(G)$.

Since $\Ss$ and $\mathcal M^\T$ are isomorphic, in order  to show that
 $i$ is continuous it is enough to prove that $j$ is continuous from
 $\mathcal M^\T$ into  $L^2(G)$. Both are Fr\'echet spaces, 
hence it is sufficient to show that $j:\mathcal M^\T\to L^2(G)$
has sequentially closed graph.  If $(f_n)_n$ is a sequence in  $\mathcal M^\T$
converging to $f$ in $\mathcal M^\T$  and $(j(f_n))_n$ converges  to $\varphi$ in $L^2(G)$,
then possibly passing to a subsequence, we can assume that  $(f_n(x))_n$
converges for almost all $x\in G$. Hence $\varphi(x)=f(x)$ almost
everywhere.

Item b) of Proposition~\ref{cuno}
gives that $\pi$ leaves $\Ss$ invariant. Since for all $x\in G$ and $v\in\Ss$
\[ 
V_0\tau(x)v=\ell(x)V_0v,
\]
the restriction $\tau$ is  a continuous representation on $\Ss$ because $\ell$
is a continuous  representation on~$\T$.
The fact that $\pi$ is reproducing
implies that $\lspan{\pi(x)u\mid x\in G}\subset\Ss$ is dense in $\hh$,
so that $i$ has dense range. 

Finally, since $V_0\tau(x)u=\ell(x)K$ for all $x\in G$,
Assumption~\ref{H3} is another way of saying that $u$ is a cyclic vector for $\tau$.  As
for the properties of $\tr{i}$, Corollary~3 Chapter II.6.3 of
\cite{BTVS} shows that $i$ is continuous from $\Ss_{weak}$ into $\hh_{weak}$.
Hence, Corollary of Proposition~5 Chapter II.6.4 of \cite{BTVS} gives
that $\tr{i}$ is continuous from $\hh_s=\hh_{weak}$ into $\Ss_s'$ and
$\tr{\,(\!\tr{i})}=i$. Finally, Corollary 2 Chapter II.6.4 of
\cite{BTVS} shows that since $i$ is injective and has dense range,
$\tr{i}$ has the same properties.
\end{proof}
As shown in the above proof, all the hypotheses 
of Corollary~\ref{topology} are satisfied.  This implies that, whenever a fundamental family $\set{q_i}_i$ of
semi-norms of $\T$ is given, then $\set{q_i(V_0(\cdot))}_i$ is a fundamental family of
semi-norms of $\Ss$. This is yet another way to get a direct handle on its topology
when a family of seminorms  of $\T$ is  known.

We regard the dual $\Ss'$ of $\Ss$ as  the space of {\it
  distributions} and we define the {\it extended voice transform} on it
by setting for all $T\in\Ss'$
\begin{equation}
V_e:\Ss'\to C(G),\qquad  V_eT=\scal{T}{\tau(\cdot)u}_\Ss.
\label{EVT2}
\end{equation}
The definition works because $\Ss$ is $\tau$-invariant, $u\in\Ss$ and
$\tau$ is a continuous representation. The following theorem states the
main properties of $V_0$ and $V_e$. We recall that 
the contragredient representations $\tr{\tau}$ and $\tr{\ell}$ are continuous
representations acting  on $\Ss'$ and $\T'$, respectively, where the
dual spaces are endowed with the topology of the convergence
on compact subsets (see  Proposition~3 Chapter VIII.2.3 of \cite{BINT2}).
Furthermore, since $\pi$ is a reproducing representation,
Proposition~\ref{tre} ensures that for all $f\in L^2(G)$
the Fourier transform of $f$ at $u$ exists in $\hh$.
\begin{thm}\label{lsei}
The restricted voice transform $V_0$ is an injective strict morphism\footnote{A
  strict morphism is a continuous linear map whose image is closed.}
from $\Ss$ into $\T$ with image  $\mathcal M^\T$. For all $f\in\mathcal M^\T$,  we have
\[ 
\pi(f)u\in\Ss, \qquad V_0\pi(f)u=f
\]
 and,  for all $v\in\Ss$, we have
 \[\pi(V_0v)u=v.\]
Furthermore, $V_0$ intertwines  $\tau$ and $\ell$ and 
its transpose $\tr{\,V_0}: \T'_s\to \Ss'_s$ is a surjective continuous
map,  intertwining the representations $\tr{\ell}$ and $\tr{\tau}$.

The  extended voice transform $V_e$ intertwines $\tr{\tau}$ with $\la$,  is injective and continuous from $\Ss'$ to $C(G)$, where both spaces are endowed with the topology of compact  convergence.
Finally, for all $\Phi\in \T^\#\subset \T'$, we have
\begin{equation}
  \label{eq:22}
  V_e\tr{\,V_0}\Phi=\Phi\con \vu.
\end{equation}
\end{thm}
\begin{proof} 
By Theorem~\ref{tuno}, $V_0$ induces a topological linear isomorphism from
$\Ss$  onto $\mathcal M^\T$, which is a closed subspace of $\T$.
  Corollary~1, Chapter II.4.2 of \cite{BTVS} gives that $\tr{\,V_0}$ is
  surjective. By Corollary of Proposition~5,  Chapter II.6.4 of~\cite{BTVS},
the map  $\tr{\,V_0}$ is continuous when both $\T'$ and
  $\Ss'$ are equipped with the topology of the simple convergence. 

Since $\pi$ is reproducing  and $j(f)\in L^2(G)$ for all
$f\in\mathcal M^\T$, Proposition~\ref{tre} shows
that $V_2\pi(j(f))u= j(f)\in j(\T)$. Hence, by definition of $\Ss$, $\pi(j(f))u\in\Ss$
and the construction of $V_0$ gives that $V_0\pi(f)u=f$, where, with slight abuse of notation,
$\pi(f)u$ is the Fourier transform of $j(f)$ at $u$. Take now
$v\in\Ss$. Since
$V_0v\in\mathcal M^\T$,   again Proposition~\ref{tre}  yields $\pi(V_0v)u=v$.

The intertwining property is straightforward: for any  $x,y\in G$
\[ 
\left(V_e\tr{\tau}(x)T\right)(y)= \scal{T}{\tau(x^{-1})\tau(y)u}_\Ss = V_eT(x^{-1}
y).\]
Injectivity is due to the fact that $u$ is cyclic for $\tau$.
To prove that $V_e$ is continuous, fix a compact subset $Q$ of $G$.
Since $x\mapsto \tau(x)u$ is continuous, the set $A=\tau(Q)u$ is
compact in $\Ss$,  and $T\mapsto \sup_{v\in
  A}\abs{\scal{T}{v}_\Ss}$ is continuous on $\Ss'$. 
  Finally, take $\Phi\in \T^\#$. Then for all $x\in G$
\[
V_e\tr{\,V_0} \Phi(x)=\scal{\Phi}{V_0\tau(x)u}_\T
=\int_G \Phi(y)\overline{\scal{\pi(x)u}{\pi(y)u}_{\mathcal H}}dy
= (\Phi\con V_2u)(x). \qedhere
\]
\end{proof}
We add a remark  on the finer  topological properties of $V_e$. 
If $B$ is a bounded subset of $\Ss'$ or, equivalently, of $\Ss'_s$, 
the restriction of $V_e$ to $B$, endowed with the topology of $\Ss'_s$, 
into $C(G)$, with the topology of the compact  convergence, is
continuous. Indeed, since $\Ss$ is a Fr\'echet space, then it is
barrelled (Corollary of Proposition~2 Chapter III.4.2 of \cite{BTVS}). Hence 
\[ \text{strongly bounded} \Leftrightarrow  \text{weakly bounded}
\Leftrightarrow  \text{equicontinuous}.\]  
(Scolium and Definition~2 Chapter III.4.2 of \cite{BTVS}).
Proposition 5 Chapter III.3.4 of \cite{BTVS} implies that on $B$ the
topology of the simple convergence is equivalent to the topology of
precompact subsets. Hence, for  any compact subset ${\mathcal K}$ of $G$,
since $x\mapsto \tau(x)u$ is continuous, the set $A=\tau({\mathcal K})u$ is
compact in $\Ss$,  hence precompact  and, by the above reasoning 
$B\ni T\mapsto \sup_{v\in  A}\abs{\scal{T}{v}_\Ss}$ is continuous with
respect to the topology of  the simple convergence.
   
The next assumption requires that the reproducing formula holds for any
distribution in $\Ss'$.

\begin{ass}\label{H4}
For all $T\in\Ss'$, $K\, V_eT\in L^1(G)$ and $V_eT\con K=V_eT$.
\end{ass}
Since the representation $\tr{\tau}$ leaves  $\Ss'$ invariant and $V_e$
intertwines $\tr{\tau}$ with $\la$, the requirement ${K\, V_eT}\in
L^1(G)$  implies that $V_eT\con K$ exists. Furthermore, if  $V_e\Ss'$ is
contained in $\T^\#$, then ${K\, V_eT}\in L^1(G)$ holds for all $T\in\Ss'$ since $K\in\T$.  

In the two propositions that follow, we  give sufficient conditions implying Assumption~\ref{H4}. 
\begin{prop}\label{repr-formula}
Assume that $\mathcal M^\T$ is a reflexive space  and ${K\, V_eT}\in
L^1(G)$  for all $T\in\Ss'$. Then  the reproducing formula $V_eT\con
K= V_eT$ holds true for all $T\in\Ss'$.
\end{prop}
\begin{proof}
Since $\Ss$ and $\mathcal M^\T$ are isomorphic (Theorem~\ref{tuno}), then also $\Ss$ is a
reflexive space. Regard $\Ss$ as the dual of $\Ss'$, which has the property (GDF) by
  Proposition~3 Chapter 6. Appendix No.2 of \cite{BINT1}.  The
  assumption implies that the map $x\mapsto \tau(x)u
  \scal{\pi(y)u}{u}_\hh$ is scalarly integrable from $G$ to $\Ss_s$,
  hence Theorem~\ref{Gelfand-Dunford} shows that there exists $v_u\in\Ss$ such that
  \[ \scal{T}{v_u}_\Ss = \int_G \scal{T}{\tau(x)u}_\Ss
  \scal{\pi(x)u}{u}_\hh dx.\] 
By Theorem~\ref{tuno}, $\hh$ is dense in
  $\Ss_s'$ and by \eqref{eq:2} $v_u=u$, which means that
\[ \scal{T}{u}_\Ss = \int_G \scal{T}{\tau(x)u}_\Ss
  \scal{\pi(x)u}{u}_\hh dx.\] 
Given $y\in G$, by applying the above equality to $\tr{\tau(y^{-1})}T$, we get 
    \begin{align*}
      V_eT\,(y) &=\scal{T}{\tau(y)u} _{\mathcal S}= \scal{\tr{\tau(y^{-1})}T}{u} _{\mathcal S}\\
& =\int
      \scal{\tr{\tau(y^{-1})}T}{\tau(x)u}_\Ss
      \scal{\pi(x)u}{u}_\hh\,dx \\
&  = \int
      \scal{T}{\tau(yx)u}_\Ss \scal{\pi(x)u}{u}_\hh\, dx \\
&  = \int
      \scal{T}{\tau(x)u}_\Ss \scal{u}{\pi(x^{-1}y)u}_\hh\, dx,
    \end{align*}
where  the last line is due to the change of variable $x\mapsto
y^{-1}x$ and the fact that $\pi$ is a unitary representation. Hence
the convolution $VT\con K$ exists and is equal to $V_eT$.
\end{proof}

The property
 ${K\, V_eT}\in
L^1(G)$  for all $T\in\Ss'$ means that the map $x\mapsto \tau(x)u
  \scal{\pi(y)u}{u}_\hh$ is scalarly integrable from $G$ to
  $\Ss$, {\em i.e.}, there exists {a linear map} $\omega:\Ss'\to\C$ such that
\[ \omega(T)= \int_G \scal{T}{\tau(x)u}_\Ss
  \scal{\pi(x)u}{u}_\hh dx.\]
Furthermore, since $\hh$ is continuously embedded in $\Ss'$ and $\pi$ is a
reproducing representation, for all ${w}\in \hh$ we have
\[ \omega(\tr{i}(w))=\scal{w}{u}_\hh.\]
By Theorem~\ref{tuno}, the map $\tr{i}$ has a dense image in
$\Ss'_s$. However, $\omega$ is continuous with respect to the {weak-$*$
topology of $\Ss'$ if and} only if
$\omega\in\Ss$. In the setting of reproducing representations, the
requirement that the reproducing formula holds for all distributions is
equivalent to assuming that $\omega\in \Ss$ and, in this case,
$\omega$ is precisely $u$. The hypothesis $\omega\in\Ss$ is
precisely property (R4) in \cite{chol11}. Furthermore,
if $\Ss$ is a Banach space, as in the classical setting, and if the map $x\mapsto \tau(x)u
  \scal{\pi(y)u}{u}_\hh$ is Bochner-integrable, then it is scalarly
  integrable and, clearly, $\omega$ is always in $\Ss$. 

Here is another sufficient condition.  
\begin{prop}\label{suff2}
Assume that 
$\T^\#=\T'$ and suppose that  $\abs{\vu}\con \abs{\vu}$ exists and belongs to
 $\T$. Then  $V_eT\con K=V_eT$ for all $T\in\Ss'$. 
\end{prop}
\begin{proof} 
By Theorem~\ref{tuno}, $\tr{\,V_0}$ is surjective, so that if $\T'=\T^\#$, then for any $T\in\Ss'$   there exists $\Phi\in\T^\#$  such that $\tr{\,V_0}\Phi=T$. Furthermore, if  $\abs{\vu}\con \abs{\vu}$ exists and belongs to
 $\T$, then 
 \begin{align*}
  \int_{G\times G} \abs{\Phi(zx)\scal{\pi(x)u}{\pi (y)u}_\hh \scal{\pi(y)u}{u}_\hh}dx\,dy & = \int_G
    \abs{\Phi(zx)} \left(\int_G \abs{\vu(y)}
      \abs{\vu(y^{-1}x)}dy\right)\,dx  \\
& =   \int_G \abs{\la(z^{-1})\Phi(x)} (\abs{\vu}\con \abs{\vu})(x)\,dx
\end{align*}
and, since $\abs{\la(z^{-1})\Phi}\in\T^\#$, the last integral is finite
for all $z\in G$. By~\eqref{eq:34}  and~\eqref{eq:1},  we have that $\widecheck{ \abs{\vu}\con
  \abs{\vu}}=\abs{\vu}\con \abs{\vu}$, hence Fubini theorem implies
that the convolution $\abs{\Phi}\con
(\abs{\vu}\con \abs{\vu})$ exists, and ~\eqref{eq:35} in the appendix shows
\[ (\Phi\con\vu)\con\vu=\Phi\con(\vu\con\vu).\]
Finally, \eqref{eq:22} and~\eqref{eq:2} give
\[ V_eT\con\vu=(\Phi\con\vu)\con\vu=\Phi\con(\vu\con\vu)=\Phi\con\vu=V_eT. \qedhere \]
\end{proof}

\subsection{Coorbit spaces}

We now fix a Banach space $Y$, with norm $\nor{\cdot}_Y$, continuously
embedded in $\Lz$ and $\la$-invariant.  In order to be consistent with
the current literature, we do not indicate the explicit embedding as
we did for the other spaces. The results in this section hold true under the weaker
  assumption that $Y$ is a Fr\'echet space. However, we do not need
  this generality because the main example that we are interested in 
is the case when $Y$ is a  weighted $L^p$ space for a fixed value of $p$.

The {\it coorbit space} of $Y$ is
\begin{equation}
 \Co{Y}  = \set{T\in \Ss' \mid V_e T\in Y}
\label{eq:21}
\end{equation}
endowed with the norm
\begin{equation}
 \nor{T}_{\Co{Y}}  = \nor{V_eT}_Y. \label{eq:23}
\end{equation}
Since $V_e$ is a linear injective map, $\nor{\cdot}_{\Co{Y}}$
is clearly  a norm. We will prove below that $\Co{Y}$ is in fact a Banach space.

Just as for the target space $\T$, the two basic assumptions for the space $Y$   
may be formulated in terms of K\"othe duals and have to do with the
kernel $K$ (compare Assumption~\ref{H5} below with
Assumption~\ref{H1}) and with the image of the voice transform
$V_2\hh$ (compare Assumption~\ref{H6} below with
Assumption~\ref{H2}). They should also be compared with the
corresponding assumptions made  in  \cite{C2012,chol09,chol11}.
As above, we write
\[
\mathcal M^Y=\set{f\in Y \mid f \con \vu=f}.
\]
\begin{ass}\label{H5} For all $f\in Y$, we have $f \vu \in L^1(G)$, that is, $K\in Y^\#$.
\end{ass}

\begin{ass}\label{H6}
For all $f\in\mathcal M^Y$ and all $v\in\Ss$, we have $f{V_0v}\in L^1(G)$, {\em i.e.},
$V_0\Ss\subset (\mathcal M^Y)^\#$.
\end{ass}
By Proposition~\ref{uno} applied to  $\F=Y$,  Assumption~\ref{H5}
implies that $\mathcal M^Y$ is a $\la$-invariant closed subspace of $Y$. Furthermore, by
Proposition~\ref{due} with $\E=\Ss$, Assumption~\ref{H6} implies that
for all $f\in\mathcal M^Y$ the Fourier transform of $f$ at $u$ exists
in $\Ss'$.

In the following proposition we list the main properties of $\Co{Y}$.
\begin{prop}\label{prop6}
The space $\Co{Y}$ is a Banach space invariant under the action of the
representation $\tr{\tau}$. The extended voice transform is an
isometry from $\Co{Y}$ onto $\mathcal M^Y$ and its inverse is the
Fourier transform at $u$. Therefore
\begin{align*}
&  V_e\Co{Y}= \mathcal M^Y,\\
& \set{\pi(f)u\mid f\in\mathcal M^Y}= \Co{Y},\\
& V_e \pi(f)u = f, \qquad f\in \mathcal M^Y, \\
& \pi(V_eT)u = T, \qquad T\in \Co{Y}.
\end{align*}
\end{prop}
\begin{proof} 
The proposition is a restatement of Proposition~\ref{cuno} and
Corollary~\ref{topology} with $\E=\Ss$ and $\F=\mathcal
M^T$. The hypothesis~\eqref{eq:36} is Assumption~\ref{H6} and the hypothesis
in item~d) of Proposition~\ref{cuno} is satisfied by Assumption~\ref{H3} and Assumption~\ref{H4}.  
\end{proof}

As in the classical setting, we have the following
  canonical identification.
  \begin{cor}\label{coltwo} 
The Hilbert space $L^2(G)$ satisfies
 Assumptions $\ref{H5}$ and $\ref{H6}$, and 
$\Co{L^2(G)}=\hh$.
  \end{cor}
  \begin{proof}
Since $\pi$ is a reproducing representation,  Assumptions \ref{H5},
and \ref{H6}  are clearly satisfied, and $\hh\subset \Co{L^2(G)}$. Take now
$T\in \Co{L^2(G)}$. By   Proposition~\ref{prop6}
$T=\pi(V_eT)u$. However, since $V_2T\in
L^2(G)$, by Proposition~\ref{tre} $\pi(V_eT)u\in\hh$.
  \end{proof}
Even though $\T$ is not a Banach space, the space
\[\Co{\T}=\set{ T\in\Ss'\mid V_eT\in \T}\]
is well defined and, under Assumption~\ref{H4}, Corollary~\ref{coltwo}
 and the definition of $\Ss$ imply that, as in the classical
setting, $\Co{\T}=\Ss$.  The above identification suggests to characterize the space
\[\Co{\T'}=\set{ T\in\Ss'\mid V_eT\in \T^\#}\subset\Ss'.\]
The equality $\Co{\T'}=\Ss'$ is equivalent to require that $j(f)V_eT\in
L^1(G)$ for all $f\in\T$ and $T\in\Ss'$, that is, $V_e\Ss'\subset
\T^\#$, which is in general stronger than Assumption~\ref{H4}.

Let us compare our approach with the theory developed by
 J.~Christensen and G.~{\'O}lafsson in 
\cite{C2012,chol09,chol11}.  Assumptions~\ref{H1}$\div$\ref{H6} ensure
that the test space $\Ss$ defined by~\eqref{test} satisfies  the
properties (R1)$\div$(R4), and some of our claims can be directly deduced by
the results contained in \cite{chol11} (for example, compare Theorem~2.3 of
\cite{chol11} with our Proposition~\ref{prop6}). 
In our setting, which is somehow parallel to the classical $L^1$ case, we first
introduce the target space $\T$, which is independent of the reproducing
representation, and then we define the test space $\Ss$ as the set of vectors for
which the voice transform belongs to $\T$. The introduction of the
target space $\T$ makes our construction closer to the classical approach
by H. Feichtinger and K. Gr\"ochenig,   and  Assumptions~\ref{H1}, \ref{H2}, \ref{H3}
and \ref{H4}
involve only the target space $\T$ without any reference to the model
space $Y$. Moreover, our proofs mainly rely on
the theory of weak integrals {\it \`a la Dunford-Pettis}, which
allows us  to state our hypotheses as integrability conditions, rather than a
continuity requirement as in \cite{chol11}.

Assumption \ref{H4} requires that the reproducing formula $V_eT\con
K=V_eT$ holds for all $T\in\Ss'$. However, in the proof of 
Proposition~\ref{prop6}, the reproducing formula is needed only for the
distributions in $\Co{Y}$ (compare with item d) of Proposition~\ref{cuno}).
The following lemma shows some equivalent conditions, weaker
than Assumption~\ref{H4}, under which Proposition~\ref{prop6} remains
true. 

\begin{lem}\label{lnove} Take $\T$ and $Y$ such that Assumptions~$\ref{H1},\ref{H2},\ref{H3}$ and
  Assumptions~$\ref{H5},\ref{H6}$ hold true.  Then the following facts are equivalent: 
\begin{enumerate}[a)]
\item\label{item:1} for all $T\in \Co{Y}$,  $V_eT\in \mathcal M^Y$;
\item\label{item:2} for all $T\in \Co{Y}$,  $V_eT\con \vu$ exists and $V_eT\con
  \vu=V_eT$;
\item\label{item:3} for all $T\in \Co{Y}$, the map $x\mapsto
    \scal{T}{\tau(x)u}_\Ss \scal{\pi(x)u}{u}_\hh =V_eT(x)\overline{K(x)}$ is in $L^1(G)$ and 
\begin{equation}
 \int_G \scal{T}{\tau(x)u}_\Ss \scal{\pi(x)u}{u}_\hh dx =
\scal{T}{u}_\Ss;
\label{REPco}
\end{equation}
\item\label{item:4} for all $T\in \Co{Y}$, the map $x\mapsto V_eT\,(x) \tr{i}(\pi(x)u)\in \Ss_s'$ is scalarly
  integrable and its scalar integral is $T$, that is
  \begin{equation}
T = \int_G \scal{T}{\tau(x)u}_\Ss \tr{i}(\pi(x)u)\,dx.\label{eq:41bis}
\end{equation}
\end{enumerate}
\end{lem}
\begin{proof}
By definition of coorbit space, $V_eT\in Y$ whenever $T\in\Co{Y}$. Hence   \ref{item:1})  is equivalent to
\ref{item:2}). Since $\Co{Y}$ is $\tr{\tau}$-invariant, \ref{item:3})
implies that the map  $y\mapsto
    \scal{\tr{\tau(x^{-1})}T}{\tau(y)u}_\Ss \scal{\pi(y)u}{u}_\hh $ is
    integrable  for all $x\in G$ and
    \begin{align*}
      V_eT\,(x) =\scal{T}{\tau(x)u} _{\mathcal S}& =\int
      \scal{\tr{\tau(x^{-1})}T}{\tau(y)u}_\Ss \scal{\pi(y)u}{u}_\hh\,dy \\
&  = \int
      \scal{T}{\tau(y)u}_\Ss \scal{\pi(y)u}{\pi(x)u}_\hh\, dy.
    \end{align*}
Hence c) implies item~\ref{e2}) of Lemma~\ref{rsei}. The converse is 
also true by evaluation at the identity. Therefore c) is equivalent to item~\ref{e2}) 
of Lemma~\ref{rsei}, which provides the equivalence between \ref{item:2}) and
\ref{item:3}) and shows that \ref{item:4}) implies \ref{item:3}). 

Assume now that $V_eT\in \mathcal M^Y$. Proposition~\ref{due} with
$f=V_eT$ gives that $V_eT$ satisfies~\eqref{eq:17}, that $\pi(V_eT)u\in \Ss'$ exists and $V_e\pi(V_eT)u=
V_eT$.  Finally, since $V_e$ is injective by Theorem~\ref{lsei}, we know
from item~\ref{c2})  of Lemma~\ref{rsei} that \ref{item:1}) implies
\ref{item:4}). 
\end{proof}

\subsection{Dependency on the admissible vector}\label{AV}
We now examine the dependence of  space $\Ss$  on  the choice of the admisible vector $u$. 
For this reason, in this section, we write $\Ss_u$ instead of $\Ss$, and accordingly for other choices of admissible vectors.

\begin{prop}\label{dependency}
Suppose that $j(\T)\con j(\widecheck{\T})\subset j(\T)$ and that  for all $g\in \T$ the map 
\begin{equation}
f\mapsto f\con\check g\label{eq:59}
\end{equation}
is continuous from $\T$ into itself, where  $j(f\con\check g)=j(f)\con\check j(g)$.
If ${\tilde u}\in\Ss_u\subset \hh$ is another admissible vector satisfying Assumptions $\ref{H1}$, $\ref{H2}$ and $\ref{H3}$, then the test function spaces $\Ss_{\tilde u}$ and $\Ss_u$ coincide as Fr\'echet spaces. Furthermore, $\widecheck{\Ss_u}= \Ss_u$ for any admissible $u$ and
\[
V_{\tilde u}v=V_{u}v\con\widecheck{\overline{V_u\tilde u}}.
\]
for all $v\in\Ss_{\tilde u}=\Ss_u$.
\end{prop}
\begin{proof}
Let $v\in\Ss_u$ and $x\in G$. Since $\pi$ is reproducing and $u$ is admissible,
\begin{align*}
 V_{2,{\tilde u}}v\, (x)  & = \scal{v}{\pi(x){\tilde u}}_\hh \\
& = \int_G \scal{v}{\pi(y)u}_\hh
 \overline{\scal{\pi(x){\tilde u}}{\pi(y)u}_\hh}\, dy \\ 
& = \int_ G \scal{v}{\pi(y)u}_\hh
\overline{\scal{{\tilde u}}{\pi(x^{-1}y)u}_\hh}\, dy\\
&=  V_{2,u}v\con \widecheck{\overline{V_{2,u}{\tilde u}}} \,(x), 
\end{align*}
where  $V_{2,u}v,V_{2,u}\tilde u\in j(\cT)$ since $v,{\tilde u}\in\Ss_u$. The hypothsis on $\T$  implies that   
$V_{2,{\tilde u}}v\in j(\T)$, so that $\Ss_u\subset \Ss_{\tilde u}$. 
 We now prove that the embedding of  $\Ss_u$ into $\Ss_{\tilde u}$ is continuous. Fix a semi-norm  
 $\nor{\cdot}_{i,\Ss_{\tilde u}}$ of
$\Ss_{\tilde u}$, {\em i.e},  fix a semi-norm $\nor{\cdot}_{i,\T}$ of $\T$ such
that $\nor{v}_{i,\Ss_{\tilde u}}=\nor{V_{2,{\tilde u}} v}_{i,\T}$ for all
$v\in\Ss_{\tilde u}$. By~\eqref{eq:59}  with $f=V_2v$ and $g=\overline{V_0{\tilde u}}$, there exist a constant $C>0$ and
a semi-norm $\nor{\cdot }_{j,\T}$ of $\T$ such that 
\[
\nor{v}_{i,\Ss_{\tilde u}}=\nor{V_{2,{\tilde u}} v}_{i,\T} \leq C \nor{V_2 v}_{j,\T} = C  \nor{v}_{j,\Ss_u}
\]
where $ \nor{\cdot}_{j,\Ss_u}$ is a semi-norm of $\Ss_u$. Hence, the embedding is continuous.
Interchanging the roles of $u$ and ${\tilde u}$, we obtain that  $\Ss_{\tilde u}\subset \Ss_u$
with a continuous embedding.  Finally by~\eqref{eq:34}  in the appendix and~\eqref{eq:1},  for all $v\in\Ss_u$,
\[ \widecheck{ V_0v}= \widecheck{ V_0v \con K} = \overline{K}\con
\widecheck{ V_0v}\in\T\]
by assumption,  so that $\widecheck{ \Ss_u}\subset\Ss_u$ and, hence, $\widecheck{ \Ss_u}=\Ss_u$.
\end{proof}
In the classical framework, $\pi$ is irreducible and $\T=L^1(G,w\beta)$,
where $w$ is a continuous density satisfying~\eqref{eq:33} and  \eqref{eq:46}  in the appendix and 
\begin{equation}  
\label{eq:45}
w(x) =   w(x^{-1}) \Delta(x^{-1}). 
\end{equation}
This last condition implies that $\T=\widecheck{\T}$ so that 
the hypotheses of the above proposition are satisfied. However, a
stronger result holds true, namely
\[ \set{ u\in \hh \mid K_u\in \T}=\Ss,\]
which is the content of  Lemma 4.2 in \cite{fegr88}. 
Note that the irreducibility ensures that, if $K_u\in \T$, then $u$ is
an admissible vector.    However, if $\pi$ is not irreducible, the
above equality does not hold as shown by a counter-example  in 
\cite{F2013}.


\section{A model for the target space}\label{intersection}
In this section, we illustrate some examples. They include
band-limited functions (Section~\ref{BAND}), 
Shannon wavelets (Section~\ref{shannon})
and {\it Schr\"odingerlets} (Section~\ref{SCR})
that have inspired our theory.

\subsection{ Intersection of all $L^p_w(G)$ with $1<p<+\infty$}\label{mainexample}
In this section, $w:G\to (0,+\infty)$ will denote a continuous function, to be called {\it weight}, satisfying \begin{subequations}
  \begin{align}
    \label{eqb:33}
    w(xy) & \leq w(x)w(y) \\
w(x) & = w(x^{-1})  \label{eqb:45} 
  \end{align}
for all $x,y\in G$.  As a consequence, it also holds that
  \begin{equation}
     \inf_{x\in G} w(x)  \geq 1  \label{eqb:46}. \\    
  \end{equation}
\end{subequations}
The notion of weight  in \cite{gro91} is based on the  submultiplicative property~\eqref{eqb:33}. The symmetry~\eqref{eqb:45} can always be satisfied by
replacing $w$ with $w+\widecheck{w}$. This requirement is necessary for
our development (see item g) of Theorem~\ref{intersections}
below). Condition~\eqref{eqb:46} is explicitly stated in \cite{gro91}
and, in the classical $L^1(G)$ setting, it is necessary to ensure that
the test space is a Banach space (see Theorem~\ref{Sbanach} below).
In the usual  irreducible $L^1$ setting, it is also assumed that the weight satisfies~\eqref{eq:45}, which is actually incompatible with~\eqref{eqb:45}. However, \eqref{eq:45} is only necessary in order to see that the space of admissible vectors coincides with the test space (see Lemma~4.2 in  \cite{fegr88}). In the non irreducible case, though, this set-theoretic equality is lost anyhow, as mentioned in the introduction \cite{F2013}.

For all $p\in[1,\infty)$ define the separable Banach space
\[
L^p_w(G)=\set{f\in \Lz\mid  \int_G \abs{w(x)f(x)}^p dx<+\infty}
\]
with  norm 
\[
\nor{f}_{p,w}^p= \int_G \abs{w(x)f(x)}^p dx,
\]
and the obvious modifications for $p=\infty$.
Clearly, the map $J_p:L^p_w(G)\to L^p(G)$ defined by  $J_p(f)=wf$  is
a unitary operator. The following characterization of the K\"othe
dual holds true.
\begin{lem}\label{4.1}
 Fix $p\in [1,+\infty)$ and denote by $q=\frac{p}{p-1}\in(1,+\infty]$ the dual exponent.  Then 
\[
L^p_w(G)^\#= L^{q}_{w^{-1}}(G).
\]  
For all $g\in  L^{q}_{w^{-1}}(G)$ and $f\in L^p_w(G)$, set
\[
\scal{g}{f}_{p,w}=\int_G g(x) \overline{f(x)} dx.
\]
Then the map 
$g \mapsto \scal{g}{\cdot}_{p,w}$
is an isomorphism from $L^{q}_{w^{-1}}(G)$ onto $L^p_w(G)'$. Under
this identification, the transpose  $\tr{J_p}: L^q(G) \to L^{q}_{w^{-1}}(G)$
is given by 
\[\tr{J_p} h=w h.\]
\end{lem}
 \begin{proof}
For $g\in\Lz$ we have
$g\in L^p_w(G)^\#$ if and only if $gf\in L^1(G)$ for every  $f\in L^p_w(G)$, which is equivalent to 
$(w^{-1}g)(wf)\in L^1(G)$ for every $wf\in L^p(G)$. This, in turn,  happens if and only if
$w^{-1}g\in L^{q}(G)$, which means that
$ g\in L^{q}_{w^{-1}}(G)$.
Hence $ L^{q}_{w^{-1}}(G)=L^p_w(G)^\#\subset L_w^p(G)'$,
 the pairing
$\scal{\cdot}{\cdot}_{p,w}$ is the pairing between
$L^p_w(G)^\#$ and $L^p_w(G)$ given in  Lemma~\ref{ltre}, and
\begin{align*}
  \nor{g}_{L_w^p(G)'}=\sup_{ \nor{f}_{p,w}\leq 1} \abs{\scal{g}{f}_{p,w}} =
  \sup_{ \nor{wf}_{p}\leq 1} \abs{\scal{w^{-1}g}{wf}_{p}} =
  \nor{w^{-1}g}_{q}= \nor{g}_{q,w^{-1}}.
\end{align*}
Thus  the map $g \mapsto \scal{g}{\cdot}_{p,w}$ is an isometry from 
$L^{q}_{w^{-1}}(G)$ into $L^p_w(G)'$ and it allows to identify
$L^{q}_{w^{-1}}(G)$ with a closed subspace of $L^p_w(G)'$.

We now compute the transpose of $J_p$, taking into account that
$L^p(G)'=L^q(G)$. For a fixed $h\in L^q(G)$ and all $f\in L^p_w(G)$ we have
\[
 \scal{\tr{J_p}h}{f}_{p,w}= \int_G h(x)\overline{w(x)f(x)} dx,
\]
so that, since $w$ is positive, $\tr{J_p}h= wh\in L^p_w(G)^\#= L^{q}_{w^{-1}}(G)$. 
Since $J_p$ is unitary,  so is $\tr{J_p}$ and 
$L_w^p(G)'=L^{q}_{w^{-1}}(G)$. 
  \end{proof}

\begin{lem}\label{leftreg}
    For all $p\in [1,+\infty)$, the left regular representation leaves 
    $L^p_w(G)$ invariant.   The restriction $\ell$ of $\la$ to
    $L^p_w(G)$ is a continuous representation with $\nor{\ell(x)}\leq w(x)$ for all $x\in G$.
\end{lem}
\begin{proof}
Fix $x\in G$.  By \eqref{eqb:33}, for all $f\in L^p_w(G)$
  \[ \int_G \abs{w(y)f(x^{-1}y)}^p dy =\int_G \abs{w(xy)f(y)}^p dy
  \leq w(x)^p \int_G \abs{w(y)f(y)}^p dy,\]
so that $\la(x)$ leaves  $L^p_w(G)$ invariant  and the norm of  the restriction
$\ell(x)$  is bounded by $w(x)$.  We now prove that $\ell$ is
continuous by applying the concluding remark of Section~\ref{Reps} in the appendix. For any compact subset ${\mathcal K}$ of $G$, 
since $w$ is continuous,  $w({\mathcal K})$ is bounded  and, hence, $\ell({\mathcal K})$ is
equicontinuous. Furthermore, if $f\in C_c(G)$, the map $x \mapsto \ell(x)f$ is clearly continuous from
$G$ into $L^q_w(G)$ by the dominated convergence theorem. The proof is completed by observing that  $C_c(G)$ is a dense subset of $L^p_w(G)$.
\end{proof}

Let $I=(1,+\infty)$. We define the target space as the set
\[ \T_w =\bigcap_{p\in I} L^p_w(G) \]
with the initial topology, which makes each inclusion $i_p:\T_w\hookrightarrow
L^p_w(G) $ continuous, and endow
\[ \U_w = {\operatorname{span}}\bigcup_{q\in I} L^q_{w^{-1}}(G)\]
with the final topology, which makes each inclusion
$\tilde{\iota}_q:L^q_{w^{-1}}(G)\hookrightarrow\U_w$ continuous. 

Recall that by definition of initial and of final topology, for any  topological space
  $X$, a map $A:X\to\T_w$  is continuous if for all $p\in I$ there exists a
  continuous map $A_p: X\to L^p_w(G)$ such that $i_p A = A_p$, and a map $B:\U_w\to X$ is
  continuous if all $q\in I$ there exists a continuous map
  $B_q: L^q_{w^{-1}}(G)\to X$ such that $B\tilde{\iota}_q=B_q$.
  
As for notation,  given the nature of $\T_w$, the inclusion $j:\T_w\to L^0(G)$ is   set-theoretically tautological  because  the elements of $\T_w$ are (equivalence classes of) measurable functions. However,
we keep it to emphasize that the two spaces, $\T_w$ and $L^0(G)$,  have different topologies.

The following theorem states the main properties of $\T_w$ and $\U_w$. 
\begin{thm}\label{TU}
The space $\T_w$ is a reflexive Fr\'echet space, whose topology is given
by the  fundamental family of semi-norms $\set{ \nor{\cdot}_{p,w}}_{p\in I}$. 
It is closed under complex conjugation and $f\mapsto \overline{f}$ is continuous. 
The canonical inclusion $j:\T_w\to \Lz$ is
continuous, the left regular representation $\lambda$ leaves 
$\T_w$ invariant and the restriction $\ell$  of $\la$ to $\T_w$
is a continuous representation of $G$ on $\T_w$.

The space $\U_w$ is a complete reflexive  locally
convex topological vector space.  For each $g\in \U_w$, the
anti-linear map  from $\T_w$ into $\C$ given by
\[ f\mapsto \int_G g(x)\overline{f(x)}\,dx=\scal{g}{f}_{\T_w} \]
is continuous and $g\mapsto \scal{g}{\cdot}_{\T_w} $ identifies, as
topological vector spaces,   the  dual of  $\T_w$  with
$\U_w$. Furthermore the K\"othe dual of $\T_w$ is $\U_w$, so that
\begin{equation}
  \label{eqb:2}
  \T_w'=\T_w^\#=\U_w.
\end{equation}
For each $f\in \T_w$, the anti-linear map from $\U_w$ to $\C$
\[ g\mapsto \int_G f(x)\overline{g(x)}\,dx=\scal{f}{g}_{\U_w} \]
is continuous and  $f\mapsto \scal{f}{\cdot}_{\U_w}$ identifies, as
topological vector spaces,   the  dual of  $\U_w$  with~$\T_w$. 
\end{thm}
\begin{proof}
The proof is based on the content of the article \cite{damuwe70}, whose main
results are summarized by Theorem~\ref{classic} in the appendix, 
where  $\T=\T_1$ and  $\U=\U_1$ (in  \cite{damuwe70} it is assumed that $w=1$).

By definition of initial topology, $\T_w$ is a locally convex
topological space and $\set{ \nor{\cdot}_{p,w}}_{p\in I}$ is a fundamental family of
semi-norms.  

Clearly, $\T_w$ is closed under complex conjugation and is
left invariant by  $\la$.  We show that $\ell$ is a continuos
representation. 
Given $x\in G$,  $\ell(x):\T_w\to\T_w$ is continuous
because $i_p\ell(x)=\lambda(x)i_p$.
Given $f\in \T_w$, the map $x\mapsto
\ell(x)f$ is continuous from $G$ to $\T_w$ since such are the maps $x\mapsto
i_p\ell(x)f=\ell(x)i_pf$ for all $p\in I$.  The proof that  complex conjugation is
continuous is similar.  

Define the linear map $J:\T_w \to \T$, $Jf=wf$. Since $w>0$, $J$ is a
bijection whose inverse is given by $J^{-1}g=w^{-1}g$. Both maps are
continuous by definition of initial topology since
\[ i_pJ=J_pi_p \qquad  i_p J^{-1}= J_p^{-1}i_p,\]
for all $p$.
Hence $J$ is a topological isomorphism.  By Theorem~\ref{classic}, we infer that   $\T$ is a reflexive
Fr\'echet space and, hence, $\T_w$ is a reflexive
Fr\'echet space, too.

Define $\tilde{J}:\U \to \U_w$, $\tilde{J} h=   wh$, which is clearly
bijective and whose inverse is $\tilde{J}^{-1}g=w^{-1}g$. By definition of final topology,
both are  continuos  since for all
$q\in I$
\[
\tilde{J}\,\tilde{\iota}_q=\tilde{\iota}_q \tr{J_{\frac{q}{q-1}}}\qquad \tilde{J}^{-1}\tilde{\iota}_q=\tilde{\iota}_q \,J_q^{-1}
\]
(with slight abuse, here $\tilde{\iota}_q$ denotes the inclusion of
$L^q(G)$ into $\U$). Hence
$\tilde{J}$ is an isomorphism from $\U$ onto $\U_w$. 
Therefore,  by Theorem \ref{classic},  $\U_w$ is a complete barelled locally convex topological vector space
since such is $\U$. 

Since $J$ is an isomorphism between two Fr\'echet spaces,  by
Corollary~5 of Chapter IV.4.2 of~\cite{BTVS}, $\tr{J}$ is an isomorphism from $\U$ onto $\T_w'$ explicitly given by   
\[
\scal{\tr{J}h}{f}_{\T_w}=\sum_{i=1}^n c_i \int _G h_i(x) w(x) \overline{f(x)} dx
= \int_G  (\tilde{J}h)(x)  \overline{f(x)} dx,
\]
where $h= \sum c_i h_i$ with $c_1,\ldots,c_n\in\C$ and $h_1\in
L^{q_1}(G)$, \ldots,  $h_n\in
L^{q_n}(G)$.  Hence, we can identify $\T_w'$ and $\U_w$ as 
topological vector spaces by means of  the map $\tilde{J}\tr{J}^{-1}$,
and the pairing between $\U_w$ and $\T_w$ is
\[ \scal{g}{f}_{\T_w}= \int_G g(x)\overline{f(x)}dx.\qedhere \]
\end{proof}

Observe that~\eqref{eqb:46} implies $w(x)^{-1}\leq w(x)$ for all $x\in
G$, so that $\T_w\subset \U_w=\T_w'$.  Furthermore, \eqref{eqb:45} ensures that $\check{f}\in \T_w$ if and only
if $\widecheck{wf}\in \T$.\\

We are now ready to state the main result of this section. 

\begin{thm}\label{intersections}  Take a reproducing representation $\pi$ of $G$ acting on
  the Hilbert space $\hh$ and a weight $w$ satisfying \eqref{eqb:33}, \eqref{eqb:45}   and~\eqref{eqb:46}. 
Choose an admissible vector $u\in\hh$ such that
\begin{equation}
K(\cdot)=\scal{u}{\pi(\cdot)u}_\hh \in L^p_w(G)\text{ for all } p\in I\label{eqb:1}
\end{equation}
and set
\begin{align*}
&\Ss_w =\set{v\in\hh\mid \scal{v}{\pi(\cdot)u}_\hh \in L^p_w(G)\text{ for all } p\in I},\\
&\nor{v}_{p,\Ss_w}= \left(\int_G \abs{\scal{v}{\pi(x)u}_\hh}^pw(x)^p dx\right)^{\frac{1}{p}}. 
 \end{align*}
 Then:
  \begin{enumerate}[a)]
   \item the space $\Ss_w$ is a reflexive Fr\'echet space with respect
     to the topology induced by the family of semi-norms
     $\set{\nor{v}_{p,\Ss_w}}_{p\in I}$, the canonical inclusion
     $i:\Ss_w\to\hh$ is continuous and  with dense range;
\item the representation $\pi$ leaves  $\Ss_w$ invariant,  its restriction $\tau$ is a continuous
    representation acting on $\Ss_w$, and
\[ i(\tau(x)v)=\pi(x)i(v) \qquad x\in G,\,v\in\Ss_w;\]
\item if  $\hh$ and $\Ss'_w$ are endowed with the weak topology, the transpose $\tr{i}:\hh\to \Ss_w'$ is continuous, injective, 
  with dense range and satisfies the intertwining
\[ \tr{\tau(x)}\tr{i}(v)=\tr{i}(\pi(x)v) \qquad x\in G,\,v\in \hh;\]
\item the restricted voice transform $V_0:\Ss_w\to \T_w$, given by
\[
V_0v(x)=\scal{i(v)}{\pi(x)u}_\hh\qquad x\in G,\, v\in\Ss_w,
\] 
is an injective  strict morphism onto the closed subspace
\[ \mathcal M^{\T_w}=\set{f\in \T_w  \mid j(f)\con K= j(f)},\]
and it  intertwines $\tau$ and $\ell$; 
\item\label{bar} every $f\in \T_w$, has at $u$ a Fourier transform  in $\Ss_w$  and
\[ 
j(V_0\pi(f)u)= j(f)\con K;
\]
furthermore, the map
\[ \T_w \ni f\mapsto \pi(f)u\in \Ss_w \]
is continuous and its restriction to $\mathcal M^{\T_w}$
is the inverse of $V_0$;
\item\label{sopra} every $\Phi\in\U_w$ has  at $u$ a Fourier
  transform in $\Ss_w'$  and
\[
 V_e \pi(\Phi)u= \Phi\con K;
 \]
\item the extended voice transform given by \eqref{EVT2} takes values in  $\U_w$,   it is injective,   continuous
  (when both spaces are endowed with the strong topology) and
  intertwines  $\tr{\tau}$ and $\la$; the range
  of $V_e$ is the closed subspace 
\begin{equation}
\mathcal M^{\U_w}=\set{\Phi\in \U_w \mid \Phi\con K= \Phi}
  =  \operatorname{span}\bigcup_{p\in I}\mathcal
  M^{L^p_w(G)} \subset L^\infty_{w^{-1}}(G)
  \label{eq:38}\end{equation}
and for all $T\in\Ss_w'$ and $v\in\Ss_w$
\begin{equation}
\scal{T}{v}_{\Ss_w}= \scal{V_eT}{V_0v}_{\T_w};\label{eqb:52}
\end{equation}
\item\label{sotto} the map 
\[  
\mathcal M^{\U_w}\ni \Phi\mapsto \pi(\Phi)u\in \Ss_w'
\]
is the inverse of $V_e$   and coincides with the restriction of the map
  $\tr{\,V_0}$ to $\mathcal M^{\U_w}$, namely 
  \begin{equation}
    \label{eq:39}
   V_e( \tr{\,V_0}\Phi)= V_e\pi(\Phi)u= \Phi\qquad \Phi\in \mathcal M^{\U_w}.
  \end{equation}
\item  $\displaystyle{\tr{i}i(\Ss_w)=\set{T\in\Ss_w'\mid V_eT\in
      \T_w}=\set{\pi(f)u \mid f\in \mathcal M^{\T_w}}}$.
  \end{enumerate}
\end{thm}

The fundamental requirement~\eqref{eqb:1}  states that $K\in \T_w\subset \T_w'=\T_w^\#=\U_w$, 
whereas~\eqref{eqb:52} is the reconstruction formula for the distributions in
$\Ss_w'$, namely the fact that for any $T\in\Ss_w'$ the formula
\begin{equation}
\scal{T}{v}_{\Ss_w} =\int_G
\scal{T}{\tau(x)u}_{\Ss_w} \scal{\pi(x)u}{i(v)}_\hh\,dx \label{eqb:44}
\end{equation}
holds for all $v\in\Ss_w$, where the integral converges since $\scal{\pi(\cdot)u}{i(v)}_\hh$ is
in $\T_w$ by definition of $\Ss_w$, and $\scal{T}{\tau(\cdot)u}_{\Ss_w} $ is
in $\U_w=\T_w^\#$ since the range of $V_e$ is contained in $\U_w$ (for
arbitrary target spaces this last property could fail).

\begin{proof}
Since $K\in j(\T_w)\subset \U_w=\T_w^\#$, Assumption~\ref{H1} is
satisfied. Furthermore, since $j(\T_w)\subset L^2_w(G)\subset L^2(G)$ by \eqref{eqb:46},
also   Assumption~\ref{H2} is satisfied. The
topology induced by the family of semi-norms
$\set{\nor{v}_{p,\Ss_w}}_{p\in I}$ is the initial topology on $\Ss_w$
induced by the map $V_0$, as in the proof of Corollary~\ref{topology}.
  \begin{enumerate}[a)]
  \item By  Theorem~\ref{tuno} which does not depend on Assumption~\ref{H3},
    the space $\Ss_w$ is a  Fr\'echet  space 
    isomorphic to  $\mathcal M^{\T_w}$ and the canonical inclusion
     $i:\Ss_w\to\hh$ is continuous and  with dense range.
     Furthermore, since $\mathcal M^{\T_w}$ is a closed subspace of a
    reflexive Fr\'echet  space,  both $\mathcal M^{\T_w}$  and $\Ss_w$
    are reflexive.
\item Apply Theorem~\ref{tuno}.
\item  Apply Theorem~\ref{tuno} for the main statement; the intertwining property is easily checked.
\item Apply Theorem~\ref{lsei}.
\item Fix $f\in \T_w$. By~\eqref{eqb:46} $j(f)\in L^2(G)$ and by Proposition~\ref{tre} there
  exists $\pi(f)u\in\hh$ such that $V_2\pi(f)u=j(f)\con K$.
We claim that $j(f)$ and $K$ are convolvable (see   the appendix for the definition) and $j(f)\con K \in j(\T_w)$. It is enough to show that for all $r\in I$, $\abs{wf}\con
\abs{wK}\in L^r(G)$. Indeed,  given $x\in G$
\begin{align*}
w(x)\int_G \abs{f(y)}\abs{K(y^{-1}x)}\, dy & =  
 \int_G\abs{w(y)f(y)}\abs{w(y^{-1}x)K(y^{-1}x)}\frac{w(x)}{w(y)w(y^{-1}x)}dy\\
& \leq \abs{wj(f)}\con \abs{wK}(x),
\end{align*}
by~\eqref{eqb:33}. Define $p=q= \frac{2r}{1+r}>1$, so that
$\frac{1}{p}+\frac{1}{p}=\frac{1}{r}+1$. Then by assumption $wj(f)\in L^p(G)$, 
$wK \in L^q(G)$ and $\widecheck{wK}=w\check{K}\in L^q(G)$. 
Hence~\eqref{eq:5} applies, showing  that  $\abs{wj(f)}$ and  $\abs{wK}$ are
convolvable,  $\abs{wj(f)}\con \abs{wK}\in L^r(G)$ and 
$ \nor{\abs{wj(f)}\con \abs{wK} }_r\leq C \nor{wf}_p$, where $C$ is a
constant depending on $q$ and $K$.  Hence $j(f)\con K\in
L^r_w(G)$ and 
\begin{equation}
\nor{j(f)\con K }_{r,w}\leq C \nor{j(f)}_{p,w}.\label{eq:24}
\end{equation}
Therefore $j(f)\con K\in j( \T_w)$.
By definition of $\Ss_w$,  $\pi(f)u\in\Ss_w$ and $j(V_0\pi(f)u)=j(f)\con K$. 
The map $f\mapsto  \pi(f)u$ is continuous by~\eqref{eq:24}.
By Theorem~\ref{lsei}, the map $\mathcal M^{\T_w}\ni f\mapsto \pi(f)u\in \Ss_w$
is the inverse of $V_0$. 
\item Fix $\Phi\in \U_w$. By linearity we can assume that $\Phi$ is in some $L^p_{w^{-1}}(G)$, so that $\Phi j(f)$ is in $L^1(G)$ for all $f\in \T_w$.  In particular  for all $v\in\Ss_w$ we have $\Phi j(V_0v)\in L^1(G)$ because $V_0v \in \T_w$, and by Proposition~\ref{due} there exists
$\pi(\Phi)u\in \Ss_w'$.
Furthermore, recalling that $\tr{\,V_0}$ is a linear map from $\T_w'$ onto
$\Ss_w'$,  for all $v\in\Ss_w$ we have
\[ \scal{\tr{V_0}\Phi}{v}_{\Ss_w}  = \scal{\Phi}{V_0 v}_{\T_w}=\int_G \Phi(x)
\overline{\scal{i(v)}{\pi(x)u}_\hh}dx= \int_G \Phi(x)
\scal{\pi(x)u}{i(v)}_\hh dx.\]
Comparing this equation with~the definition of $\pi(\Phi)u$  we get
$\tr{\,V_0}\Phi= \pi(\Phi)u$ and, by~\eqref{eq:22},  $V_e \pi(\Phi)u = \Phi \con \vu$.
\item Fix  $T\in \Ss_w'$. Since $\,\tr{\,V_0}$ is surjective, there exists $\Phi\in\U_w=\T_w'$ such that
$T=\,\tr{\,V_0}\Phi=\pi(\Phi)u$, so that  $V_eT= \Phi \con \vu$.  
To show that $V_eT\in\U_w$ we prove that 
$\Phi$ and $K$ are convolvable whenever $\Phi\in \U_w$, their convolution $\Phi\con K$ is in $\U_w$ and the map
$\Phi\mapsto \Phi\con K$ is continuous from $\U_w$ into $\U_w$.
By definition of $\U_w$,  it is
enough to show that, given $p\in I$, for all
$\Phi \in L^p_{w^{-1}}(G)$, $\Phi$ and $K$ are convolvable,  $\Phi\con K\in L^{2p}_{w^{-1}}(G)$ and the
map $\Phi\mapsto \Phi\con K$ is continuous from $L^p_{w^{-1}}(G)$ into $L^{2p}_{w^{-1}}(G)$.
As above,
\begin{align*}
  w(x)^{-1} \int_G \abs{\Phi(y)} \abs{K(y^{-1}x)}\, dy & = w(x)^{-1}
  \int_G \abs{\Phi(xy)} \abs{K(y^{-1})}\, dy \\
& = \int_G \abs{w^{-1}(xy)\Phi(xy)}
\abs{K(y^{-1})}\frac{w(xy)}{w(x)}\, dy \\
& \leq \int_G \abs{w^{-1}(xy)\Phi(xy)}
\abs{w(y^{-1})K(y^{-1})} \, dy \\
& = \abs{w^{-1}\Phi}\con \abs{wK}(x),
\end{align*}
where in the third line we used both~\eqref{eqb:33}
and~\eqref{eqb:45}. By assumption $w^{-1}\Phi\in
L^p(G)$, $wK \in L^q(G)$ where we set $q=\frac{2p}{2p-1}>1$,  and
$\widecheck{wK}=w\check{K}\in L^q(G)$. Since 
$\frac{1}{p}+\frac{1}{q}=\frac{1}{2p}+1$, ~\eqref{eq:5} gives that  $\abs{w^{-1}\Phi}$ and $\abs{wK}$ are
convolvable,  $\abs{w^{-1}\Phi}\con \abs{wK}\in L^{2p}(G)$ and 
\[
\nor{\abs{w^{-1}\Phi}\con \abs{wK} }_{2p}\leq C \nor{w^{-1}\Phi}_p
\]
where $C$ is a
constant depending on $p$ and $K$.  Hence $\Phi\con K\in
L^{2p}_{w^{-1}}(G)$ and 
\[
\nor{\Phi\con K }_{{2p},w^{-1}}\leq C \nor{\Phi}_{p,w^{-1}}.
\]
This proves the claim. Note that, since $\Phi\mapsto \Phi\con K$ is
continuous,  $\mathcal M^{\U_w}$ is closed. We observe {\it en passant} that $\U_w$ is not
  a Fr\'echet space,  so that Proposition~\ref{uno} does not apply.

We now prove that $V_eT=V_eT\con \vu$. Since $\abs{\vu}\in \T_w$,
reasoning as in the proof of item~\ref{bar}), $\abs{\vu}\con\abs{\vu}\in\T_w$.
Furthermore,  if  $g=\abs{\vu}\con\abs{\vu}$, so that $\check{g}=g$,  then as above $\Phi$ and $g$
are convolvable and  \eqref{eq:35} gives that
\[ V_eT\con \vu= (\Phi\con \vu)\con \vu=   \Phi\con (\vu\con \vu)=
\Phi\con \vu=V_eT,\]
so that the range of $V_e$ is contained in $\mathcal M^{{\U_w}}$.
From~\ref{sopra}) we know  that  
$\pi(\Phi)u\in\Ss_w'$ whenever $\Phi\in \mathcal M^{{\U_w}}$ and that $V_e \pi(\Phi)u=\Phi\con K=\Phi$, showing that
$V_e$ is onto $\mathcal M^{{\U_w}}$ and  that the map 
\[  \mathcal M^{{\U_w}}\ni \Phi\mapsto \pi(\Phi)u\in \Ss_w'\]
is the inverse of $V_e$, as claimed in item~\ref{sotto}).

Next, we prove \eqref{eqb:52}. Fix $v\in \Ss_w$ and define the map $\Psi:G\to\Ss_w$ by
\[ \Psi(x)=\scal{\pi(x)u}{i(v)}_\hh \tau(x)u =
\overline{V_0v(x)}\tau(x)u.\] 
For all  $T\in \Ss_w'$, $V_eT\in {\U_w}$ and $V_0v\in \T_w$, so that
$x\mapsto \scal{T}{\Psi(x)}_{\Ss_w} $ is in $L^1(G)$.   

Since  $\Ss_w$  is a reflexive
Fr\'echet  space, we can regard  $\Ss_w$ as the dual of $\Ss_w'$, which has the property
  (GDF) by Proposition~3 Chapter 6. Appendix No.2 of \cite{BINT1}. The
fact that the map $x\mapsto \scal{T}{\Psi(x)}_{\Ss_w} $ is in $L^1(G)$, means that $\Psi$ is scalarly integrable.
Theorem~\ref{Gelfand-Dunford}
shows that its (scalar) integral is in $\Ss_w$, {\em i.e.}
 that  there exists $\psi\in\Ss_w$ such that 
\[ \scal{T}{\psi}_{\Ss_w}  = \int_G \scal{T}{\tau(x)u}_{\Ss_w} 
\scal{\pi(x)u}{v}_\hh dx.\] 
With the choice $T=\tr{i}(z)$, $z\in \hh$, \eqref{eq:2} gives that
$\scal{z}{i(\psi)}_\hh=\scal{z}{i(v)}_\hh$. Since this last equality
holds true for all $z\in \hh$ and $i$ is injective, then $\psi=v$ and this proves \eqref{eqb:52}.
Furthermore, the reproducing formula~\eqref{eqb:52}  implies that
$V_e$ is injective.

Finally, we prove that $V_e$ is continuous. Fix a bounded subset $B$
in $\T_w$. By e)  the map $f \mapsto \pi(f)u$ is continuous from
$\T_w$ into $\Ss_w$. Then
$B'=\pi(B)u$ is a bounded subset of $\Ss_w$. Furthermore, given $T\in \Ss_w'$
\[
\sup_{f\in B}\abs{ \scal{V_eT}{f}_{\T_w}} = \sup_{f\in B}\abs{\int_G
\scal{T}{\tau(x)u}_{\Ss_w} \overline{f(x)} dx} = \sup_{f\in
B}\abs{\scal{T}{\pi(f)u}_{\Ss_w} }= \sup_{v\in
B'}\abs{\scal{T}{v}_{\Ss_w} }.
\]
Since $B'$ is a 
bounded subset of $\Ss_w$, the map $T\mapsto \sup_{v\in
B'}\abs{\scal{T}{v}_{\Ss_w}}$ is continuous, hence  $V_e$ is
such.  The rightmost equality~\eqref{eq:38} is a consequence of the definition
  of ${\U_w}$ and the inclusion follows from the fact that
  $V_e\Ss'\subset L^\infty_{w^{-1}}(G)$. Indeed, for all $x\in G$ 
\[
\abs{V_eT(x)} 
=\abs{\scal{T}{\tau(x)u}}_\Ss\leq C_T \max_{i\leq n}\nor{\tau(x)u}_{p_i,\Ss}
\leq C \max_{i\leq n}\nor{\ell(x)K}_{p_i} \leq C  \max_{i\leq n}\nor{K}_{p_i}\, w(x)
\]
where $C$ is a constant depending on $T$,  $p_1,\ldots,p_n$
are suitable numbers in $I$ also depending on $T$, and the last bound
is a consequence of Lemma~\ref{leftreg}. 

\item See the proof of the above item.
\item Apply d) of Proposition~\ref{cuno} with $\F=\T_w$ and $\E=\Ss_w$, taking into account~\ref{bar}). \qedhere
  \end{enumerate}
\end{proof}

We summarize the findings in this section in the following theorem
which is one of the main results of this paper since it shows that our
analysis is indeed applicable.
\begin{thm}\label{stephan}
If $K\in\T_w$, then   Assumptions $\ref{H1}\div\ref{H4}$ are satisfied for $\T_w$.
\end{thm}
\begin{proof}
Under the hypothesis~\eqref{eqb:1}, Assumption~\ref{H1} is 
satisfied, and  $j(\T_w)\subset L^2_w(G)\subset L^2(G)$ implies
 Assumption~\ref{H2}. The reconstruction formula~\eqref{eqb:44}
clarifies that $u$ is a cyclic vector for $\tau$, which is equivalent to
 Assumption~\ref{H3} since $V_0\tau(x)u=\ell(x)K$ and $V_0$ is
an injective strict morphism from $\Ss_w$ onto $\mathcal M^{\T_w}$.
Finally,~\eqref{eqb:44}  with $v=\tau(x)u$ implies that 
Assumption~\ref{H4} holds true.
\end{proof}

Observe that Theorem~\ref{stephan} paves the way for a coorbit space theory
with a specific choice of target space, namely $\T_w$. 
Indeed, if $Y$ is a Banach space  continuously
embedded in $\Lz$ and $\la$-invariant, and we assume that $K\in Y^{\#}$
and that  $\mathcal Mxs^Y$ is
a subspace of $\U_w$, then
Assumptions $\ref{H5}$ and $\ref{H6}$ are satisfied.
Hence Proposition~\ref{prop6} holds true, giving rise to a coorbit theory for $Y$.

We summarize the general scheme in the following picture.
\begin{figure}[H]
\begin{center}
\includegraphics[scale=.55]{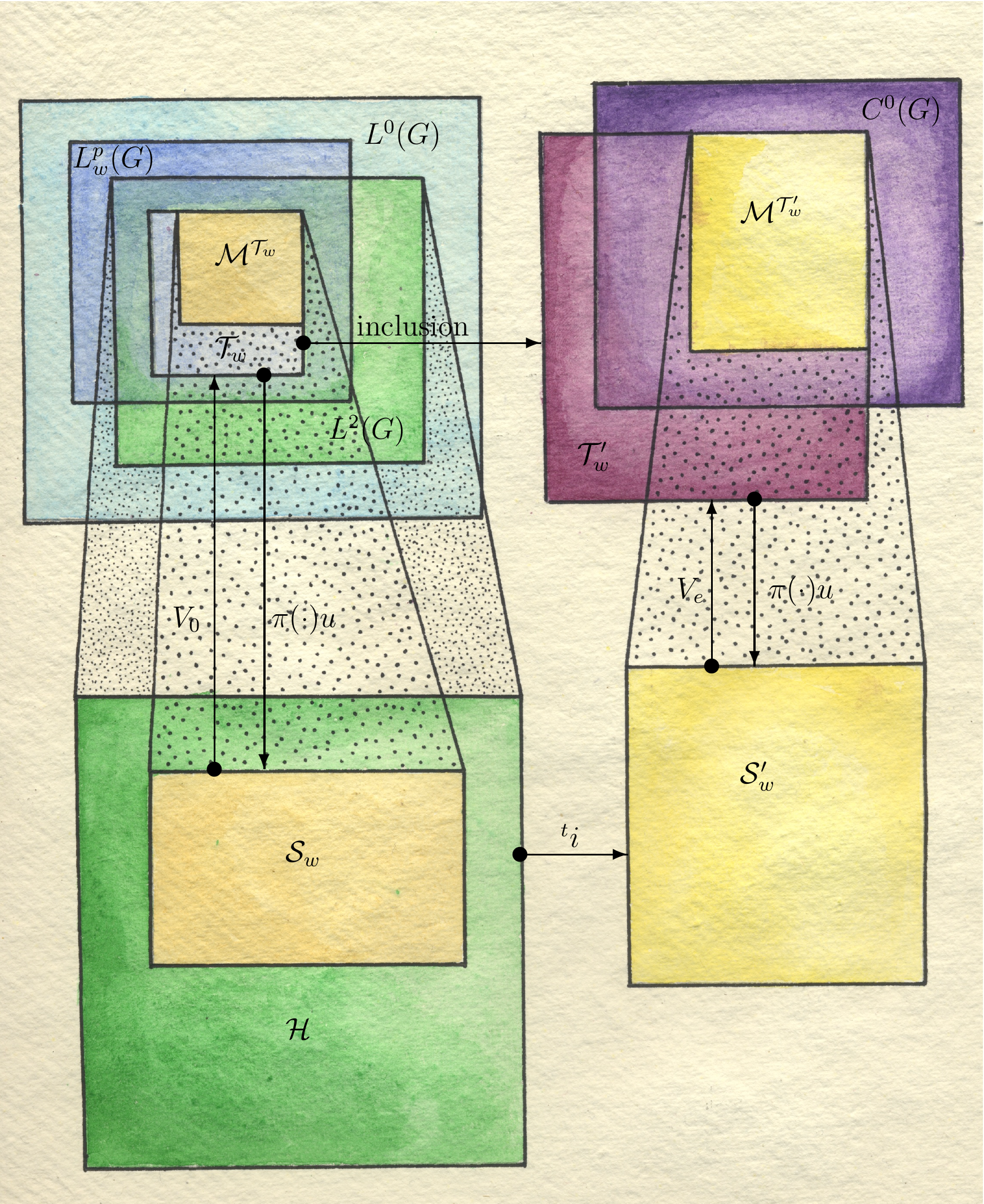}
\caption{Objects on the group on top row, signals on bottom row.}
\label{default}
\end{center}
\end{figure}

\subsection{Band-limited functions}\label{BAND} 
As a toy example, we consider the space of band-limited signals on the real line.
 Although  elementary, and certainly very natural, this case can not be handled by the classical coorbit machinery. 
 This is somewhat unsatisfactory, because the sinc function is one of the first examples of reproducing kernels which comes to mind.
Our theory does handle it, and the natural coorbit spaces  that arise  are the Paley--Wiener $p$-spaces.

In this section, $G$ is the additive group $\R$ and the Haar
  measure  is the Lebegue measure $dx$.  We denote by 
$\mathbb S(\R)$ the Fr\'echet space of rapidly
  decreasing functions and by $\mathbb S(\R)'$ the 
  space of tempered distributions. The Fourier transform on $\mathbb
  S(\R)$ and $\mathbb S(\R)'$ is denoted by $\mathcal F$. Regarding
  $L^2(\R)$ as a subspace of $\mathbb S(\R)'$,  we set $\widehat{v}=\mathcal F v$  
  for any $v\in L^2(\R)$. 

The representation $\pi$ is the regular representation restricted to 
 the Paley--Wiener space of functions with band in the fixed compact 
interval $ \Omega \subset \R $, namely
\[
\hh=B^2_\Omega = \{ v \in L^2(\R) : \supp{\widehat{v}} \subseteq \Omega \} . 
 \]
 Strictly speaking, the elements of $B^2_\Omega$ are not functions, but equivalence classes of functions.
 However, in view of the Paley--Wiener--Schwartz theorem \cite{hor90},
 each equivalence class in $B^2_\Omega$ has a unique representative which is continuous (in fact smooth).
 We are thus allowed to identify each class with its smooth representative, and  we shall do so.
 
 Since the group $\R $ acts on $B^2_\Omega$ by translations:
\[
 \pi(b)v(x) = v(x-b), \qquad v \in B^2_\Omega ,
 \]
on the frequency side,   $ \widehat{\pi} = \ff\pi\ff^{-1} $   acts on
$ \ff\hh = L^2(\Omega)$ by modulations: 
\[
\widehat{\pi}(b)\widehat{v}(\xi) = \e^{-2\pi i b\xi} \widehat{v}(\xi), \qquad v \in{B^2_\Omega}.
\]
This representation is not irreducible: any  subset $ \Xi \subseteq \Omega $ gives a subrepresentation on $ B^2_\Xi $.
 For the reader's convenience, we summarize in the next proposition the main facts that are relevant 
 to our discussion.
 \begin{prop} \label{prop:band}
The representation $\pi$ is reproducing and the following facts
  hold true.
  \begin{enumerate}[a)]
   \item \label{item:band-wavelet}
    A vector $ u \in B^2_\Omega $ is admissible if and only if $ \abs{\widehat{u}} = 1 $ almost everywhere on $\Omega$.
    In this case, the kernel $K$ is
    $$ K = \scal{u}{\pi(\cdot)u}_\hh=\ff^{-1}\chi_\Omega , $$
    where $\chi_\Omega$ is the characteristic function on $\Omega$.
   \item \label{item:band-voice}
    Let $u$ be an admissible vector.
    Then $u(x) = K(x)$ for every $x\in\R$ 
if and only if the corresponding voice transform $V_2$ is the inclusion
    $$ {V_2} : B^2_\Omega \hookrightarrow L^2(\R) . $$
   \item
    If $ \Omega = [-\omega,\omega] $ is a symmetric interval, then the kernel is the sinc function
    $$ K(b) = 2\omega\sinc(2\omega\pi b) , $$
    where $ \sinc x = \sin x/x $.
  \end{enumerate}
 \end{prop}
 \begin{proof}
The fact that $\pi$ is reproducing follows from item a).
  \begin{enumerate}[a)]
   \item
    Applying the Plancherel identity, we can compute the voice transform as
    \begin{equation} \label{eq:band-voice}
     {V_2} v(b) = \scal{v}{\pi(b)u} = \scal{\widehat{v}}{\widehat{\pi}(b)\widehat{u}}
     = \int_{\widehat\R} \widehat{v}(\xi)\overline{\widehat{u}}(\xi)\e^{2\pi i b\xi} d\xi = \ff^{-1}(\widehat{v}\overline{\widehat{u}})(b) ,
    \end{equation}
    whose squared norm, again by Plancherel, is
    $$ \nor{{V_2} v}^2 = \int_{\widehat\R} \abs{\widehat{v}(\xi)\overline{\widehat{u}}(\xi)}^2 d\xi . $$
    On the other hand, Plancherel also entails
    $$ \nor{v}^2 = \int_{\R} \abs{v(x)}^2 dx = \int_{\widehat\R} \abs{\widehat{v}(\xi)}^2 d\xi . $$
    Therefore, $u$ is admissible if and only if $
    \abs{\widehat{u}(\xi)} = 1 $ for almost every $ \xi \in
    \supp{\widehat{v}} \subseteq \Omega $.   Since $\Omega$ is compact,
    vectors $u\in B^2_\Omega$ satisfying the above
    condition  clearly exist and, hence,  $\pi$ is reproducing.
 If $u$ is admissible, using \eqref{eq:band-voice} we obtain $ K ={V_2} u = \ff^{-1}(\abs{\widehat{u}}^2) = \ff^{-1}(\chi_\Omega) $.
    
   \item
    Suppose that $ u = K $.
    Then, in view of item \eqref{item:band-wavelet} we have $
    \widehat{u} = \chi_\Omega $, and  equality~\eqref{eq:band-voice} gives
    ${V_2} v = \ff^{-1}(\widehat{v}) = v $ for every $ v \in B^2_\Omega $,
so that  ${V_2}$ is the natural inclusion. Conversely, if this is the case, then $ K = {V_2} u = u $.
   \item If  $\Omega = [-\omega,\omega]$, from \eqref{eq:band-voice} it follows
    \begin{equation*}
     K(b) = {V_2} u(b) = \int_{\widehat\R} \abs{\widehat{u}(\xi)}^2 \e^{2\pi ib\xi} d\xi = \int_{-\omega}^\omega \e^{2\pi ib\xi} d\xi
     = 2\omega\sinc(2\omega\pi b) . \qedhere
    \end{equation*}
  \end{enumerate}
 \end{proof}

 From now on we set $\Omega= [-\omega,\omega] $ and  
\[u = K = \ff^{-1}\chi_\Omega= 2\omega\sinc(2\omega\pi\cdot).\]
Clearly, $K$ is not in $L^1(\R)$, but it belongs to $L^p(\R)$ for every $ p > 1 $.
We  thus choose the weight $w=1$ and take
 $$ \T = \bigcap_{{p\in I}} L^p(\R) $$
 as target space to construct coorbits (recall that $I=(1,+\infty)$).

 For $ p \in [1,+\infty) $, we define the Paley--Wiener $p$-spaces
\[
B^p_\Omega = \{ f \in L^p(\R) : \supp{\ff f}\subseteq \Omega \}.
\]
 Recall that the Fourier transform maps $L^p$ to $L^{p'}$ for $ p \leqslant 2 $ ,
 while for $ p > 2 $ we get distributions that in general are not  functions  \cite{hor90}.
 
 The spaces $B^p_\Omega$ are usually defined in the literature as
 the spaces of the entire functions of fixed exponential type whose restriction to the real line is $p$-integrable \cite{plpo37}.
 This definition is equivalent  to ours since a Paley--Wiener theorem holds for all $ p \in [1,+\infty) $.
 In particular, all these functions  are indefinitely differentiable on $\R$.
 Moreover, if $ f \in B^p_\Omega $ with $ p < +\infty $, then $ f(x) \rightarrow 0 $ as $ x \rightarrow \pm\infty $, hence
\[
B^p_\Omega \subset C^\infty_0(\R), \qquad 1 \leqslant p < +\infty.
\]
Consequently, the Paley--Wiener spaces are nested and increase with $p$:
 $$ B^p_\Omega \subseteq B^q_\Omega \qquad 1 \leqslant p \leqslant q <+\infty . $$
We are going to identify our coorbit spaces as Paley--Wiener spaces.
To show this, we shall repeatedly make use of the following fact.
 \begin{lem} \label{lemma:band}
There exists
 a  family of functions $\{\widehat{g_\varepsilon}\} \subset C^\infty_c(\widehat\R) $ satisfying
  \begin{enumerate}[i)]
   \item $ \lim\limits_{\varepsilon \rightarrow 0} \widehat{g_\varepsilon} = \chi_\Omega $ in $L^q(\widehat\R)$ for every $ q \geqslant 1 $;
   \item $ \chi_{[-\omega+\varepsilon,\omega-\varepsilon]} \leqslant \widehat{g_\varepsilon} \leqslant \chi_{[-\omega-\varepsilon,\omega+\varepsilon]} $;
   \item $ \nor{ \partial \widehat{g_\varepsilon}}_\infty \lesssim \varepsilon^{-1} $.
  \end{enumerate}
such that for all  $ f \in L^p(\R) $, with $p \geqslant 1 $,  in $\mathbb S'(\widehat\R)$
  \begin{equation}
\ff(f \con K) = \lim_{\varepsilon \rightarrow 0} \ff(f)
  \widehat{g_\varepsilon}.\label{eq:30} 
\end{equation}
 \end{lem}
Take $f\in L^p(\R)$.  By Young's inequality~\eqref{eq:5} we know that  $f \con K\in
  L^r$ for some
  $r> 1$, so that $f\con K\in\mathbb S'(\R)$ and the left hand side is the Fourier transform of a tempered distribution. Similarly, 
  $\ff(f)  \widehat{g_\varepsilon}\in\mathbb S'(\R)$ because  $\widehat{g_\varepsilon}\in C^\infty_c(\R)$, and on the right hand side we also have Fourier transforms of  tempered distributions.

 \begin{proof}
 Take
\[
 \widehat{h} \in C^\infty_c(\R), 
 \quad \supp{\widehat{h}} \subseteq [-1,1],
\quad \widehat{h} \geqslant 0,
 \quad \int_{\widehat\R} \widehat{h}(\xi) d\xi = 1,
  \]
  and then consider the corresponding  approximate identity
    $\{\widehat{h}_\varepsilon\}$ defined by  the dilations of $\hat{h}$
  \[
   \widehat{h}_\varepsilon(\xi) = \varepsilon^{-1}\widehat{h}(\xi/\varepsilon),\quad \varepsilon>0,
   \]
so that $ \widehat{h}_\varepsilon \in C_c^\infty(\widehat\R)$, and 
define $ \widehat{g_\varepsilon} = \widehat{h}_\varepsilon \con \chi_\Omega$.
Since both factors are $L^1$-functions, the convolution theorem gives that
 $g_\varepsilon= h_\varepsilon K$.
A classical result, see Corollary 3.4 of \cite{la93}, shows that 
$\widehat{g_\varepsilon} \rightarrow \chi_\Omega $ in $L^q(\R)$ for every $ q
\geqslant 1 $ and  $ \widehat{g_\varepsilon} \in C_c^\infty(\R) $.  Moreover,
  $$ \supp{\widehat{g_\varepsilon}} \subseteq \supp{\widehat{h}_\varepsilon} + \supp{\chi_\Omega} \subseteq [-\varepsilon,\varepsilon] + \Omega
  = [-\omega-\varepsilon,\omega+\varepsilon] . $$
  Expanding the convolution, we get
  $$ \widehat{g_\varepsilon}(\xi) = \varepsilon^{-1} \int_\Omega \widehat{h}((\xi-\tau)/\varepsilon) d\tau
  = \int_{(\Omega+\xi)/\varepsilon} \widehat{h}(\tau) d\tau $$
  by a change of variable.
  Notice here that if $ \abs{\xi} \leqslant \omega-\varepsilon $, then
  $ \abs{\xi + \varepsilon\tau} \leqslant \omega $ whenever $
  \abs{\tau} \leqslant 1 $, which means that
 $ (\Omega+\xi)/\varepsilon \supseteq [-1,1] \supseteq \supp{\widehat{h}} $. It follows that
  $$ \widehat{g_\varepsilon}(\xi) = \int_{\widehat\R} \widehat{h}(\tau) d\tau = 1 $$
  for every $ \xi \in [-\omega+\varepsilon,\omega-\varepsilon] $.
  The derivative of $\widehat{g_\varepsilon}$ is $ \partial(\widehat{h}_\varepsilon \con \chi_\Omega) = \partial\widehat{h}_\varepsilon \con \chi_\Omega $,
  and $ \partial\widehat{h}_\varepsilon(\xi) = \varepsilon^{-2}\partial\widehat{h}(\xi/\varepsilon) $.
  A change of variable then yields
  $$ \abs{\partial \widehat{g_\varepsilon}(\xi)} \leqslant \varepsilon^{-1} \int_{[-1,1]} \abs{\partial\widehat{h}(\tau)} d\tau
   \leqslant 2\sup(\abs{\partial\widehat{h}}) \varepsilon^{-1} . $$
   Finally, let $ h_\varepsilon =
   \ff^{-1}\widehat{h}_\varepsilon$. By dominated convergence, for all $x\in\R$
\[\lim_{\varepsilon
     \rightarrow 0} h_\varepsilon(x)=\int_\R \hat{h}(\xi)\,d\xi =1,\]
 so that
  $$ \lim_{\varepsilon \rightarrow 0} g_\varepsilon(x) 
  =\lim_{\varepsilon \rightarrow 0} h_\varepsilon(x)K(x)  = K(x) , $$
  and at the same time
  $$ \abs{h_\varepsilon(x)K(x)}^q \leqslant \nor{h}_\infty^q \abs{K(x)}^q $$
  for any $ q > 1 $. Therefore $ h_\varepsilon K \rightarrow K $ in $L^q(\R)$, by dominated convergence.
Young's inequality  implies now that $ f \con h_\varepsilon K \rightarrow f \con K $ in some $L^r(\R)$, hence as tempered distributions.
Therefore
  $$ \ff(f \con K) = \lim_{\varepsilon \rightarrow 0} \ff(f \con g_\varepsilon)
  = \lim_{\varepsilon \rightarrow 0} \ff(f)\ff(g_\varepsilon)
  = \lim_{\varepsilon \rightarrow 0} \ff(f)\widehat{g_\varepsilon} $$
  by the continuity of the Fourier transform and an application of the convolution theorem,
  because $f\in{\mathbb S}'(\R)$ and $g_\varepsilon\in{\mathbb S}(\R)$ (see Theorem~XV, Ch.~VII 
  in~\cite{sch66}).
 \end{proof}
 
 We are now ready to state the characterization of  the natural coorbit spaces relative to  band limited functions.
 \begin{prop}
  Let $ \Omega = [-\omega,\omega] $ and take $u = K = \ff^{-1}\chi_\Omega $.
  The space of test functions is
  $$ \Ss = \bigcap_{{p\in I}} B^p_\Omega $$
  and the space of distributions is
\[
\Ss' = \bigcup_{{p\in I}} B^p_\Omega. 
\]
The extended voice transform is the inclusion
  $$ V_e : \Ss' \hookrightarrow \U $$
  and the coorbits of the $L^p$ spaces are
\[
\Co{L^p(\R)} = \mathcal{M}^p = B^p_\Omega . 
\]
 \end{prop}
 \begin{proof}
  Since $V_2$ is the inclusion by item \ref{item:band-voice}) of Proposition \ref{prop:band}, we have
  $$ \Ss = \{ v \in B^2_\Omega : v \in L^p(\R) \ \forall p > 1 \} = \bigcap_{{p\in I}} B^p_\Omega . $$
In order to describe  $\Ss'$,  we first observe that 
\[ 
\bigcup_{{p\in I}}\mathcal{M}^p=\operatorname{span}
\bigcup_{{p\in I}}\mathcal{M}^p=\mathcal M^{\U}.
\]
Indeed, since $w=1$, by item~\eqref{eq:38}  $\mathcal{M}^p\subset L^\infty(G)$, so
that $\mathcal{M}^p\subset \mathcal{M}^q$ whenever $p\leq q$. 
Furthermore, item~\ref{sotto}) of Theorem~\ref{intersections} implies
that $V_e=\tr{\,V_0}$ establishes a linear isomorphism
  from $ \mathcal{M}^{\U} = \bigcup_{{p\in I}}\mathcal{M}^p $ to $\Ss'$.
  In particular, $\Ss'$ is the range $\tr{\,V_0}\mathcal{M}^{\U}$.
 
Since  $ V_0 : \Ss \hookrightarrow \T $ is the inclusion, 
the transpose map $ \tr{\,V_0} : \T' \to \Ss' $ is simply the restriction on the subspace $\Ss$.
  Therefore, we can explicitly represent $\Ss'$ as the space of anti-linear functionals
\[
 v \in \Ss \longmapsto \int_{\R} \Phi(x) \overline{v(x)} dx,
  \qquad \Phi \in \bigcup_{{p\in I}}\mathcal{M}^p . 
  \]
  The two spaces are thus canonically identified and, with this identification,
  $V_e$ is the inclusion $ \Ss' \hookrightarrow \U $.
    
  We next prove that $ \mathcal{M}^p = B^p_\Omega $ for all $p\in I$.
  Let $ f \in \mathcal{M}^p $ and  $\varphi$ be a smooth function with compact support  contained in 
  $\Omega^c $.
  Then, by Lemma \ref{lemma:band}, we have
  $$ \langle \widehat{f} , \varphi \rangle
  = \lim_{\varepsilon \rightarrow 0} \int \widehat{f}(\xi) \widehat{g_\varepsilon}(\xi) \varphi(\xi) d\xi , $$
  for some $\widehat{g_\varepsilon}$ with $ \supp{\widehat{g_\varepsilon}} \subseteq \Omega + [-\varepsilon,\varepsilon] $.
  Since $ \Omega + [-\varepsilon,\varepsilon] \cap \supp{\varphi} = \emptyset $ for $\varepsilon$ small enough, the limit is zero.
  This means that $ \supp{\widehat{f}} \subseteq \Omega $, that is $ f \in B^p_\Omega $.
  
  Conversely, let $ f \in B^p_\Omega $.
  We shall prove that $ \ff f = \ff(f \con K) $, whence $ f = f \con K $ and $ f \in \mathcal{M}^p $.
  Thanks to formula~IV' at page~111 in \cite{plpo37}, there exists a continuous function $\psi$, periodic on 
  $\Omega$ such that
  \[
   f(x) = \int_\Omega [(1-x)\psi(\omega) + x\psi(\xi)] \e^{2\pi ix\xi} d\xi.
   \]
Hence
  $$ f(x) = (1-x)\psi(\omega)\ff^{-1}\chi_\Omega(x) + x\ff^{-1}(\psi\chi_\Omega)(x) , $$
  so that, in $\mathbb SÕ(\R)$,
  $$ \ff f = \psi(\omega)(1 - \frac{i}{2\pi}\partial)\chi_\Omega + \frac{i}{2\pi}\partial(\psi\chi_\Omega) . $$
  This tempered distribution acts on any   function
  $\varphi\in\mathbb S(\R)$ by
  \begin{align}
   \langle \ff f , \varphi \rangle
   &= \psi(\omega) \langle \chi_\Omega , (1 + \frac{i}{2\pi}\partial)(\varphi) \rangle
    - \langle \psi\chi_\Omega , \frac{i}{2\pi} \partial\varphi \rangle \notag \\
   &= \int_\Omega \left[ \psi(\omega) ( 1 + \frac{i}{2\pi}\partial )\varphi(\xi)
    - \psi(\xi) \frac{i}{2\pi} \partial\varphi(\xi) \right] d\xi . \label{eq:temp}
  \end{align}
  On the other hand, we know from Lemma \ref{lemma:band} that
  $$ \ff(f \con K) = \lim_{\varepsilon \rightarrow 0} \ff(f) \widehat{g_\varepsilon}, $$
  where the  limit is  taken in ${\mathbb S}'(\R)$.
  Compute now
  \begin{align*}
   \langle \ff(f) \widehat{g_\varepsilon} , \varphi \rangle &= \langle \ff f , \widehat{g_\varepsilon} \varphi \rangle \\
   &= \psi(\omega) \langle \chi_\Omega , (1 + \frac{i}{2\pi}\partial)( \widehat{g_\varepsilon}\varphi) \rangle
    - \langle \psi\chi_\Omega , \frac{i}{2\pi}\partial(\widehat{g_\varepsilon}\varphi) \rangle \\
   &= \int_\Omega \left[ \psi(\omega) ( 1 + \frac{i}{2\pi}\partial )\varphi(\xi)
    - \psi(\xi) \frac{i}{2\pi} \partial\varphi(\xi) \right] \widehat{g_\varepsilon} d\xi \\
   &+ \frac{i}{2\pi} \int_\Omega (\psi(\omega) - \psi(\xi)) \varphi(\xi) \partial \widehat{g_\varepsilon}(\xi) d\xi .
  \end{align*}
By ii) of Lemma~\ref{lemma:band},   $ \widehat{g_\varepsilon} \rightarrow \chi_\Omega $ pointwise, then the  limit of the first integral is precisely \eqref{eq:temp}.
  It remains to verify that the last integral tends to zero.
  Notice that it vanishes on $ [-\omega+\varepsilon,\omega-\varepsilon] $ because,
  by Lemma \ref{lemma:band}, $ \widehat{g_\varepsilon} = 1 $, hence $ \partial \widehat{g_\varepsilon} = 0 $.
  The rest of the integral is dominated by
  \begin{align*}
   &\left[ \sup_{-\omega\leqslant\xi\leqslant-\omega+\varepsilon}\abs{\psi(\omega)-\psi(\xi)}
   + \sup_{\omega-\varepsilon\leqslant\xi\leqslant\omega}\abs{\psi(\omega)-\psi(\xi)} \right] \nor{\varphi}_\infty
   \nor{\partial \widehat{g_\varepsilon}}_\infty \ \varepsilon \\
   \lesssim & \sup_{-\omega\leqslant\xi\leqslant-\omega+\varepsilon}\abs{\psi(\omega)-\psi(\xi)}
   + \sup_{\omega-\varepsilon\leqslant\xi\leqslant\omega}\abs{\psi(\omega)-\psi(\xi)} ,
  \end{align*}
  thanks to iii) of Lemma \ref{lemma:band}. But this tends to zero as $\varepsilon\to0$, because $ \psi(\omega) = \psi(-\omega) $.
  We have proven that $ \mathcal{M}^p = B^p_\Omega $ for all $p$.
  Thus we finally obtain
  \begin{equation*}
   \Ss' = \bigcup_{{p\in I}} B^p_\Omega \quad \mbox{and} \quad \Co{L^p(\R)} = B^p_\Omega . \qedhere  
  \end{equation*}
 \end{proof}

\vskip1truecm
\subsection{Shannon wavelet}\label{shannon}

We consider now the special case of a non integrable kernel for the
wavelet representation of the affine group on  $L^2(\R)$.   It is well known that this representation is
reproducing and admits admissible vectors whose kernel is integrable
\cite{grmorpa85,ag00,fegr88}.  The resulting
coorbit spaces are completely understood as homogeneous Besov spaces
\cite{feigro85,fegr88}.  However, there are
admissible vectors whose kernel is not integrable, as for example the
Shannon wavelet, which provides another example for which our
theory applies.

At first sight, this result might look surprising. Indeed, in
\cite{fegr88}, H. Feichtinger and K. Gr\"ochenig claim that any band limited
function whose Fourier transform has compact support bounded away from
zero leads to an integrable kernel. However, a careful inspection of
the proof reveals that the additional assumption that the Fourier
transform is continuous is implicitly needed.
  
Let $ G = \R \rtimes \R_+ $ be the connected component of the affine
group with  left Haar measure $ db\,da/a^2 $.  The wavelet representation $\pi$ acts on $ \hh = L^2(\R) $ by dilations and translations:
 $$ \pi(b,a)v(x) = a^{-1/2}v((x-b)/a) . $$
 The (real) Shannon wavelet  is defined as
 $$ \widehat{u}(\xi) = \chi_{[1/4,1/2]}(\abs{\xi}) = \chi_{[-1/2,1/2]}(\xi) - \chi_{[-1/4,1/4]}(\xi) , $$
 that is
 $$ u(x) = \sinc(\pi x) - \frac{1}{2}\sinc(\frac{\pi}{2}x) = \frac{1}{2}\sinc(\frac{\pi}{4}x)\cos(\frac{3}{4}\pi x).$$
 It is easily seen that $ u \notin L^1(\R) $, but $ u \in
   L^p(\R) $ for all $ p > 1 $. We now prove that the corresponding
   kernel has the same behavior. 
\begin{lem}
 The kernel $K=V_2u$ associated with the Shannon wavelet 
\[ u(x)=\frac{1}{2}\sinc(\frac{\pi}{4}x)\cos(\frac{3}{4}\pi x)\]
is in
 $L^p(G) $ for all $ p > 1 $, but it is not in $L^1(G)$.
 \end{lem}
 \begin{proof}
   Since $u$ is real and even, the voice transform ${V_2}$ is
 \[
 {V_2} v(b,a) = \int v(x) a^{-1/2}u((x-b)/a) dx = \left(v \con \pi(0,a)u\right) (b).
 \]
The Shannon kernel is thus
 $$ K(b,a) = {V_2} u(b,a) = (u \con \pi(0,a)u) (b) . $$
Since $u$ is admissible, $ K \in L^2(G) $ and, by Fubini
   theorem, $u \con\pi(0,a)u \in L^2(\R) $ for almost every $ a > 0
   $. Then,  by the convolution theorem for $L^2$-functions
 $$ \ff(u \con\pi(0,a)u)(\beta) = \widehat{u}(\beta)\widehat{\pi(0,a)u}(\beta)
 = a^{1/2} \chi_{[1/4,1/2] \cap [1/4a,1/2a]}(\abs{\beta}) . $$ It
 follows that:
 \begin{enumerate}[a)]
 \item \label{enum:Shannon-0} $ K(\cdot,a) \neq 0 $ only if $ a \in
   (1/2,2) $;
 \item \label{enum:Shannon-F} if $ a \in (1/2,2) $ the Fourier
   transform of $ K(\cdot,a) $ is a non-zero characteristic function.
 \end{enumerate}
 By~\ref{enum:Shannon-F}), for almost all $ a \in (1/2,2) $ the
   function $ u \con \pi(0,a)u $ cannot be in $L^1(\R)$, otherwise
   its Fourier transform would be continuous. Hence $ K \notin L^1(G)
 $.  Let us show that  $K \in L^r(G)$  for all $r > 1$. From
 \ref{enum:Shannon-0}) we have
\[
 \int_{\R_+} \int_\R \abs{K(b,a)}^r db \frac{da}{a^2} 
 = \int_{1/2}^2 \int_\R \abs{u \con\pi(0,a)u (b)}^r db \frac{da}{a^2}
 = \int_{1/2}^2 \nor{u \con\pi(0,a)u}_r^r \frac{da}{a^2} . 
 \]
 Recall that $ u \in L^p(\R) $ for all $ p > 1 $, so that the same holds for $\pi(0,a)u$ .  
 By Young's inequality \eqref{eq:5} for the unimodular group $G=\R$, we can estimate
 the inner norm and obtain
 $$ \nor{K}_r^r \leqslant \nor{u}_p^r \int_{1/2}^2 \frac{\nor{\pi(0,a)u}_q^r}{a^2} da , $$
 where $p$ and $q$ are such that $ 1/p + 1/q = 1/r + 1 $.  This
 integral is finite, because the function $ a \mapsto
 \nor{\pi(0,a)u}_q^r/a^2 $ is continuous and the interval $[1/2,2]$
 is compact.
\end{proof}
A Shannon wavelet coorbit theory can thus be implemented taking  voices in the target space 
 $ \T = \bigcap_{{p\in I}} L^p(G) $, but not in $L^1(G)$.

\subsection{Schr\"odingerlets}\label{SCR} 
In this section, we illustrate the example that has motivated the search for a full coorbit theory in which one encounters reproducing kernels that do enjoy nice integrability properties but are not necessarily in $L^1(G)$. The main feature of this example is that once the admissibility conditions are worked out, it is relatively easy to exhibit  kernels in $\bigcap_{{p\in I}}L^p(G)$ but hard  to find a kernel in $L^1(G)$. This example has shown up in the classification of reproducing triangular subgroups of 
$Sp(2,\R)$, which was recently achieved in  \cite{ABDD13, ADDM13}.

We shall be concerned with the three-dimensional group generated by rotations, dilations and flows  of two-dimensional signals, in a sense to be made precise below.
The group acts on functions via radial affine transformations, and the associated voice transform can thus be seen as a Fourier series of one-dimensional wavelets.
This representation is highly reducible, and reproducing.

The group $G$ is the  direct product of the (connected component of) affine group of the line with the unit circle
\[
 G = (\R \rtimes \R_+) \times S^1
\] 
and its elements are parametrized by $(b,a,\varphi)$ with $b\in\R$, $a>0$ and $\varphi\in[0,2\pi)$. A  left Haar measure is
\[ 
 dx =\frac{db\,da}{ a^{2}}\, \frac{d\varphi}{2\pi}.
\]
Notice that $G$ is not unimodular and has  modular function  $
\Delta(b,a,\varphi)=a^{-1} $.

The representation $\pi$ that  we are going to define acts on $ L^2(\R\times S^1)$,  
endowed with the tensor product of the Lebesgue
  measure and the normalized Haar measure on $S^1$. The action is
\begin{equation} \label{eq:schlets_repr}
 \pi(b,a,\varphi)v(x,\vartheta) = a^{-1/2} \ v((x-b)/a,\vartheta-\varphi), \qquad v \in  L^2(\R\times S^1).
\end{equation}
Since $L^2(\R\times S^1)=  L^2(\R) \otimes L^2(S^1) $,  $\pi$
is simply the tensor product $ \pi= w \otimes \lambda $ where $w$ is
the wavelet representation of the affine group
\begin{equation}
 w(b,a) g(x) = a^{-1/2} \ g((x-b)/a), \qquad g\in L^2(\R),
\end{equation}
and $\la$ is the  left regular representation of $S^1$ on $L^2(S^1)$, namely
\begin{equation}
 \lambda(\varphi) h(\vartheta) = h(\vartheta-\varphi), \qquad h \in L^2(S^1).
\end{equation}
In what follows, we denote by $\ff_x$ the unitary Fourier transform
from $L^2(\R)$ onto $L^2(\widehat{\R})$, which is also regarded as a
unitary map from $L^2(\R\times S^1)$ onto $L^2(\widehat{\R})\otimes
L^2(S^1)$. Furthermore, we denote by $\ff_\vartheta$ the unitary Fourier transform
from $L^2(S^1)$ onto $\ell^2(\Z)$, which is also regarded as a
unitary map from $L^2(\R\times S^1)$ onto $L^2(\R)\otimes
\ell^2(\Z)$.  Explicitly, if $v \in C_c^\infty(\R\times S^1)$, then
\begin{subequations}
  \begin{align}
    \ff_x v(\rho,\vartheta) & = \int_{\R} v(x,\vartheta) \e^{-2 \pi i \rho
      x} dx, \label{partialFx}\\
    \ff_\vartheta v(x,n) & =\int_{\R} v(x,\vartheta)
    \e^{-in\vartheta} \frac{d\vartheta}{2\pi}=\int_{\R} v(x,\vartheta)
    \overline{e_n(\vartheta)}\frac{d\vartheta}{2\pi},
    \label{partialF}
  \end{align}
where $e_n(\vartheta)= \e^{in\vartheta}$.
\end{subequations}
The partial Fourier transform  $\ff_\vartheta v$ of any  $v\in L^2(\R\times S^1)$
  can be identified with  the
sequence of functions $(v_n)_{n\in\Z}$ in $L^2(\R)$, where $v_n=\ff_\vartheta
v(\cdot,n) $. Hence
\begin{equation}
v= \sum_{n\in\Z} v_n \otimes e_n, \qquad 
\nor{v}^2_{L^2(\R\times  S^1)}= \sum_{n\in\Z} \nor{v_n}^2_{L^2(\R)}.\label{eq:60}
\end{equation}
To simplify the computations,  we restrict the
representation $\pi$ to the closed subspace
$\hh=\ff_x^{-1}L^2(\widehat{\R}_+)\otimes L^2(S^1)$,
so that the wavelet representation $w$  acts irreducibly on $
\ff_x^{-1}L^2(\widehat{\R}_+) $.
Given a vector $u\in \hh$, we denote
by $V$ the voice transform corresponding to the representation $\pi$
of $G$ and the analyzing vector $u$, namely
\[ Vv(b,a,\vartheta)= \scal{v}{\pi(b,a,\vartheta)u}_\hh \qquad v\in\hh,\]
and by $V^{\mathrm{w}}_n$ the voice transform corresponding to the representation
$w$ of the affine group and the analyzing vector $u_n$, {\em i.e.}
\[ 
V^{\mathrm{w}}_ng(b,a)= \scal{g}{w(b,a)u_n}_{L^2(\R)}.
\]
We use the unitary operator $\ff_x:\hh\to L^2(\widehat{\R}_+\!\times\!S^1)$ to obtain an intermediate equivalent version of $\pi$, denoted $\ff_x (\pi)$,  acting on 
$L^2(\widehat{\R}_+\!\times\!S^1)$. This is defined via the intertwining 
$\ff_x\circ\pi(g)= \ff_x (\pi)(g)\circ\ff_x$ for every $g\in G$. The analytic expression of $\ff_x (\pi)$ is 
immediately computed to be 
\[
\ff_x(\pi)(b,a,\varphi)v(\xi,\vartheta) = a^{1/2} \e^{-2 \pi i b \xi}  v(a\xi,\vartheta-\varphi),
\]
whereas from the structural point of view it may be written as
\[
 \ff_x (\pi) = \widehat{w} \otimes \lambda, 
\]
where  $ \widehat{w}(b,a)=\ff_x\circ w(b,a)\circ\ff_x^{-1}$.

The group $G$ can be realized as the triangular subgroup of  $Sp(2,\R)$ consisting of the matrices
 $$ \begin{bmatrix}
      \ a^{-1/2} R & 0 \\
      b a^{-1/2} R & a^{1/2} R
    \end{bmatrix}, \qquad b \in \R , a>0 , R \in SO(2).
 $$
Thus, $G$ may also be seen as the semidirect product  $\R \rtimes (\R_+\!\times\!SO(2))$,
 where the homogeneous factor $\R_+\!\times\!SO(2)$ acts on the normal subgroup $\R$ by isotropic dilations. 
We shall not distinguish between $S^1$ and $SO(2)$ and write rotations as 
 \[
 R_\varphi=\begin{bmatrix}\cos\varphi&-\sin\varphi\\\sin\varphi&\cos\varphi\end{bmatrix},
 \qquad\varphi\in[0,2\pi).
 \]
We  show below that $\pi$ is equivalent to the metaplectic representation $\mu$ as restricted to the above group,  defined in the frequency domain by
 $$
 \mu(b,a,\varphi)v(\xi) = a^{1/2} e^{-2 \pi i b \abs{\xi}^2} v(a^{1/2}R_{-\varphi}\xi),
  \qquad v \in L^2(\widehat{\R^2}).
$$
The space-domain version of this representation explains the reason of the name {\it Schr\"odingerlets}.
Denote by $\widehat\mu$ the representation obtained by conjugating $\mu$ with the Fourier transform, namely
\[
\widehat\mu(g)f=\ff^{-1}\circ\mu(g)\circ\ff.
\]
We now interpret $b\in\R$ as a time parameter and look at the evolution  {\it flow} of 
$f\in  L^1(\R^2)\cap L^2(\R^2)$
\[
(b,x)\mapsto\widehat\mu_bf(x)=\widehat\mu(b,1,0)f(x)=\int_{\widehat{\R^2}}\hat f(\xi)e^{-2\pi i b|\xi|^2}e^{2\pi i x\cdot\xi}\,d\xi.
\]
It is then straightforward to verify that the flow $\widehat\mu_bf$ satisfies the Schr\"odinger equation
\[
\Bigl(2\pi i\frac{\partial}{\partial b}+\Delta\Bigr)\widehat\mu_bf(x)=0,
\]
where $\Delta$ is the spacial Laplacian
\[
\Delta=\frac{\partial^2}{\partial x_1^2}+\frac{\partial^2}{\partial x_2^2}.
\]
It is in this sense that the group is generated by (dilations, rotations and)  flows.

We now prove the equivalence. The unitary map $\Psi : L^2(\widehat{\R^2}) \to L^2(\widehat{\R}_+\!\times\!S^1)$, defined by
\[
 \Psi v(\xi,\vartheta) = {\pi^{1/2}} \ v(\sqrt{\xi}\cos\vartheta,\sqrt{\xi}\sin\vartheta),
 \]
 intertwines $\mu$ with $\ff_x(\pi)$ because, for $v \in L^2(\widehat{\R^2})$,  we have on the one hand
 \begin{align*}
\Psi\left(\mu(b,a,\varphi)v\right)(\xi,\vartheta)
&=\Psi\left(a^{1/2} e^{-2 \pi i b \abs{\cdot}^2} v(a^{1/2}R_{-\varphi}(\cdot))\right)(\xi,\vartheta)     \\
&={\pi^{1/2}} a^{1/2}e^{-2 \pi i b \abs{(\sqrt{\xi}\cos\vartheta,\sqrt{\xi}\sin\vartheta)}^2}
v\left(a^{1/2}R_{-\varphi}(\sqrt{\xi}\cos\vartheta,\sqrt{\xi}\sin\vartheta)\right)     \\
&={\pi^{1/2}} a^{1/2}e^{-2 \pi i b\xi}
v\left(\sqrt{a}(\sqrt{\xi}\cos(\vartheta-\varphi),\sqrt{\xi}\sin(\vartheta-\varphi))\right)
\end{align*}
and on the other hand
\begin{align*}
\ff_x(\pi)(b,a,\varphi)\left(\Psi v\right)(\xi,\vartheta)&= a^{1/2} \e^{-2 \pi i b \xi} (\Psi v)(a\xi,\vartheta-\varphi)     \\
&= a^{1/2} \e^{-2 \pi i b \xi}{\pi^{1/2}}  v(\sqrt{a\xi}\cos(\vartheta-\varphi),\sqrt{a\xi}\sin(\vartheta-\varphi)),
\end{align*}
as claimed. In conclusion, since $\pi$ and $\ff_x(\pi)$ are equivalent, so are $\pi$ and $\mu$.

We point out that $\Psi$ is simply the change from rectangular to polar-like  coordinates $(\sqrt{\xi} , \vartheta)$, together with the appropriate $L^2$-normalization.

Since  the wavelet representation is irreducible, while $\lambda$
completely reduces to ${\oplus_{n\in\Z} e_{-n}} $, {where each function $e_n$ is
regarded as a character of $S^1$}, then
\[
\pi = \bigoplus_{n \in \Z} w \otimes {e_{-n}},
\]
which  expresses $\pi$ as a  sum of irreducibles. This allows  us to view the voice transform of $\pi$ as a Fourier series of one-dimensional wavelet transforms, as clarified in the next proposition.
\begin{prop} Let $u=\sum_{n\in \Z} u_n\otimes e_n\in\hh$.
 The voice transform $V$ associated with $\pi$ and $u$ admits the trigonometric expansion
 \begin{equation} \label{eq:schlets-series}
 V v(b,a,\varphi) = \sum_{n \in \Z}{V^{\mathrm{w}}_n}v_n(b,a) \e^{i n \varphi},
 \end{equation}
where  the series converges pointwise for all $(b,a,\varphi)\in
  G$ and where
 \begin{equation} \label{eq:schlets-coeff}
 V^{\mathrm{w}}_n v_n (b,a)= \int_{S^1} Vv(b,a,\vartheta) e^{-in\vartheta} \frac{d\vartheta}{2\pi}.
 \end{equation}
\end{prop}
\begin{proof} 
Since  $\ff_\vartheta$ is a unitary map,  for all $(b,a,\varphi)\in G$ 
\begin{align*}
\langle v , \pi(b,a,\varphi)u \rangle_\hh & = \langle \ff_\vartheta v ,
\ff_\vartheta\left( w(b,a) \otimes \la(\vartheta)\right)u
\rangle_{L^2(\R)\otimes \ell^2(\Z)} \\
& =   \langle \ff_\vartheta v ,
\left( w(b,a) \otimes \ff_\vartheta\la(\varphi) \ff_\vartheta^{-1}\right)
\ff_\vartheta u \rangle_{L^2(\R)\otimes \ell^2(\Z)} \\
&= \sum_{n\in \Z} \scal{v_n}{ w(b,a)u_n}_{L^2(\R)}
\overline{e_{-n}(\varphi)} \\
& =  \sum_{n\in \Z}{V^{\mathrm{w}}_n}v_n(b,a) e^{in\varphi},
\end{align*}
where the third line is due to the fact that the action of $\ff_\vartheta\la(\vartheta) \ff_\vartheta^{-1}$ on
$\ell^2(\Z)$ is the multiplication operator by the sequence $( e_{-n}(\varphi))_n$.
For fixed $(b,a)\in \R\rtimes\R_+$,  the function $\vartheta\mapsto
Vv(b,a,\vartheta)$ is continuous and, hence, integrable on $S^1$. Therefore, by de la
Vall\'ee--Poussin theorem, (see  iii) of Theorem~11.3 in
\cite{zyg2002}),  we obtain~\eqref{eq:schlets-coeff}.
\end{proof}
The Fourier expansion \eqref{eq:schlets-series} shows how to construct
admissible vectors as series of wavelets. This result has been
originally  obtained in \cite{ADDM13} and in general setting in \cite{fuhr2005}. For the reader's convenience, we give here a more direct proof. 
\begin{prop}
The representation $\pi$ is reproducing and a vector $u=\sum_{n\in\Z}u_n\otimes e_n\in\hh$ is admissible for $\pi$ if and only if  for all $ n \in \Z $
 \begin{equation} \label{eq:nCalderon}
  \int_{\R_+} |{\ff_xu_n}(\xi)|^2 \frac{d\xi}{\xi} = 1.
 \end{equation}
 Furthermore, given a sequence ${(u_n)_{n\in\Z}}$ with  $u_n\in \ff^{-1}L^2(\widehat{\R}_+) $, we have
 \begin{equation} \label{eq:H-convergence}
  \sum_{n\in\Z}{ u_n \otimes \, e_n}\in \hh \qquad 
  \iff \qquad \sum_{n\in\Z} \int_{\R_+} \abs{{\ff_xu_n}(\xi)}^2 d\xi < +\infty.
 \end{equation}
If $u$ is an admissible vector, the voice
  transform from $\hh$ into $L^2(G)=L^2(\R\rtimes\R_+)\otimes L^2(S^1)$ is
  \begin{equation}
    \label{eq:61}
    V_2v= \sum_{n\in\Z} V^{\text{w}}_n v_n\otimes e_n\qquad v=\sum_{n\in Z}
    v_n\otimes e_n\in\hh
  \end{equation}
and the series~\eqref{eq:schlets-series} converges also in
$L^2(G)$.
\end{prop}
\begin{proof}
 Admissibility of $u$ means that $ \nor{{V} v}_{L^2(G)} = \nor{v}_\hh $ must hold for all $v\in\hh$.
{Fix $n\in\Z$ and choose $v=v_n\otimes e_n$ with $v_n\in
  \ff^{-1}L^2(\widehat{\R}_+)$. By} \eqref{eq:schlets-series},
when computing the norm, the integral on the circle is equal to $1$,
whereas the integration on ${\R\rtimes \R_+}$ provides  the classical admissibility condition,  namely Calder\'on's equations~\eqref{eq:nCalderon}.
 
 Conversely, suppose \eqref{eq:nCalderon}  true for every $n\in\Z$.
 Fix $v\in\hh$.  Given $(b,a)\in\R\rtimes\R_+$, \eqref{eq:schlets-coeff}
 implies that the function $Vv(b,a,\cdot)$ is in $L^2(S^1)$ if and only if the sequence
 $(Vv_n(a,b))_{n\in\Z}$ is in $\ell^2(\Z)$ and, under this assumption,
Fubini theorem yields 
\begin{equation}
  \nor{{V} v}^2 _{L^2(G)} = \int_{\R\rtimes\R_+} \sum_{n\in\Z}
  \abs{{V^{\text{w}}_n}v_n(b,a)}^2 \frac{dbda}{a^2}= \sum_{n\in\Z} \int_{\R} \abs{v_n(x)}^2
  dx = \nor{v}_\hh^2\label{eq:62}
\end{equation}
because by~\eqref{eq:nCalderon}  for each $n$ the voice $V^{\text{w}}_n$ is an
isometry from $\ff_x^{-1} L^2(\widehat{\R_+})$ into
$L^2(\R\rtimes\R_+, dbda/a^2)$.  The last equality is due to~\eqref{eq:60}.
Equation~\eqref{eq:H-convergence}  is a consequence of~\eqref{eq:60}
and the fact that $\ff_x$ is unitary.  

To prove that $\pi$ is reproducing, it is enough to show there exists
a sequence $(u_n)_n$ in  $\ff^{-1}L^2(\widehat{\R}_+)$ satisfying both~\eqref{eq:nCalderon}
and~\eqref{eq:H-convergence}. 
Fix $u_0\in \ff^{-1}L^2(\widehat{\R}_+)$
satisfying~\eqref{eq:nCalderon}, {\em i.e.},  an admissible vector for
the wavelet representation $w$. For all $n\in \Z$ 
define the functions $u_n \in\ff^{-1}L^2(\widehat{\R}_+)$  as
\[ u_n(x)=a_n u_0 (a_nx),\qquad \ff_x u_n(\xi) = \ff_x u_0 (a_n^{-1}\xi),\]
where $a_n>0$ and $\sum_{n\in\Z} a_n < +\infty$.  Since $u_0$ satisfies
\eqref{eq:nCalderon}, so do all the functions $u_n$. Further,
$$ 
\sum_{n\in\Z} \int_{\R_+} \abs{\ff_xu_0(\xi)}^2 d\xi =
\sum_{n\in\Z} \int_{\R_+} \abs{\ff_xu_0(a_n^{-1}\xi)}^2 d\xi =
\nor{u_0}^2 \sum_{n\in\Z} a_n < +\infty,
$$
so that by~\eqref{eq:H-convergence} the vector $ u = \sum u_n\otimes
e_n$ is in $\hh$ and  is admissible for $\pi$.

Finally, we prove~\eqref{eq:61}. By~\eqref{eq:62}  the series
$\sum_{n\in\Z} V^{\text{w}}_n v_n\otimes e_n$ converges in $L^2(G)$
to $V_2v$.
\end{proof}

Now we come to the integrability question. The idea is based on the  very simple observation that
 Calder\'on's equation \eqref{eq:nCalderon} is invariant under dilations.

\begin{prop}
There exist admissible vectors $ u \in \hh $ whose kernel $K={V_2}
 u$ belongs to $\bigcap_{p\in I}L^p(G)$ but  not to $L^1(G)$.
\end{prop}
\begin{proof}
Define $u$ as in the second part of the proof of the above
 proposition. Using~\eqref{eq:61} we write $K=\sum K_n \otimes e_n$ where 
 $ K_n = {V^{\text{w}}_n}u_n $ and the series converges both in $L^2(G)$ and pointwise. By a simple change of variable, we get that 
 $ K_n(b,a) = a_n K_0(a_n b,a)$.  Therefore for any $p\in I$
 \[
 \nor{K}_p \leqslant \sum_{n\in\Z} a_n \left( \int_{\R\rtimes\R_+} \abs{K_0(a_n b,a)}^p \frac{dbda}{a^2} \right)^{1/p}
    = \nor{K_0}_p \sum_{n\in\Z} a_n^{1-1/p}.
 \]
In order to construct $u$, it is therefore sufficient to take a
positive sequence for which $\sum_{n\in\Z}a_n^\alpha$ converges for
every $\alpha\in(0,1]$.

 We now prove that the kernel is not in $L^1(G)$.  By
  contradiction, assume that $K\in L^1(G)$. Fubini's theorem implies
  that for almost all $a\in \R_+$ the function $K(\cdot,a,\cdot)$ is
  in $L^1(\R\times S^1)$. Hence, regarding $\R\times S^1$ as an
  abelian group, its Fourier transform
\begin{align*}
  \mathcal F K(\xi,n) & = \int_{\R\times S^1} K(b,a,\varphi)
  e^{-in\varphi} e^{-2\pi i b\xi} db \frac{d\varphi}{2\pi} \\
   & =  \int_{\R} \left(\int_{S^1} K(b,a,\varphi)
  e^{-in\varphi} db \right) e^{-2\pi i b\xi} \frac{d\varphi}{2\pi} 
\end{align*}
is in $C_0(\widehat{\R}\times\mathbb Z)$. By~\eqref{eq:schlets-coeff}, it holds
that
\begin{align*}
  \mathcal F K(\xi,n) & = \int_{\R} a_n K_0(a_nb,a) e^{-2\pi i
    b\xi}\,db \\
& = \int_{\R}  K_0(b) e^{-2\pi i b\frac{\xi}{a_n}}\,db =  \hat{g}(\frac{\xi}{a_n}),
\end{align*}
where $\hat g$ is the Fourier transform of the function
$K_0(\cdot,a)$, which is in $L^1(\R)$ by Fubini's theorem. Fix
$\overline{\xi}\in\widehat{\R}$ and set $\xi=a_n \overline{\xi}$ in
the above equality. Then 
\[
\hat{g}(\overline{\xi} ) =\lim_{n\to\infty} \mathcal F K(a_n
\overline{\xi},n) = 0,
\]
because $ \mathcal F K\in C_0(\widehat{\R}\times\mathbb Z)$. Hence, by
injectivity of the Fourier transform, $K_0(b,a_0)=0$ for almost all
$b\in \R$. Since the above equality holds for almost all $a\in
\R_+$, we get that $K_0=0$, which is a contradiction.
\end{proof}

\section{$L^1$-kernels: the non irreducible coorbit theory}\label{L1}

In this section, we apply our machinery  and show
that the standard setup of coorbit theory makes sense without assuming
that the representation $\pi$ is irreducible, because it corresponds
to  the case arising from the classical choice $\T=L^1_w(G)$.
The fact that irreducibility is a somewhat redundant assumption has been perhaps known to some extent, but it is not easy to pin down precise statements in the literature. 
Theorem~\ref{nonirreducible} below contains a summary of the most relevant facts.

It is perhaps worthwhile observing that the present case is structurally different from the case discussed in Section~\ref{mainexample} because $L^1_w(G)$ is not a reflexive space.

Throughout this section, we fix a continuous function $w:G\to (0,+\infty)$ satisfying the following
assumptions (see \cite{gro91}):
\begin{subequations}
  \begin{align}
    \label{eq:33}
    w(xy) & \leq w(x)w(y) \\
    w(x) & \geq 1  \label{eq:46} 
  \end{align}
for all $x,y\in G$.
\end{subequations}
We choose as target space $\T$ the Banach space $\Luw$ and denote by $j$
the canonical inclusion into
$\Lloc\subset\Lz$. Since $j$ is canonical, we do not write it
explicitly, especially because it would conflict with the current literature,
where no explicit mention of the embedding is ever made.

Since Lemma~\ref{4.1} and Lemma~\ref{leftreg} do not depend on
  Assumption~\eqref{eqb:45},
\[\Luw^\#=\set{\Phi\in \Lz \mid w^{-1} \Phi \in L^\infty(G)}\]
the left regular representation $\la$ leaves  
    $L^1_w(G)$ invariant, and  the restriction $\ell$ of $\la$ to
    $L^1_w(G)$ is a continuous representation that satisflies
     \begin{equation}
\norop{\la(x)}{1,w}\leq w(x),\qquad x\in G.\label{eq:51}
\end{equation}

We assume that there exists an admissible vector $u\in \hh$ whose voice transform
$V_2u$ is in $\Luw$  and construct the corresponding reservoir $\Ss_w$ of
test functions.  We are in a position of stating  the main properties of the standard setup, without the assumption that the representation is irreducible.
\begin{thm}\label{nonirreducible}  Take a reproducing representation $\pi$ of $G$ acting on
  the Hilbert space $\hh$ and a weight $w$ satisfying \eqref{eq:33} and \eqref{eq:46}. 
Choose an admissible vector $u\in\hh$ such that
\[ 
K(\cdot)=\scal{u}{\pi(\cdot)u}_\hh \in L^1_w(G)
\]
and set
\begin{align*} 
&\Ss_w =\set{v\in\hh\mid \scal{v}{\pi(\cdot)u}_\hh \in L^1_w(G)},\\
&\nor{v}_{\Ss_w}= \int_G \abs{\scal{v}{\pi(x)u}_\hh}w(x) dx. 
\end{align*}
  \begin{enumerate}[a)]
   \item The space ${\Ss_w}$ is a Banach space and the canonical inclusion
     $i:{\Ss_w}\to\hh$ is continuous and  with dense range.
\item The representation $\pi$ leaves ${\Ss_w}$ invariant,  its restriction $\tau$ is a continuous
    representation acting on ${\Ss_w}$, the operator norms satisfy $\norop{\tau(x)}{{\Ss_w}}\leq
    w(x)$ for all $x\in G$, and
\[ i(\tau(x)v)=\pi(x)i(v) \qquad x\in G,\,v\in{\Ss_w}.\]
\item Endowing $\hh$ with the weak topology and $\Ss'$ with the
    weak-$*$ topology, the transpose mapping $\tr{i}:\hh\to {\Ss'_w}$ is continuous, injective
and  with dense range, and satisfies the intertwining 
\[ \tr{\tau(x)}\tr{i}(w)=\tr{i}(\pi(x)w) \qquad x\in G,\,w\in \hh.\]
\item The restricted voice transform $V_0:{\Ss_w}\to \Luw$ is an
  isometry intertwining $\tau$ and $\la$ and its 
range  is the closed subspace
\[ \mathcal M^1=\set{f\in \Luw  \mid f\con K= f},\]
\item For all $f\in \Luw$, the  Fourier
  transform of $f$ at $u$ exists  in $\Ss_w$ and satisfies
\[
V_0\pi(f)u= f\con K.
\]
Furthermore, the map
\[ 
L^1_w(G)\ni f\mapsto \pi(f)u\in {\Ss_w} 
\]
is continuous and its restriction to $\mathcal M^1$
is the inverse of $V_0$.
\item The extended voice transform $V_e:{\Ss'_w}\to\Liw$ is injective,
  continuous (when both spaces are endowed with the topology of the
  compact convergence) and intertwines  $\tr{\tau}$ and $\la$. The range
  of $V_e$ is the closed subspace 
\[ \mathcal M^{\infty}=\set{\Phi\in \Liw \mid \Phi\con K=
  \Phi}\subset C(G).\]
For all $T\in \Ss_w'$ and $v\in \Ss_w$
\begin{equation}
\scal{T}{v}_{\Ss_w}= \scal{V_eT}{V_0v}_{\Luw} .\label{eq:52}
\end{equation}
\item For all $\Phi\in\Liw$  the Fourier
  transform of $\Phi$ exists at $u$ in $\Ss_w'$  and satisfies
\[ V_e \pi(\Phi)u= \Phi\con K.\]
\item The map 
\[  
\mathcal M^{\infty}\ni \Phi\mapsto \pi(\Phi)u\in {\Ss'_w}
\]
is the inverse of $V_e$ and coincides with the restriction of the map
  $\tr{\,V_0}$ to $\mathcal M^{\infty}$, namely 
  \begin{equation}
    \label{eq:39bis}
   V_e( \tr{\,V_0}\Phi)= V_e\pi(\Phi)u= \Phi,\qquad \Phi\in \mathcal M^{\infty}.
  \end{equation}
\item  $\displaystyle{\tr{i}i({\Ss_w})=\set{T\in{\Ss'_w}\mid V_eT\in
      \Luw}=\set{\pi(f)u \mid f\in \mathcal M^1}}$.
  \end{enumerate}
\end{thm}

\begin{proof}  Since $\Luw\subset L^1(G)$ and $K, V_2v\in L^\infty(G)$ for all
    $v\in\hh$, Assumptions~\ref{H1} and \ref{H2} are satisfied.
  \begin{enumerate}[a)]
  \item[a)] By the first part of Theorem~\ref{tuno}  which does not depend on Assumption~\ref{H3}. The
  space $\Ss$ is a Banach space because $\mathcal M^{1}$ is such.
\item[b)] Apply Theorem~\ref{tuno}. Moreover, by ~\eqref{eq:51}
\[\nor{\tau(x)v}_{\Ss_w}=\nor{V_0
  \tau(x)v}_{1,w}=\nor{\la(x)V_0v}_{1,w}\leq w(x)\nor{V_0v}_{1,w}=
w(x) \nor{v}_{\Ss_w}.\]
\item[c)]  Apply Theorem~\ref{tuno}.
\item[d)] Apply Theorem~\ref{lsei}.
\item[e)] Fix $f\in \Luw$ and set $\Psi:G\to{\Ss_w}$, $\Psi(x)=
f(x) \tau(x)u$. We show that $\Psi$
is Bochner-integrable with respect to $\beta$. The map $\Psi$ is 
$\beta$-measurable since $f\in \Lz$ and $x\mapsto\tau(x)u$ is
continuous from $G$ into $\Luw$, and, by item b),
\[ \nor{\Psi(x)}_{\Ss_w}= \abs{f(x)}\nor{\tau(x)u}_{\Ss_w}\leq w(x) \nor{u}_{\Ss_w}
\abs{f(x)},\]
which is in $L^1(G)$ since $f\in\Luw$.  Set
\[ \pi(f)u=\int_G f(x) \tau(x)u\, dx.\]
Clearly, for all $v\in{\Ss_w}$ 
\[ \scal{\tr{i}i \pi(f)u}{v}_{\Ss_w}= \int_G f(x) \scal{i(\tau(x)u)}{i(v)}_\hh
dx = \int_G f(x) \scal{\pi(x) u}{i(v)}_\hh dx.\]
Hence $\tr{i}i \pi(f)u$ satisfies ~\eqref{eq:12}  and we can identify
$\pi(f)u$ with $\tr{i}i \pi(f)u$. So $V_e\pi(f)u=V_0\pi(f)u$. The fact
that $V_0\pi(f)u= \Phi\con K$ follows from \eqref{eq:37} with $\F=\Luw$ and $\E={\Ss_w}$. The fact that
$f\mapsto \pi(f)u$ is the inverse of $V_0$ follows from~\eqref{eq:38c} in
Proposition~\ref{cuno}.
\item[f)] We first prove that $V_e\Ss'_w\subset \Liw$.  Take $T\in\Ss'_w$. For all $x\in G$,  by~\eqref{eq:51}
\[ \abs{\scal{T}{\tau(x)u}_{\Ss_w}}\leq \nor{T}_{\Ss'_w} \nor{\tau (x)u}_{\Ss_w}=
\nor{T}_{\Ss'_w} \nor{\la(x)Vu}_{1,w} \leq w(x)   \nor{T}_{\Ss'_w}
\nor{K}_{1,w},\]
 so that $w^{-1}V_eT$ is bounded and continuous. 
We now prove the reconstruction formula~\eqref{eq:52}. Fix $v\in {\Ss_w}$ and define the map $\Psi:G\to {\Ss_w}$ by $\Psi(x)=
\scal{\pi(x)u}{i(v)}_\hh \tau(x)u=\overline{V_0v(x)} \tau(x)u$. We show that it
is Bochner-integrable with respect to $\beta$. It is 
continuous since both $V_0v$ and $\tau(\cdot)u$ are continuous, and
\[ \nor{\Psi(x)}_{\Ss_w}= \abs{V_0v(x)}\nor{\tau(x)u}_{\Ss_w}\leq w(x) \abs{V_0v(x)}
\nor{K}_{1,w},\]
which is in $L^1(G)$ by definition of ${\Ss_w}$.  Hence there exists
$w_v\in {\Ss_w}$ such that
\[ w_v= \int_G \scal{\pi(x)u}{i(v)}_\hh \tau(x)u\, dx.\]
For all $z\in\hh$ we have $\tr{i}(z)\in \Ss'_w$ and 
\begin{align*}
  \scal{z}{i(w_v)}_\hh=\scal{\tr{i}(z)}{w_v}_{\Ss_w} & = \int_G \scal{\pi(x)u}{i(v)}_\hh
  \scal{\tr{i}(z)}{\tau(x)u}_{\Ss_w} dx  \\
 &  = \int_G \scal{\pi(x)u}{i(v)}_\hh \scal{z}{\pi(x)u}_\hh dx  = \scal{z}{i(v)}_\hh,
\end{align*}
that is, $w_v=v$.  Hence
\[ v= \int_G \scal{\pi(x)u}{i(v)}_\hh \tau(x)u\, dx,\]
where the integral is a Bochner integral in $\Ss_w$.
Take $T\in \Ss'_w$. Then for all $v\in {\Ss_w}$
\[ \scal{T}{v}_{\Ss_w} = \int_G \scal{\pi(x)u}{i(v)}_\hh
\scal{T}{\tau(x)u}_{\Ss_w}\, dx,\]
which proves the reconstruction formula. This, in turn,  implies that $V_e$ is injective.
\item[g)] Apply item a) of Proposition~\ref{cuno} with $\E={\Ss_w}$
and $\F=\Liw$.
\item[h)] Items b) and d) of Proposition~\ref{cuno} with $\F=\Liw$ and $\E={\Ss_w}$ show  that the range
  of $V_e$ is the closed subspace $\mathcal M^\infty$ and that  the inverse
  of $V_e$   is $\Phi\mapsto \pi(\Phi)u$. 
\item[i)] Since $V_0{\Ss_w}\subset\Luw$ and $V_0v=V_e\tr{i}i(v)$ for all
  $v\in{\Ss_w}$, it follows that  $\tr{i}i({\Ss_w})\subset \set{T\in\Ss'_w\mid V_eT\in
      \Luw}$. Furthermore, Proposition~\ref{cuno} with $\F=\Luw$ and
    $\E={\Ss_w}$ gives that 
\[\set{T\in\Ss'_w\mid V_eT\in      \Luw}=\set{\pi(f)u \mid f\in \mathcal M^1}.\]
Item~d) of this
  theorem finally implies  that $\set{\pi(f)u \mid f\in \mathcal M^1}={\Ss_w}$. \qedhere
  \end{enumerate}
\end{proof}
As shown in the previous proof, Assumptions~\ref{H1} and \ref{H2} are satisfied.
The reconstruction formula~\eqref{eq:52} makes
 clear that $u$ is a cyclic vector for $\tau$, which is equivalent to
 Assumption~\ref{H3} because $V_0\tau(x)u=\ell(x)K$ and $V_0$ is an
 isometry from $\Ss_w$ onto $\mathcal M^1$.
Furthermore,~\eqref{eqb:44}  with $v=\tau(x)u$ implies that also
Assumption~\ref{H4} holds true.

From now until the end of this section  we choose  a Banach space  $Y$ with a
continuous embedding $j:Y\to \Lloc$,  denoted $f\mapsto
f(\cdot)$, and with  a continuous involution $f\mapsto \overline{f}$. We further suppose that there are two  continuous representations $\ell$ and $r$ of $G$ on $Y$ satisfying 
\begin{enumerate}[i)]
\item for all $x\in G$ and all $f\in Y$ 
  \begin{equation}
    \label{eq:53}
    j(\ell(x) f)= \la(x) j(f),\qquad j(r(x) f)= \rho(x) j(f);
  \end{equation}
\item for all $f\in Y$ and almost every  $x\in G$,
  \begin{equation}
    \label{eq:54}
    j(\overline{f})(x)=\overline{j(f)(x)}.
  \end{equation}
\end{enumerate}

\begin{prop}\label{prop4}
Assume that $Y$ is a Banach space with a continuous representation $r$
for which there exists a continuous embedding $j:Y\to \Lloc$, denoted $f\mapsto f(\cdot)$, such that,
  for all $x\in G$, all $f\in Y$   and almost every $y\in G$ it holds that
  $r(x)f\,(y)=f(yx)$.  Suppose that $g\in \Lloc$ is such that for all $f\in Y$
\begin{equation}
  \label{eq:47}
  \int_G \abs{g(x^{-1})}  \nor{r(x)f}_{Y} \,dx<+\infty.
\end{equation}
Then $j(f)$ and $g$ are convolvable, there exists $f\con g\in Y$ satisfying
$j(f)\con g=j(f\con g)$ and 
\[\nor{f\con g}_Y\leq  \int_G \abs{g(x^{-1})}
\nor{r(x)f}_{Y} \,dx.\]
\end{prop}
\begin{proof}
The proof is closely related to the proof of Proposition~6 Chapter
VIII.4.2 of \cite{BINT2}.   Fix $f\in Y$ and $g\in \Lloc$ and set
\[\Psi:G\to  Y,\qquad \Psi(x)= g(x^{-1})r(x) f.\]
We claim that $\Psi$ is $\beta$-integrable in the Bochner sense. 
Since $r$ is a continuous representation,  the map $x\mapsto r(x)f$ is
continuous from $G$ to $Y$ and hence it is $\beta$-measurable.  Since $g\in
\Lloc$,  so is $\check{g}$, and hence $\Psi$ is
$\beta$-measurable. Furthermore, 
\[ x\mapsto \nor{\Psi(x)}_Y = \abs{g(x^{-1})}\nor{r(x)f}_{Y}\]
is $\beta$-integrable by assumption. Hence $\Psi$ is
$\beta$-integrable. Define
\[ v= \int_G   g(x^{-1}) r(x)f\,  dx \in Y,\]
which clearly satisfies 
\begin{equation}
\nor{v}_Y \leq \int_G \abs{g(x^{-1})}\nor{r(x)f}_{Y} .\label{eq:48}
\end{equation}
Recall that $C_c(G)\subset Y^\#$ and take $\varphi\in C_c(G)$. Then
by~\eqref{eq:50} with $q_i(f)=\nor{f}$ ($Y$ is normed)
\begin{align*}
\int_G\abs{v(y)}\abs{\varphi(y)}\,dy=
  \int_G \abs{g(x^{-1})}\left(\int_G \abs{f(yx) \varphi(y)}
  dy\right) dx 
  & \leq C \int_G \abs{g(x^{-1})}\nor{r(x)f}_Y\, dx, 
\end{align*}
which is finite by assumption.  By Fubini theorem, the function 
\[ (x,y)\mapsto g(x^{-1})f(yx) \overline{\varphi(y)}\]
is in $L^1(G\times G)$ and, hence, there exists a negligible set $N_\varphi$
such that for all $y\not\in N_\varphi$ the function 
\[ x\mapsto g(x^{-1})f(yx) \overline{\varphi(y)} \]
is in $L^1(G)$. Put $E_\varphi=\set{x\in G\mid \varphi(x)\neq 0}$.  By
the change of variable $x\mapsto y^{-1}x$,
 for all $y\in E_\varphi\setminus N_\varphi $, 
the function 
\[ x\mapsto f(x) g(x^{-1}y)\]
is in $L^1(G)$. Take a countable family  $\set{\varphi_k}_{k\in\N}$ such that
$\bigcup_k E_{\varphi_k}=G$ and set $N=\bigcup_k N_{\varphi_k}$. Then
$N$ is negligible and  for all $y\notin N$  the map $x\mapsto f(x) g(x^{-1}y)$ is
integrable.  Hence, for all $\varphi\in C_c(G)$, Fubini theorem gives
\begin{align*}
  \int_G  v(y)  \overline{\varphi (y)} dy & =\scal{v}{\varphi}_Y \\
& = \int_G g(x^{-1}) \scal{r(x)f}{\varphi}_Y \,dx\\
& = \int_G g(x^{-1}) \left(\int_G f(yx) \overline{\varphi (y)}\,dy\right)\,dx \\
& = \int_G  \overline{\varphi (y)}\left(\int_G f(yx) g(x^{-1})\,dx\right) \,dy
\end{align*}
where $\scal{\cdot}{\cdot}_Y$ denotes the duality between $Y$ and
$Y^\#\subset C_c(G)$ introduced in~\eqref{eq:9}.
By the change of variable $x\mapsto y^{-1}x $ in the inner integral, we get
\begin{equation}
\int_G  v(y)  \overline{\varphi (y)} dy = \int_G   \left(\int_G  f(x) g(x^{-1}y)dx\right)
\overline{\varphi(y)} dy.\label{eq:49}
\end{equation}
This means that $j(f)$ and $g$ are
convolvable, $j(v)=j(f)\con g$ and, by~\eqref{eq:48}, the inequality 
$\nor{f \con g}_Y\leq \int_G \abs{g(x^{-1})}\nor{r(x)f}_{Y}$ holds true.
\end{proof}

\begin{cor}\label{MYclosed}
Take a weight $w$  such that
$\norop{r(x)}{Y} \leq w(x)$ for all $x\in G$. Then:
\begin{enumerate}[a)]
\item  For all $f\in Y$ and $\widecheck{g}\in \Luw$, $j(f)$ and $g$ are convolvable,
  there exists $f\con g\in Y$ such that $j(f\con g)=j(f)\con g$, the
  map
  \[ Y\ni f\mapsto f\con g\in Y \] is continuous and $\nor{ f\con
    g}_Y\leq \nor{f}_Y \nor{g}_{1,w}$. 
\item The set
  \[ \mathcal M^Y=\set{f\in Y \mid f\con g = f},\] is a closed
  $\ell$-invariant subspace of $Y$.
\end{enumerate}
\end{cor}
\begin{proof} Item a) follows from Proposition~\ref{prop4}, observing that
  ~\eqref{eq:48} is satisfied and
  \[\int_G \abs{g(x^{-1})}
  \nor{r(x)f}_{Y} \,dx\leq \nor{f}_Y \nor{\check{g}}_{1,w} .
  \]
As for b), since $Y$ is a metrizable topological  space, it is sufficient to prove that
$\mathcal M^Y$ is sequentially closed.  Take a sequence $(f_n)_n$  in
$\mathcal M^Y$ converging to $f\in Y$. Possibly passing to a 
subsequence, we can assume that there exists a negligible set $N$ such
that  for all $x\notin N$ the sequence $(f_n(x))_n$ converges to $f(x)$.
Furthermore, possibly changing $N$, we can also assume that, for all
$n$ and $x\notin N$,  $j(f_n)\con g\,(x)=f_n(x)$.

Since $f\mapsto f\con g$  and $j$ are continuous, $j(f_n)\con g$ converges to $j(f) \con g$
in $L^1_{\rm loc}(G)$. Hence, by
Lemma~\ref{luno} in the appendix, possibly passing again to a subsequence and again redefining $N$,
we can also assume that for all  $x\notin N$  $\lim_n j(f_n)\con
g\,(x)= j(f)\con g\,(x)$. Then
\[j(f) \con g\,(x)= \lim_n j(f_n)\con g\,(x)=\lim_n f_n(x)=f(x),\]
so that $j(f)\con g=j(f)$ in $\Lz$, that is $f\in\mathcal M^Y$.
Finally, given $x\in G$ and $f\in\mathcal M^Y$, by~\eqref{eq:27} ,
\[ j(\ell(x)f)=\la(x) j(f)=\la(x)(j(f)\con g)= \la(x)j(f)\con g= j(\ell(x)f)\con g ,\]
 which means that $\ell(x)f\in\mathcal M^Y$.
\end{proof}
We apply the above corollary with the choice $g=K$, which is in $\Luw$
by assumption, together with $\widecheck K=\overline K$.  
Notice that, although Assumption~\ref{H5} is not satisfied,  b) of Corollary~\ref{MYclosed} guarantees that ${\mathcal M}^Y$ is a closed subspace of $Y$.  Furthermore we
assume that 
\begin{equation}
  \label{eq:55}
  \mathcal M^Y\subset \Liw.
\end{equation}
Since by construction $V_2v\in\Luw$ for all $v\in\Ss$ and $\Liw=\Luw^\#$,
~\eqref{eq:55} implies Assumption~\ref{H6}.  Hence we can define
\begin{align*}
& \Co{Y}  = \set{T\in \Ss'_w \mid V_e T\in Y},  \\
 & \nor{T}_{\Co{Y}}  = \nor{V_eT}_Y.
\end{align*}
The inclusion~\eqref{eq:55} is satisfied by all the classical Banach spaces considered
in \cite{fegr89a}. This fact is shown in the proof of Proposition~4.3.

\begin{thm}
  The
 space $\Co{Y}$ is a $\pi$-invariant Banach space and  the restriction of $V_e$ to
 $\Co{Y}$ is an isometry from $\Co{Y}$ onto $\mathcal M^Y$. For all
 $\Phi\in \mathcal M^Y$, $\pi(\Phi)u$ exists  in $\Ss'_w$, it actually belongs to
 $\Co{Y}$ and satisfies
 \begin{subequations}
   \begin{align}
     \label{eq:56}
     &\pi(V_eT)u  =T, \qquad T\in\Co{Y} \\
     &V_e\pi(\Phi)u  =\Phi, \qquad\Phi\in\mathcal
     M^Y. \label{eq:57}
   \end{align}
 \end{subequations}
\end{thm}
\begin{proof}
Apply Proposition~\ref{cuno}.
\end{proof}

\vskip1truecm
\subsection{Completeness and  weights bounded from below}~\\

Take a continuous function $w:G\to (0,+\infty)$ satisfying only
~\eqref{eq:33}. Take a square-integrable representation $\pi$ acting on $\hh$ and fix an
admissible vector $u\in\hh$ such that
\[
 K(\cdot)=\scal{u}{\pi(\cdot)u}_\hh \in L^1_w(G) ,
\] 
and set
\begin{align*} &{\Ss_w} =\set{v\in\hh\mid \scal{v}{\pi(\cdot)u}_\hh \in L^1_w(G)}\\
&\nor{v}_{\Ss_w}= \int_G \abs{\scal{v}{\pi(x)u}_\hh}w(x) dx.
\end{align*}

\begin{thm}\label{Sbanach}
The space ${\Ss_w}$ is a Banach space if and only if 
\begin{equation}
  \label{q:1}
  \inf_{x\in G} w(x)>0. 
\end{equation}
\end{thm}
\begin{proof}
Assume that $\inf_{x\in G} w(x)\geq c>0$.  We can always suppose that
$c=1$  so that~\eqref{eq:46} holds true. Indeed, if $c<1$, we redefine $w$ as $w/c$, 
so that
\[\frac{w(xy)}{c}\leq \frac{w(x)w(y)}{c^2} c\leq \frac{w(x)}{c}\frac{w(y)}{c}\]
and $w/c$ satisfies~\eqref{eq:33}. Since $L^1_w(G)=L^1_{w/c}(G)$
with equivalent norms, clearly the fact that  ${\Ss_w}$ is a Banach space
does not depend on the choice of $c$.   Item a) of
Theorem~\ref{nonirreducible} states that ${\Ss_w}$ is  a Banach
space.

Assume that  ${\Ss_w}$
is a Banach space.  Define ${\Ss^*_w}$ as the vector space ${\Ss_w}$ with the
norm
\[
\nor{v}_*= \max\set{\nor{v}_{\Ss_w},\nor{v}_\hh}.
\]
We claim that ${\Ss^*_w}$ is complete.  Take a Cauchy sequence $(v_n)_n$
with respect to $\nor{\cdot}_*$. By construction, it is a Cauchy sequence
also with respect to both $\nor{\cdot}_{\Ss_w}$ and $\nor{\cdot}_\hh$. Since ${\Ss_w}$
and $\hh$ are complete, there exist $v'\in{\Ss_w}$ and $v\in\hh$ such that 
\[ 
\lim_{n\to+\infty} \nor{v_n-v}_\hh=0 \qquad \lim_{n\to+\infty} \nor{v_n-v'}_{\Ss_w}=0.
\]
Since the voice trasform is an isometry both from
$\hh$ into $L^2(G)$ and from ${\Ss_w}$ into $L^1_w(G)$,
the sequence $(Vv_n)_n$ converges to $Vv$ in $L^2(G)$ and to $Vv'$
in $L^1_w(G)$. Hence, possibly passing to a subsequence,  $(Vv_n)_n$
converges  almost everywhere to $Vv$ and  to $Vv'$. Since $w>0$, 
$Vv=Vv'$ almost everywhere and, hence, 
$v=v'$ by the injectivity of $V$, so that $v\in{\Ss^*_w}$. Furthermore
\begin{align*}
\lim_{n\to+\infty} \nor{v_n-v}_*&= \lim_{n\to+\infty}\max\set{\nor{v_n-v}_{\Ss_w},\nor{v_n-v}_\hh}\\ 
&= \max\set{\lim_{n\to+\infty}\nor{v_n-v}_{\Ss_w},\lim_{n\to+\infty}\nor{v_n-v}_\hh}=0.
\end{align*}
Hence ${\Ss^*_w}$ is complete and the natural inclusion $i:{\Ss^*_w}\to{\Ss_w}$ is
clearly continuous and  bijective. Since ${\Ss_w}$ is a Banach space, the open mapping theorem implies that
the inverse is also continuous, so that  there is a constant $c>0$ such
that
\[
c\nor{v}_* \leq \nor{v}_{\Ss_w} \leq \nor{v}_*.
\]
As usual, for all $x\in G$ and $v\in{\Ss_w}$
\[\nor{\pi(x)v}_{\Ss_w}= \nor{\la(x)Vv}_{L^1_w(G)}\leq w(x)\nor{v}_{\Ss_w},
\qquad \nor{\pi(x)v}_\hh=\nor{v}_\hh.\]
Then, if $v\neq 0$, for all $x\in G$
\[ c \nor{v}_\hh =c \nor{\pi(x)v}_\hh \leq
c\nor{\pi(x)v}_* \leq \nor{\pi(x)v}_{\Ss_w}\leq w(x)\nor{v}_{\Ss_w}\]
taking the infimum over $G$, we get 
\[0<c\leq \inf_{x\in G} w(x).\]
\end{proof}

\section{Appendix: some functional analysis}

\subsection{Notation}\label{Notation}

The set $\Lz$ is a metrizable complete topological vector space, and 
$C_0(G)$ is dense in $\Lz$
(Propositions~19 and ~20 Chapter IV.5.11 of \cite{BINT1}). Since $G$
is second countable,  $C_0(G)$ is separable (with respect to the
convergence on compact subsets, hence with respect to the convergence in measure)
and $\Lz$ is separable, too. Furthermore, 
if $(f_n)$ is a sequence converging to $f$  in $\Lz$,
then there exist a subsequence  $(f_{n_k})_k$ and a negligible set $N\subset G$ such
that
\begin{equation}
  \label{eq:25}
  \lim_{k\to+\infty}  f_{n_k}(x)= f(x)\qquad \text{for all  } x\in
  G\setminus N.
\end{equation}
(Corollary~of Proposition~19 Chapter IV.5.11 of
\cite{BINT1}).

If $G$ is compact $\Lloc=L^1(G)$ is a separable Banach space. Otherwise,
a   saturated fundamental system of semi-norms is given as  follows (recall that a family is saturated if the maximum of any finite set of seminorms is in the family). 
Since $G$ is second countable, take a countable increasing family $({\mathcal K}_i)_{i\in\N}$ of compact
  subsets of $G$ such that ${\mathcal K}_i\subset \ring{{\mathcal K}}_{i+1}$ and $\bigcup_i
  {\mathcal K}_i=G$. For all $i\in\N$ put
  \begin{equation}
    q_i(f)= \int_{{\mathcal K}_i} \abs{f(x)}dx.\label{eq:10}
  \end{equation}
  Then, for each compact set ${\mathcal K}$ there exists $i\in\N$ such that ${\mathcal K}\subset
  {\mathcal K}_i$ and
  \[ \int_{\mathcal K} \abs{f(x)}dx \leq q_i(f).\]
With the induced topology, $\Lloc$ is complete (see Ex.~31 Chapter~V.5 of
\cite{BINT1}), hence it is a Fr\'echet space.

\begin{lem}\label{luno}
If $(f_n)$ is a sequence in $\Lloc$ converging to $f$ in $\Lloc$, then there exists
a subsequence  $(f_{n_k})_k$ that converges to $f$
 almost everywhere.
\end{lem}
\begin{proof}
If $G$ is compact, the claim is clear. If not, take  the
increasing sequence of compact subsets $(K_i)_{i\in\N}$ defining 
the fundamental family of semi-norms~\eqref{eq:10}.  The topology of
$\Lloc$ is such that $(f_n)_n$ converges to
$f$ in $L^1(K_i)$ for all $i\in\N$. We procede by induction on $\N$. Suppose that we have found a subsequence
$(f_{n^{(i)}_k})_k$ and  a negligible subset $N_i\subset K_i$ such that
\[ \lim_{k\to+\infty}  f_{n^{(i)}_k}(x)= f(x)\qquad \text{for all } x\in K_i\setminus
N_i.\]
Clearly $(f_{n^{(i)}_k})_k$ converges to $f$ in $L^1(K_{i+1})$ and we
can further extract a subsequence $(f_{n^{(i+1)}_k})_k$ for which there exists a negligible subset $N_{i+1}\subset
K_{i+1}$ such that 
\[ \lim_{k\to+\infty}  f_{n^{(i+1)}_k}(x)= f(x)\qquad \text{for all } x\in K_{i+1}\setminus
N_{i+1}.\] 
Set $N=\bigcup_{i\in\N} N_i$ and $f_{n_k}= f_{n^{(k)}_k}$.  Given $x\notin N$, fix ${h}$ such that $x\in
K_{h}$ , so that  $x\in K_i \setminus
N_i$ for all $i\geq{h}$. Then $(f_{n^{(k)}_k}(x))_{k\geq{h}}$ is a subsequence
of $(f_{n^{({h})}_k}(x))_{k\geq{h}}$ which converges to $f(x)$.
\end{proof}

Given $f\in L^0(G)$, $\check{f}\in\Lz $ since a subset $E\subset G$ is negligible if
and only if $E^{-1}$ is negligible. Notice that 
\begin{align}
  \label{eq:26}
&   \check{f}\in L^p(G)\qquad\Longleftrightarrow \qquad
  \Delta^{-1/p} f\in L^p(G) \qquad\Longleftrightarrow \qquad
  f\in L^p(G,\Delta^{-1}\cdot\beta) \\
& \nor{\check{f}}_p= \nor{\Delta^{-1/p} f}_p.
\end{align}
Since $\Delta$ is continuous, $\Lloc$ is invariant under the map
$f\mapsto \check{f}$. 

\subsection{Representations}\label{Reps}
Let $E$ be a locally convex space with a saturated fundamental system $\set{q_i}_i$ of semi-norms
and $\tau$ a (linear) representation of $G$ on $E$. 
\begin{enumerate}[i)]
\item\label{sepcontrep} The representation $\tau$ is separately continuous if 
  \begin{enumerate}[a)]
  \item for all $x\in G$, $\tau(x)$ is continuous from $E$ to $E$;
   \item for all $v\in E$, $x\mapsto \tau(x)v$ is continuous from $G$ into $E$.
  \end{enumerate}
\item\label{contrep} The representation $\tau$ is  continuous if 
\begin{enumerate}[a)]
  \item if $(x,v)\mapsto \tau(x)v$ is continuous from $G\times E$ into $E$.
  \end{enumerate}
\item The representation $\tau$ is  equicontinuous if 
\begin{enumerate}[a)]
  \item if $(x,v)\mapsto \tau(x)v$ is continuous from $G\times E$ into $E$;
   \item $\tau(G)$ is equicontinuous, {\em i.e.} for every semi-norm
     $q_i$ there exists a semi-norm $q_j$ and
     a constant~$C$ such that $q_i(\tau(x))\leq C q_j(\tau(x))$ for all $x\in G$.
  \end{enumerate}
\end{enumerate}

If $E$ is a Fr\'echet space, then \ref{sepcontrep}) implies \ref{contrep}) 
 (Proposition~1 Chapter VIII.2.1 of
\cite{BINT1}). Furthermore, $\tau$ is continuous if and only if for
any compact set $Q$ of $G$, $\tau(Q)$ is equicontinuous and the map
$x\mapsto \tau(x)v$ is continuous for all $v\in D$, where $D$ is a
total set in $E$ (Remark 2 of Definition~1 Chapter VIII.2.1 of
\cite{BINT1}).

\subsection{Convolutions}\label{convolution}

The basic property of convolution is given by the following lemma.
\begin{lem}
If $f\con g$ exists, it  is
a $\beta$-measurable function whose  equivalence class in $\Lz$ depends only on the equivalence classes
of $f$ and $g$.
\end{lem}
\begin{proof}
Without loss of generality, we can suppose that both $f$ and $g$ are
positive.  The topological isomorphism $\psi:G\times G\to G\times G$,
$\psi(x,y)=(x,y^{-1}x)$ has the property that a set 
$E\subset G\times G$ is $\beta\otimes\beta$-negligible if and only if  
$\psi^{-1}(E)$ is $\beta\otimes\beta$-negligible. Indeed, take $E$
a Borel measurable subset of $G\times G$, then 
\begin{align*}
 \beta\otimes\beta( \psi^{-1}(E))=\int_G \beta(\psi^{-1}(E)_x) dx =
 \int_G \beta(xE_x^{-1})dx =  \int_G \beta(E_x^{-1}) dx
\end{align*}
where $E_x=\set{y\in G\mid (x,y)\in E}$ and $\psi^{-1}(E)_x=\set{y\in
  G\mid (x,y^{-1}x)\in E}=xE_x^{-1}$. Hence $\beta\otimes\beta(
\psi^{-1}(E))=0$ if and only if $\beta(E_x^{-1})=0$ for almost all
$x\in G$, which is equivalent to the fact that $\beta(E_x)=0$ for almost all
$x\in G$, {\em i.e.} $\beta\otimes\beta(E)=0$.  As a consequence, the
map $\varphi=(f\otimes g)\psi$ is $\beta\otimes \beta$-measurable, and if we
change $f\otimes g$ on a negligible set, $\varphi$ will
change on a negligible set, too. 
Since $G$ is second countable, the measure $\beta$ is moderated and Proposition~7.b) Chapter
V.8.3 of \cite{BINT1} shows that the map $x\mapsto \int_G
\varphi(x,y)\,dy$ is $\beta$-measurable,  where the integral is finite by assumption. 
Therefore,   $\int_G
\varphi(x,y)dy$ depends
only on the equivalence class of $f$ and $g$.
\end{proof}

If  $f,g\in \Lloc$, $f\con g$ exists and  $\abs{f}\con\abs{g}$
is in $\Lloc$, then we say that $f$ and $g$ are {\it convolvable}.
Since  $f,g\in \Lloc$, then $\mu=f\cdot \beta$ and $\nu=g\cdot \beta$ are (Radon)
measures on $G$.  The fact that $f$ and $g$ are convolvable is equivalent to
the fact that $\mu$ and $\nu$ admit a convolution, {\em i.e.} for all
$\varphi\in C_c(G)$ the function $(x,y)\mapsto \varphi(xy)$ is
$\mu\otimes\nu$-integrable, namely
\[ \int_{G\times G} \abs{\varphi(xy)} \abs{f(x)}\abs{g(x)} dx dy
<+\infty, \qquad\varphi\in C_c(G).\]
The two definitions agree, since:
\begin{enumerate}[i)]
\item if $\mu$ and $\nu$ admit a convolution, the map $\varphi\mapsto \int_G
\varphi(xy)d\mu(x)d\nu(y)$ defines a measure on $G$ whose  density
is precisely $f\con g$ (Proposition 10 Chapter VIII.3.2 of \cite{BINT2} and
Proposition 10 Chapter VIII.4.5 of \cite{BINT2});
\item if  $\abs{f}\con\abs{g}$ exists and is in $\Lloc$, then $\mu$ and $\nu$
  admit a convolution (Proposition 9 Chapter VIII.4.5 of \cite{BINT2}).
\end{enumerate}

We recall the following sufficient conditions.
\begin{subequations}
\begin{enumerate}[a)]
\item Corollary 20.14 of \cite{hero1}: if $f\in L^1(G)$ and $g\in
  L^p(G)$ with $p\in [1,+\infty]$, then $f$ and $g$ are convolvable, $f\con g$ belongs
  to $L^p(G)$ and
  \begin{equation}
    \nor{f\con g}_ p\leq \nor{f}_1 \nor{g}_p\label{eq:3}.
  \end{equation}
  \item Theorem~20.18 of \cite{hero1}: if $f\in L^p(G)$, $g\in L^q(G)$
  and $\widecheck{g}\in L^q(G)$ with $1<p <+\infty$, $1<q<+\infty$ satisfying
  $\frac{1}{p}+\frac{1}{q}=1+\frac{1}{r}$ with $r> 1$, then $f$ and $g$ are convolvable and $f\con
  g$ belongs to $L^r(G)$. Furthermore, if $  \nor{\check{g}}_q=\nor{g}_q$, then 
  \begin{equation}
    \nor{f\con g}_ r \leq \nor{f}_p \nor{g}_q\label{eq:5}.
  \end{equation}
   
\item Theorem 20.16 of \cite{hero1}: under the same assumptions on $f$ and $g$ as in the previous item, if
  $\frac{1}{p}+\frac{1}{q}=1$ with $1<p<+\infty$, then $f$ and $g$ are convolvable and $f*g$ 
  belongs to $ C_0(G)$ and
  \begin{equation}
    \label{eq:6}
    \nor{f\con g}_\infty \leq \nor{f}_p \nor{\widecheck{g}}_q.
  \end{equation}
\item\label{young} Theorem 20.16 of \cite{hero1}: if $f\in L^1(G)$ and $g\in
  L^\infty(G)$ (which is equivalent to $\widecheck{g}\in L^\infty(G)$) or if
  $f\in L^\infty(G)$ and $\widecheck{g}\in L^1(G)$, then $f$ and $g$ are convolvable, $f*g$
 is a bounded continuous function, and
  \begin{equation}
    \label{eq:7}
    \nor{f\con g}_\infty \leq \nor{f}_1 \nor{g}_\infty
    \quad\text{ or } \nor{f\con g}_\infty \leq \nor{f}_\infty \nor{\widecheck{g}}_1
  \end{equation}
\end{enumerate}
\end{subequations}
In general, the convolution is not associative. We recall a sufficient
condition as well as some other useful relations.
\begin{subequations}
  \begin{lem}
Given $f,g\in\Lz$, 
  \begin{equation}
    \label{eq:34}
    \widecheck{f\con g}= \check{g}\con\check{f}
  \end{equation}
and,  for all $x\in G$,
  \begin{align}
    \la(x)f\,\con g & =\la(x)(f\con g) &&
    \rho(x)f\,\con g
    =\Delta(x^{-1})(f\con \la(x^{-1})g) \label{eq:27}\\
    f\,\con \la(x)\,g & =\Delta(x^{-1})
    (\rho(x^{-1})f\con g) && f\con \rho(x)g
    =\rho(x)(f\con g), \label{eq:28}
  \end{align}
provided that one of the two sides of each equality exists.

If $f,g,h\in\Lz$ are such that either $\abs{f}\con\abs{g}$  and
$(\abs{f}\con\abs{g})\con \abs{h}$ exist  or $\abs{g}\con\abs{h}$  and
$\abs{f}\con(\abs{g}\con \abs{h})$ exist, then 
\begin{equation}
  \label{eq:35}
  f\con (g\con h)= (f\con g)\con h
\end{equation}
and all the convolutions exist. 
  \end{lem}
  \begin{proof}
To prove~\eqref{eq:34} just compute
 \[ \check{g}\con\check{f}(x)=\int_G  \check{f}(y^{-1}x)
 \check{g}(y) dy=  \int_G  f(x^{-1}y)
 g(y^{-1})dy =  \int_G  f(y)
 g(y^{-1}x^{-1})dy = f\con g(x^{-1}).\]
Next we prove~\eqref{eq:35}.  Fubini theorem gives that,
for all $x\in X$ 
\begin{align*}
  (\abs{f}\con\abs{g})\con \abs{h}(x) & = \int_{G\times
    G} \abs{f(y)} \abs{g(y^{-1}z)}\abs{h(z^{-1}x)} dydz \\  & =
  \int_{G\times G} \abs{f(y)} \abs{g(z)}\abs{h(z^{-1}y^{-1}x)} dydz= \abs{f}\con(\abs{g}\con \abs{h})(x)
\end{align*}
by the change of variable $z\mapsto yz$.  Hence the two assumptions
implies that the map  $(y,z)\mapsto f(y) g(y^{-1}z)
h(z^{-1}x) $ is in $L^1(G\times G)$. Since  $\abs{f\con g}\leq
\abs{f}\con\abs{g}$, all the convolutions in~\eqref{eq:35}
exist and Fubini theorem implies the claimed equality.
The remaining assertions are standard.
\end{proof}
\end{subequations}

\subsection{Scalar integration}\label{scalar_integration}

Let $\E$ be a locally convex topological vector space, and let $X$ be
a Hausdorff locally compact second countable topological space
with a positive measure $dx$, which is finite on all compact subsets.
A function $ \Psi : X \to \E $ is called scalarly integrable
if the scalar function $ \scal{T}{\Psi(\cdot)}_E $ is integrable for every $ T \in \E' $.
If $\Psi$ is scalarly integrable, the map
$$ T \mapsto \int_X \scal{T}{\Psi(x)}_E dx $$
defines a linear functional on $\E'$, possibly not continuous;
that is, there exists an element in the algebraic dual $\E'^*$,
called the scalar integral of $\Psi$ and denoted
$$ \int_X \Psi(x) dx \in \E'^* , $$
such that
$$ \langle T , \int_X \Psi(x) dx \rangle_E = \int_X \scal{\Psi(x)}{T} _Edx . $$

Usually one is interested to understand under which conditions the scalar integral lies in $E$.
In our paper we often look ar the case in which the argument takes its values in a dual space (or in a space which embeds into a dual space), namely
$$ \Psi : X \longrightarrow E'_s,$$
where  $E'_s$ is the space $E'$ endowed twith the weak*-topology,
namely the topology of simple convergence, so that $(E_s')'=E$.

A locally convex space $\E$ is said to have the property (GDF)\footnote
{The acronym GDF stands for ``graphe d\'enombrablement ferm\'e\,'',
namely ``countably closed graph''.}
if every linear map from $\E$ to a Banach space which has sequentially closed graph is actually continuous
(that is, the closed graph theorem holds true for Banach space-valued linear maps defined on $\E$).
All the Fr\'echet spaces enjoy the property (GDF) (\cite{BTVS}, Chapter I.3.3, Corollary 5).
Also, the dual space of any reflexive Fr\'echet space satisfies the property with 
respect to the strong topology, namely the topology of the convergence on bounded sets
(\cite{BINT1}, Chapter 6, Appendix, n$^\circ$ 2, Proposition 3).

The key theorem for the convergence of scalar integrals with values in a dual vector space is the following.
\begin{thm}[Gelfand--Dunford, \cite{BINT1}, Theorem 1, Chapter VI.1.4] \label{Gelfand-Dunford}
 Let $\E$ be a Hausdorff locally convex topological vector space with the property (GDF).
 Then, for any scalarly integrable function $ \Psi : X \to E'_s $, we have
 $$ \int_X \Psi(x) dx \in E' . $$
\end{thm}

\subsection{Intersections of $L^p$ spaces}

In this final section we recall, for the reader's convenience, the main results
obtained in \cite{damuwe70}, specialized to our setting.
Set $I=(1,+\infty)$ and define
\[ \T =\bigcap_{p \in I} L^p(\mu) \]
with the initial topology, which makes each inclusion $\T\hookrightarrow
L^p(G) $ continuous, and
\[ \U = \operatorname{span}\bigcup_{q\in I} L^q(G)\]
with the final topology, which makes each inclusion
$L^q(G)\hookrightarrow\U$ continuous.

\begin{thm}[\cite{damuwe70}]\label{classic}
The space $\T$ is a reflexive Fr\'echet space and $\U$ is a complete
reflexive
locally convex topological vector space.  For each $g\in \U$, the
linear map  
\[ f\mapsto \int_G g(x){f(x)}\,dx =g(f)\]
is continuous and $g\mapsto g(\cdot)$ identifies, as
topological vector spaces,   the  dual of  $\T$  with $\U$.
For each $f\in \T$, the linear map 
\[ g\mapsto \int_G f(x){g(x)}\,dx=f(g) \]
is continuous and  $f\mapsto f(\cdot)$ identifies, as topological vector spaces,   the  dual of  $\U$  with~$\T$.
\end{thm}
\begin{proof}
Here we refer to \cite{damuwe70}. 
Observe that the Haar measure $\beta$ is $\sigma$-finite since $G$ is
  locally compact and second countable and $\beta$ is  finite on
  compact subsets. Furthermore, denoted by $I'=\set{\frac{p}{p-1}\mid
  p\in I}$, clearly $I'=I$.  Proposition~2.1 and the following remark
show that the map $f\mapsto \scal{f}{\cdot}_\U$ is a topological 
isomorphism from $\T$  onto the strong dual of $\U$. 

Theorem~2.1 and Corollary 3.2  show that the map $g\mapsto \scal{g}{\cdot}_\T$ is a topological 
isomorphism from $\U$  onto the strong dual of $\T$. 

Hence  we can identify, as topological vector spaces, the dual of $\T$ with
$\U$ and the dual of $\U$ with $\T$. So that both $\T$ and $\U$ are
reflexive locally convex vector spaces. Theorem 3.1 proves that $\T$
is a Frech\'et space and Corollary 3.3 shows that $\U$ is complete.
\end{proof}
Note that, since $\T$ and $\U$ are reflexive spaces, they are  barrelled (Theorem
2 IV.2.3 of \cite{BTVS}).

\section*{Acknowledgements} 
\noindent 
S. Dahlke acknowledges support from Deutsche Forschungsgemeinschaft
(DFG), Grant DA 360/19--1. 
F. De Mari, E.~De Vito and S. Vigogna were partially supported by
Progetto PRIN 2010-2011 "Variet\`a reali e complesse: geometria,
topologia e analisi armonica ". 
D. Labate acknowledges support from NSF grants DMS 1008900 and
1005799. 
S. Dahlke, G. Steidl and G. Teschke  were partially supported by DAAD
Project 57056121,  ''Hochschuldialog mit Südeuropa 2013''.


\end{document}